\pdfoutput=1
\documentclass[11pt,a4paper]{article}
\usepackage{geometry} 
\usepackage{amsmath,amsfonts,amsthm,amssymb,yhmath, mathtools} 
\usepackage{enumitem} 
\usepackage{graphicx, float, subfig, overpic} 
\usepackage{stmaryrd} 
\usepackage{faktor} 
\usepackage{mathrsfs} 
\usepackage{titlesec} 
\usepackage[none]{hyphenat} 
\usepackage[font=small,labelfont=bf,tableposition=top]{caption}
\usepackage{authblk} 

\usepackage[dvipsnames]{xcolor}  
\usepackage{hyperref} 
\hypersetup{
	colorlinks=true,
	linkcolor=blue,
	citecolor=green,
}

\usepackage[style=trad-alpha, doi=false, isbn=false, url=false, eprint=false]{biblatex}
\graphicspath{{Figures/}}
\numberwithin{equation}{section}
\newcommand{\st}{\h:\h}

\newcommand{\h}{\hspace{1mm}}
\newcommand{\hh}{\hspace{5mm}}

\theoremstyle{plain}
\newtheorem{lemma}{Lemma}[section]
\newtheorem{corollary}[lemma]{Corollary}
\newtheorem{proposition}[lemma]{Proposition}
\newtheorem{theorem}[lemma]{Theorem}
\newtheorem{remark}[lemma]{Remark}

\def\be{\begin{eqnarray}}
\def\ee{\end{eqnarray}}
\def\beal{\begin{aligned}}
\def\enal{\end{aligned}}

\newcommand{\norm}[1]{\left\lVert#1\right\rVert}

\newcommand{\vabs}[1]{\left| #1 \right|}
\newcommand{\vabsSmall}[1]{| #1 |}
\newcommand{\paren}[1]{\left(#1\right)}
\newcommand{\claus}[1]{\left\{#1\right\}}
\newcommand{\boxClaus}[1]{\left[#1\right]}

\newcommand{\ol}[1]{\overline{#1}}
\newcommand{\conj}[1]{\overline{#1}}

\newcommand{\wt}{\widetilde}
\newcommand{\wh}{\widehat}
\renewcommand{\Re}{\mathrm{Re\, }}
\renewcommand{\Im}{\mathrm{Im\,}}
\renewcommand{\arg}{\mathrm{arg\,}}

\newcommand{\reals}{\mathbb{R}}
\newcommand{\naturals}{\mathbb{N}}
\newcommand{\complexs}{\mathbb{C}}
\newcommand{\torus}{\mathbb{T}}

\newcommand{\al}{\alpha}

\newcommand{\de}{\delta}
\newcommand{\D}{\Delta}

\newcommand{\la}{\lambda}
\newcommand{\La}{\Lambda}
\newcommand{\g}{\gamma}
\newcommand{\G}{\Gamma}
\newcommand{\s}{\sigma}

\newcommand{\tht}{\theta}
\newcommand{\Tht}{\Theta}

\newcommand{\Ups}{\Upsilon}
\newcommand{\bb}{B}

\newcommand{\lad}{\dot{\lambda}}
\newcommand{\Lad}{\dot{\Lambda}} 

\newcommand{\AAA}{\mathcal{A}}
\newcommand{\BB}{\mathcal{B}}

\newcommand{\EE}{\mathcal{E}}
\newcommand{\FF}{\mathcal{F}}
\newcommand{\GG}{\mathcal{G}}
\newcommand{\GGh}{\widehat{\mathcal{G}}}
\newcommand{\HH}{\mathcal{H}}
\newcommand{\II}{\mathcal{I}}
\newcommand{\JJ}{\mathcal{J}}
\newcommand{\KK}{\mathcal{K}}
\newcommand{\LL}{\mathcal{L}}
\newcommand{\MM}{\mathcal{M}}
\newcommand{\NNN}{\mathcal{N}}
\newcommand{\OO}{\mathcal{O}}
\newcommand{\PP}{\mathcal{P}}

\newcommand{\RRR}{\mathcal{R}}
\newcommand{\SSS}{\mathcal{S}}
\newcommand{\TTT}{\mathcal{T}}

\newcommand{\WW}{\mathcal{W}}

\definecolor{myGreen}{RGB}{0, 200, 0}
\definecolor{myOrange}{RGB}{255, 100, 0}
\definecolor{myYellow}{RGB}{255, 200, 0}
\definecolor{myBlue}{RGB}{0, 200, 255}
\definecolor{myPurple}{RGB}{200, 0, 200}

\newcommand{\Id}{\mathrm{Id}}

\usepackage{stackengine,scalerel}
\stackMath
\newcommand\that[1]{%
\setbox0=\hbox{$#1$}%
\ensurestackMath{%
	\stackon[2pt]{\copy0}{\,\rotatebox{90}{\stretchto{\triangleright}{\dimexpr\wd0-3pt}}}%
}%
}

\newcommand{\unstable}{{\mathrm{u}}}
\newcommand{\stable}{{\mathrm{s}}}

\newcommand{\phiA}{\varphi}
\newcommand{\phiB}{\tilde{\varphi}}
\newcommand{\wtA}{\widetilde{A}}
\newcommand{\wtB}{\widetilde{B}}

\newcommand{\HInicial}{h}

\newcommand{\Poi}{\mathrm{Poi}}
\newcommand{\sca}{\mathrm{sc}}
\newcommand{\equi}{\mathrm{eq}}
\newcommand{\osc}{\mathrm{osc}}
\newcommand{\pend}{\mathrm{p}}
\newcommand{\flow}{\mathrm{fl}}
\newcommand{\lah}{\hat{\lambda}}
\newcommand{\Lh}{\hat{L}}
\newcommand{\Lah}{\hat{\Lambda}}
\newcommand{\xh}{\hat{x}}

\newcommand{\yh}{\hat{y}}

\newcommand{\Ltres}{\mathfrak{L}}
\newcommand{\LtresLa}{\mathfrak{L}_{\Lambda}}
\newcommand{\Ltresx}{\mathfrak{L}_x}
\newcommand{\Ltresy}{\mathfrak{L}_y}

\newcommand{\vap}{\nu}

\newcommand{\inn}{\mathrm{in}}
\newcommand{\Inn}{^{\mathrm{in}}}
\newcommand{\Inner}{{\mathrm{in}}}
\newcommand{\rhoInn}{\kappa}
\newcommand{\kappaInner}{\kappa_5}
\newcommand{\CInn}{\Tht}

\newcommand{\DuInn}{\mathcal{D}^{\mathrm{u,in}}_{\rhoInn}}
\newcommand{\DsInn}{\mathcal{D}^{\mathrm{s,in}}_{\rhoInn}}

\newcommand{\DdInn}{\mathcal{D}^{\diamond,\mathrm{in}}_{\rhoInn}}

\newcommand{\EInn}{\mathcal{E}^{\mathrm{in}}_{\rhoInn}}

\newcommand{\ZuInn}{Z^{\mathrm{u}}_0}
\newcommand{\ZuInnN}{\widehat{Z}^{\mathrm{u}}_0}
\newcommand{\ZsInn}{Z^{\mathrm{s}}_0}
\newcommand{\ZsInnN}{\widehat{Z}^{\mathrm{s}}_0}
\newcommand{\ZusInn}{Z^{\mathrm{u,s}}_0}
\newcommand{\ZdInn}{Z^{\diamond}_0}
\newcommand{\ZdInnN}{\widehat{Z}^{\diamond}_0}

\newcommand{\WuInn}{W^{\mathrm{u}}_0}

\newcommand{\WdInn}{W^{\diamond}_0}
\newcommand{\WdInnN}{\widehat{W}^{\diamond}_0}

\newcommand{\XdInn}{X^{\diamond}_0}
\newcommand{\XdInnN}{\widehat{X}^{\diamond}_0}

\newcommand{\YusInn}{Y^{\mathrm{u,s}}_0}
\newcommand{\YdInn}{Y^{\diamond}_0}
\newcommand{\YdInnN}{\widehat{Y}^{\diamond}_0}
\newcommand{\DZInn}{\Delta Z_0}

\newcommand{\DYInn}{\Delta Y_0}
\newcommand{\DZInnN}{\Delta \widehat{Z}_0}

\newcommand{\out}{\mathrm{sep}}
\newcommand{\zOut}{z}

\newcommand{\zuOut}{z^{\mathrm{u}}}

\newcommand{\zsOut}{z^{\mathrm{s}}}
\newcommand{\zusOut}{z^{\mathrm{u,s}}}
\newcommand{\zdOut}{z^{\diamond}}
\newcommand{\wOut}{w}

\newcommand{\wuOut}{w^{\mathrm{u}}}

\newcommand{\wsOut}{w^{\mathrm{s}}}
\newcommand{\wusOut}{w^{\mathrm{u,s}}}
\newcommand{\wdOut}{w^{\diamond}}
\newcommand{\xOut}{x}

\newcommand{\xuOut}{x^{\mathrm{u}}}

\newcommand{\xsOut}{x^{\mathrm{s}}}

\newcommand{\xdOut}{x^{\diamond}}
\newcommand{\yOut}{y}

\newcommand{\yuOut}{y^{\mathrm{u}}}

\newcommand{\ysOut}{y^{\mathrm{s}}}

\newcommand{\ydOut}{y^{\diamond}}
\newcommand{\dzOut}{\Delta z}
\newcommand{\dwOut}{\Delta w}
\newcommand{\dxOut}{\Delta x}
\newcommand{\dyOut}{\Delta y}

\newcommand{\DuInfty}{{D}^{\mathrm{u},\infty}_{\rho_1}}
\newcommand{\DsInfty}{{D}^{\mathrm{s},\infty}_{\rho_1}}
\newcommand{\DdInfty}{{D}^{\diamond,\infty}_{\rho_1}}
\newcommand{\DuInftyRho}{{D}^{\mathrm{u},\infty}_{\rho_1}}
\newcommand{\DsInftyRho}{{D}^{\mathrm{s},\infty}_{\rho_1}}

\newcommand{\ballInftyDef}{B(\varrho)}
\newcommand{\ballInfty}{B(\varrho \de^3)}
\newcommand{\rhoInfty}{\rho_1}
\newcommand{\XcalInfty}{\mathcal{X}^{\infty}}
\newcommand{\XcalInftyTotal}{\mathcal{X}^{\infty}_{\times}}
\newcommand{\normInfty}[1]{\left\lVert#1\right\rVert^{\infty}}
\newcommand{\normInftySmall}[1]{\lVert#1\rVert^{\infty}}
\newcommand{\normInftyTotal}[1]{\left\lVert#1\right\rVert^{\infty}_{\times}}

\newcommand{\ballOuterDef}{B(\varrho)}
\newcommand{\ballOuter}{B(\varrho \de^3)}
\newcommand{\DuOut}{{D}_{\kappa,d_1,\rho_2}^{\mathrm{u, out}}}
\newcommand{\DuOutKappa}{{D}_{\kappa_1,d_1,\rho_2}^{\mathrm{u, out}}}
\newcommand{\DsOut}{{D}_{\kappa,d_1,\rho_2}^{\mathrm{s, out}}}
\newcommand{\DsOutKappa}{{D}_{\kappa_1,d_1,\rho_2}^{\mathrm{s, out}}}

\newcommand{\DdOut}{{D}_{\kappa,d_1,\rho_2}^{\diamond, \mathrm{out}}}

\newcommand{\rhoOuter}{\rho_2}
\newcommand{\dOuter}{d_1}
\newcommand{\kappaOuter}{\kappa_1}
\newcommand{\XcalOut}{\mathcal{X}^{\mathrm{out}}}
\newcommand{\XcalOutTotal}{\mathcal{X}^{\mathrm{out}}_{\times}}
\newcommand{\normOut}[1]{\left\lVert#1\right\rVert^{\mathrm{out}}}
\newcommand{\normOutSmall}[1]{\lVert#1\rVert^{\mathrm{out}}}
\newcommand{\normOutTotal}[1]{\left\lVert#1\right\rVert^{\mathrm{out}}_{\times}}
\newcommand{\normOutTotalSmall}[1]{\lVert#1\rVert^{\mathrm{out}}_{\times}}

\newcommand{\normSup}[1]{\lVert#1\rVert_{\mathrm{sup}}}
\newcommand{\normSupBig}[1]{\left\lVert#1\right\rVert_{\mathrm{sup}}}
\newcommand{\DBoomerang}{{D}_{\kappa, d}}
\newcommand{\DBoomerangKappa}{{D}_{\kappa_0, d}}
\newcommand{\dBoomerang}{d}
\newcommand{\kappaBoomerang}{\kappa_0}
\newcommand{\DOuterTilde}{\widetilde{D}_{\kappa_2,d_2,d_3}^{\mathrm{u, out}}}

\newcommand{\dOuterA}{d_2}
\newcommand{\dOuterB}{d_3}
\newcommand{\kappaOuterTilde}{\kappa_2}

\newcommand{\rhoOuterTilde}{\rho_3}
\newcommand{\YcalOuter}{\mathcal{Y}^{\mathrm{out}}}

\newcommand{\DFlow}{{D}_{\kappa_3, d_4}^{\mathrm{fl}}}
\newcommand{\dFlow}{d_4}
\newcommand{\kappaFlow}{\kappa_3}
\newcommand{\DBoomerangTilde}{\widetilde{D}_{\kappa, d}}
\newcommand{\DBoomerangTildeProp}{\widetilde{D}_{\kappa_4, d_5}}
\newcommand{\dBoomerangTilde}{d_5}
\newcommand{\kappaBoomerangTilde}{\kappa_4}
\newcommand{\YcalBoom}{\widetilde{\mathcal{Y}}}

\newcommand{\uOut}{\mathcal{U}}
\newcommand{\vOut}{\mathcal{V}}
\newcommand{\Gh}{\widehat{\Gamma}}
\newcommand{\Gu}{\Gamma^{\mathrm{u}}}
\newcommand{\normFlow}[1]{\lVert#1\rVert^{\mathrm{fl}}}

\newcommand{\normFlowTotal}[1]{\lVert#1\rVert^{\mathrm{fl}}_{\times}}

\newcommand{\XcalFlow}{\mathcal{X}^{\mathrm{fl}}}
\newcommand{\XcalFlowTotal}{\mathcal{X}^{\mathrm{fl}}_{\times}}

\newcommand{\mch}{\mathrm{mch}}
\newcommand{\kappaMch}{\kappa_6}
\newcommand{\DuMchOut}{{D}_{\kappa}^{\mathrm{mch,u}}}
\newcommand{\DsMchOut}{{D}_{\kappa}^{\mathrm{mch,s}}}
\newcommand{\DdMchOut}{{D}_{\kappa}^{\mathrm{mch},\diamond}}

\newcommand{\DuMchInn}{\mathcal{D}_{\kappa}^{\mathrm{mch,u}}}
\newcommand{\DsMchInn}{\mathcal{D}_{\kappa}^{\mathrm{mch,s}}}
\newcommand{\DdMchInn}{\mathcal{D}_{\kappa}^{\mathrm{mch},\diamond}}
\newcommand{\DusMchInn}{\mathcal{D}_{\kappa}^{\mathrm{mch,u,s}}}

\newcommand{\ZuMchO}{Z^{\mathrm{u}}}

\newcommand{\ZusMchO}{Z^{\mathrm{u,s}}}
\newcommand{\ZdMchO}{Z^{\diamond}}

\newcommand{\WuMchO}{W^{\mathrm{u}}}

\newcommand{\WdMchO}{W^{\diamond}}

\newcommand{\XuMchO}{X^{\mathrm{u}}}

\newcommand{\XdMchO}{X^{\diamond}}

\newcommand{\YuMchO}{Y^{\mathrm{u}}}
\newcommand{\YsMchO}{Y^{\mathrm{s}}}
\newcommand{\YusMchO}{Y^{\mathrm{u,s}}}
\newcommand{\YdMchO}{Y^{\diamond}}
\newcommand{\ZMchU}{Z_1}
\newcommand{\ZuMchU}{Z_1^{\mathrm{u}}}
\newcommand{\ZsMchU}{Z_1^{\mathrm{s}}}
\newcommand{\ZusMchU}{Z_1^{\mathrm{u,s}}}
\newcommand{\ZdMchU}{Z_1^{\diamond}}
\newcommand{\WMchU}{W_1}
\newcommand{\WuMchU}{W_1^{\mathrm{u}}}

\newcommand{\WdMchU}{W_1^{\diamond}}
\newcommand{\XMchU}{X_1}
\newcommand{\XuMchU}{X_1^{\mathrm{u}}}

\newcommand{\XdMchU}{X_1^{\diamond}}
\newcommand{\YMchU}{Y_1}
\newcommand{\YuMchU}{Y_1^{\mathrm{u}}}
\newcommand{\YsMchU}{Y_1^{\mathrm{s}}}
\newcommand{\YusMchU}{Y_1^{\mathrm{u,s}}}
\newcommand{\YdMchU}{Y_1^{\diamond}}
\newcommand{\XcalMch}{\mathcal{X}^{\mathrm{mch}}}

\newcommand{\XcalMchTotal}{\mathcal{X}^{\mathrm{mch}}_{\times}}
\newcommand{\XcalMchUTotal}{\mathcal{X}^{\mathrm{mch,u}}_{\times}}

\newcommand{\normMch}[1]{\left\lVert#1\right\rVert^{\mathrm{mch}}}
\newcommand{\normMchSmall}[1]{\lVert#1\rVert^{\mathrm{mch}}}

\newcommand{\normMchTotal}[1]{\left\lVert#1\right\rVert^{\mathrm{mch}}_{\times}}
\newcommand{\normMchTotalSmall}[1]{\lVert#1\rVert^{\mathrm{mch}}_{\times}}

\newcommand{\spl}{\mathrm{spl}}
\newcommand{\dzHat}{\Delta \Phi}
\newcommand{\dzHatO}{\Delta \Phi_0}
\newcommand{\dzHatU}{\Delta \Phi_1}

\newcommand{\dxOutO}{\Delta x_0}
\newcommand{\dyOutO}{\Delta y_0}

\newcommand{\DYInnC}{\Delta Y}

\newcommand{\XSpl}{\mathcal{X}^{\mathrm{spl}}}
\newcommand{\XSplTotal}{\mathcal{X}^{\mathrm{spl}}_{\times}}
\newcommand{\normSpl}[1]{\left\lVert#1\right\rVert^{\mathrm{spl}}}
\newcommand{\normSplTotal}[1]{\left\lVert#1\right\rVert^{\mathrm{spl}}_{\times}}

\newcommand{\cttTheorem}{b_0}
\newcommand{\cttOuterA}{b_1}
\newcommand{\cttOuterB}{b_2}
\newcommand{\cttOuterC}{b_3}
\newcommand{\cttOuterD}{b_4}
\newcommand{\cttOuterF}{b_5}
\newcommand{\cttOuterG}{b_6}
\newcommand{\cttOuterH}{b_7}
\newcommand{\cttInnDerA}{c_1}
\newcommand{\cttInnDerAA}{\gamma_1}
\newcommand{\cttInnDerB}{c_2}
\newcommand{\cttInnDerBB}{\gamma_2}
\newcommand{\cttInnDerC}{b_8}
\newcommand{\cttInnExist}{b_{9}}
\newcommand{\cttInnDiff}{b_{10}}
\newcommand{\cttMch}{b_{11}}
\newcommand{\cttDiffA}{b_{12}}
\newcommand{\cttDiffB}{b_{13}}
\newcommand{\cttOutInftyA}{b_{14}}
\newcommand{\cttOutOutA}{b_{15}}
\newcommand{\cttOutOutB}{b_{16}}
\newcommand{\cttOutOutC}{b_{17}}
\newcommand{\cttMchA}{b_{18}}

\newcommand{\betaBow}{\beta_0}
\newcommand{\betaOutA}{\beta_0}
\newcommand{\betaOutB}{\beta_1}
\newcommand{\betaInner}{\beta_0}
\newcommand{\betaMchA}{\beta_2}
\newcommand{\betaMchB}{\beta_3}

\title{Breakdown of homoclinic orbits to $L_3$ in the RPC$3$BP (II). An asymptotic formula} 
\author[1,2]{Inmaculada Baldom\'a}
\author[1]{Mar Giralt\thanks{Corresponding author.\\
		\emph{E-mail adresses:} \href{mailto:immaculada.baldoma@upc.edu}{immaculada.baldoma@upc.edu} (I. Baldom\'a), 
		\href{mailto:mar.giralt@upc.edu}{mar.giralt@upc.edu} (M. Giralt),
		\href{mailto:marcel.guardia@upc.edu}{marcel.guardia@upc.edu} (M. Guardia).}}
\author[1,2]{Marcel Guardia}
\affil[1]{Departament de Matem\`atiques \& IMTECH, Universitat Polit\`ecnica de Catalunya, Diagonal 647, 08028 Barcelona, Spain}
\affil[2]{Centre de Recerca Matem\`atiques, Campus de Bellaterra, Edifici C, 08193 Barcelona, Spain}
\date{July 21, 2021}

\begin{document}

\maketitle 

\begin{abstract}
 The Restricted  $3$-Body Problem models the motion of a body of negligible mass under the gravitational influence of two massive bodies called the primaries. If one assumes that the primaries perform circular motions and that all three bodies are coplanar, one has the Restricted Planar Circular  $3$-Body Problem (RPC$3$BP).
%
%
In rotating coordinates, it can be modeled by  a two degrees of freedom Hamiltonian, which has  five critical points called the Lagrange points $L_1, \ldots, L_5$. 

The Lagrange point $L_3$ is a  saddle-center critical point  which is collinear with the primaries and beyond the largest of the two.
In this paper, we obtain an asymptotic formula for the distance between the stable and unstable manifolds of $L_3$ for small values of the mass ratio $0<\mu\ll1$. In particular we show that $L_3$ cannot have (one round) homoclinic orbits.

If the ratio between the masses of the primaries $\mu$ is small, the hyperbolic eigenvalues of $L_3$ are weaker, by a factor of order $\sqrt\mu$, than the elliptic ones. 
%
This rapidly rotating dynamics makes the distance between manifolds  exponentially small with respect to $\sqrt\mu$.
Thus, classical perturbative methods (i.e the Melnikov-Poincar\'e method) can not be applied.

The obtention of this asymptotic formula relies on the results obtained in the prequel paper \cite{articleInner} on the complex singularities of the homoclinic of a certain averaged equation and on the associated inner equation.


In this second paper, we relate the solutions of the inner equation to the analytic continuation of the parameterizations of the invariant manifolds of $L_3$ via complex matching techniques. We complete the proof of the asymptotic formula for their distance showing that its dominant term is the one given by the analysis of the inner equation.
%
%
\end{abstract}

\tableofcontents


\section{Introduction} 
\label{section:introduction}
%

The Restricted Circular $3$-Body Problem 
models the motion of a body of negligible mass under the gravitational influence of two massive bodies, called the primaries, which perform a circular motion.
If one also assumes that the massless body moves on the same plane as the primaries one has the Restricted Planar Circular $3$-Body Problem (RPC$3$BP).

Let us name the two primaries $S$ (star) and $P$ (planet) and
normalize their masses so that $m_S=1-\mu$ and $m_P=\mu$, with $\mu \in \left( 0, \frac{1}{2} \right]$. 
Choosing a suitable rotating coordinate system, the positions of the primaries can be fixed at $q_S=(\mu,0)$ and  $q_P=(\mu-1, 0)$.
Then, the position and momenta of the third body, $(q,p) \in \reals^2 \times \reals^2$, are governed by the Hamiltonian system associated to the Hamiltonian
\begin{equation}\label{def:hamiltonianInitialNotSplit} 
	\begin{split}
		\HInicial(q,p;\mu) &= \frac{||p||^2}{2} 
		- q^t \left( \begin{matrix} 0 & 1 \\ -1 & 0 \end{matrix} \right) p 
		-\frac{(1-\mu)}{||q-(\mu,0)||} 
		- \frac{\mu}{||q-(\mu-1,0)||}.
	\end{split}
\end{equation}
%
Note that this Hamiltonian is autonomous. The conservation of $h$ corresponds to the preservation of the classical Jacobi constant.

For $\mu>0$, it is a well known fact that \eqref{def:hamiltonianInitialNotSplit}  has five critical points, usually called Lagrange points
(see Figure~\ref{fig:L3Outer}). 
On an inertial (non-rotating) system of coordinates, the Lagrange points correspond to periodic dynamics with the same period as the two primaries, i.e on a 1:1 mean motion resonance.
The three collinear Lagrange points, $L_1$, $L_2$ and $L_3$, are of center-saddle type whereas, for small $\mu$, the triangular ones, $L_4$ and $L_5$, are of center-center type  (see, for instance, \cite{Szebehely}).

Due to its interest in astrodynamics, a lot of attention has been paid to the study of the invariant manifolds associated to  the points $L_1$ and $L_2$ (see \cite{KLMR00, GLMS01v1, CGMM04}).
The dynamics around the points $L_4$ and $L_5$ has also been heavily studied since, due to its stability, it is common to find objects orbiting around these points (for instance the Trojan and Greek Asteroids associated to the pair Sun-Jupiter, see \cite{GDFGS89, CelGio90, RobGab06}).
Since the point $L_3$ is located ``at the other side'' of the massive primary, it has received somewhat less attention. However, the associated invariant manifolds (more precisely its center-stable and center-unstable invariant manifolds) play an important role in the dynamics of the RPC3BP  since they act as boundaries of \emph{effective stability} of the stability domains around $L_4$ and $L_5$ 
(see \cite{GJMS01v4, SSST13}). 
The invariant manifolds of $L_3$ play also a fundamental role in creating transfer orbits from the small primary to $L_3$ in the RPC$3$BP (see \cite{HTL07, TFRPGM10})  or between primaries in the Bicircular 4-Body Problem (see \cite{JorNic20, JorNic21}).


Moreover, being far from collision, the dynamics close to the Lagrange point $L_3$ and its invariant manifolds for small $\mu$ are rather similar to that of other mean motion resonances which play an important role in creating instabilities in the Solar system, see \cite{FGKR16}. On the contrary, since the points $L_1$ and $L_2$ are close to collision for small $\mu$, the analysis of the associated dynamics is quite different.

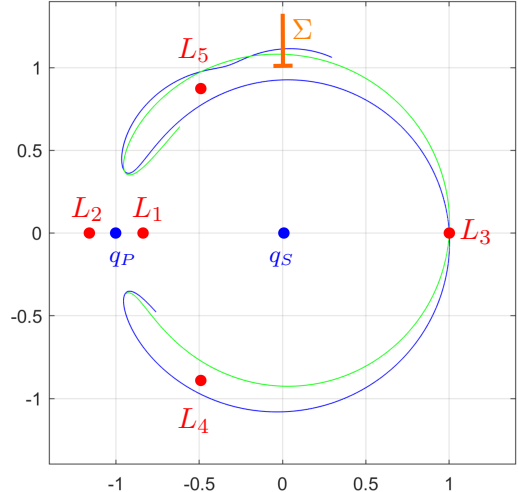
\begin{figure}
\centering
\begin{overpic}[scale=0.3]{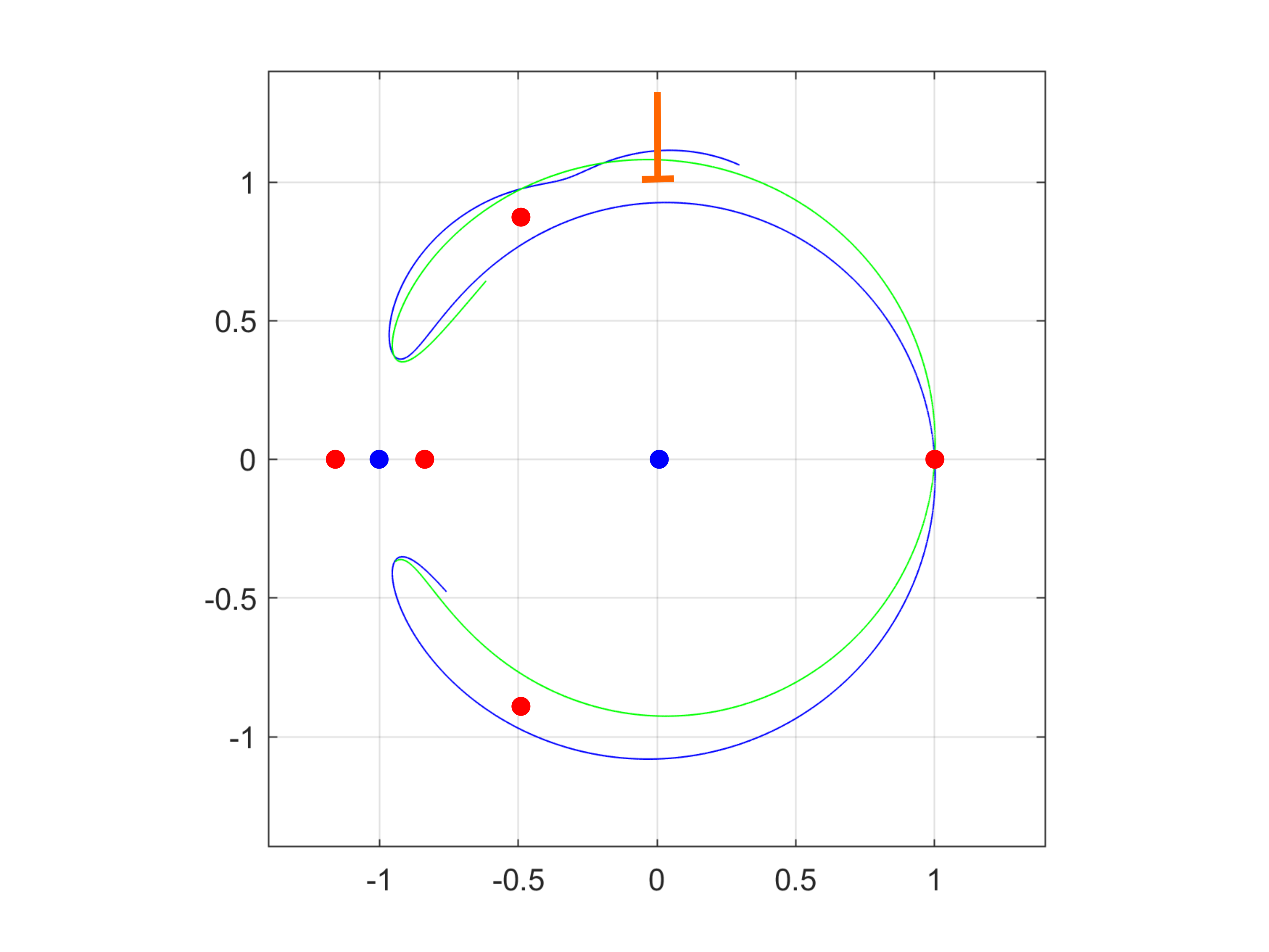}
	\put(50,35){\footnotesize{\color{blue} $q_S$ }}
	\put(29,35){\footnotesize{\color{blue} $q_P$ }}
	\put(32,41){{\color{red} $L_1$ }}
	\put(24,41){{\color{red} $L_2$ }}
	\put(75,38){{\color{red} $L_3$ }}
	\put(38,62){{\color{red} $L_5$ }}	
	\put(38,13){{\color{red} $L_4$ }}
	\put(53,64.5){{\color{orange} $\Sigma$ }}
\end{overpic}
	\caption{Projection onto the $q$-plane of the Lagrange points (red) and the unstable (blue) and stable (green) manifolds of $L_3$, for $\mu=0.0028$. }
	\label{fig:L3Outer}
\end{figure}

Over the past years, one of the main focus of study of the  dynamics ``close''  to  $L_3$ and its invariant manifolds has been  the so called ``horseshoe-shaped orbits'', first considered in~\cite{Brown1911},
which are quasi-periodic orbits that encompass the critical points $L_4$, $L_3$ and $L_5$.
The interest on these types of orbits arise when modeling the motion of co-orbital satellites, the most famous being Saturn's satellites Janus and Epimetheus, and near Earth asteroids.
Recently, in~\cite{NPR20}, the authors have proved the existence of $2$-dimensional elliptic invariant tori on which the trajectories mimic the motions followed by Janus and Epimetheus
(see also \cite{DerMur81a, DerMur81b, LlibreOlle01, CorsHall03, BarrabesMikkola05, BarrabesOlle2006, BFPC13, CPY19}).
%
%


Rather than looking at stable motions ``close to'' $L_3$ as \cite{NPR20}, the goal of this paper (and its prequel \cite{articleInner}) is rather different: its objective is to prove the breakdown of homoclinic connections to $L_3$. Indeed, since  $L_3$ is a center-saddle critical point, it possesses $1$-dimensional unstable and stable manifolds, which we denote by $W^{\unstable}(\mu)$ and $W^{\stable}(\mu)$, respectively, and a $2$-dimensional center manifold.
Theorem \ref{theorem:mainTheorem} below gives an asymptotic formula for the distance between the stable and unstable invariant manifolds (at a suitable transverse section) for mass ratio $\mu>0$ small enough.

\subsection{The distance between the invariant manifolds of \texorpdfstring{$L_3$}{L3}}

The one dimensional unstable and stable invariant manifolds of $L_3$ have two branches each (see Figure~\ref{fig:L3Outer}).
One pair circumvents $L_5$, which we denote by $W^{\unstable,+}(\mu)$ and $W^{\stable,+}(\mu)$, and the other,  $W^{\unstable,-}(\mu)$ and $W^{\stable,-}(\mu)$, circumvents $L_4$.
Since the Hamiltonian system associated to the Hamiltonian $\HInicial$ is reversible with respect to the involution
\begin{equation*}
	\Phi(q,p;t)=(q_1,-q_2,-p_1,p_2;-t),
\end{equation*}
the $+$ branches of the invariant manifolds are symmetric with respect to the $-$ branches. Thus, we restrict our analysis  to the positive branches.


To measure the distance between $W^{\unstable/\stable,+}(\mu)$, we consider the symplectic polar change of coordinates
\begin{align}\label{def:changePolars}
	q= 
	r \begin{pmatrix}
		\cos \tht \\ 
		\sin \tht
	\end{pmatrix},
	\qquad
	p = 
	R
	\begin{pmatrix}
		\cos \tht \\ 
		\sin \tht
	\end{pmatrix} 
	- \frac{G}{r} \begin{pmatrix}
		\sin \tht \\ 
		-\cos \tht
	\end{pmatrix},
\end{align} 
where 
$R$ is the radial linear momentum  and $G$ is the angular momentum.

We consider the $3$-dimensional section 
\[
\Sigma = \claus{(r,\tht,R,G) \in \reals \times \torus \times \reals^2 
\st r>1, \, \tht=\frac{\pi}2 \,}
\]
and denote by $(r^{\unstable}_*,\frac{\pi}2, R^{\unstable}_*,G^{\unstable}_*)$ and $(r^{\stable}_*,\frac{\pi}2,R^{\stable}_*,G^{\stable}_*)$ the first crossing of the invariant manifolds with this section. 

%
%

The next theorem measures the distance between these points for $0< \mu\ll 1$.

\begin{theorem}\label{theorem:mainTheorem}
There exists $\mu_0>0$ such that, for $\mu \in (0,\mu_0)$,
\[
\norm{(r^{\unstable}_*,R^{\unstable}_*,G^{\unstable}_*)-(r^{\stable}_*,R^{\stable}_*,G^{\stable}_*)}
=
\sqrt[3]{4} \,
\mu^{\frac13} e^{-\frac{A}{\sqrt{\mu}}} 
\boxClaus{\vabs{\CInn}+\OO\paren{\frac1{\vabs{\log \mu}}}},
\]
where:
\begin{itemize}
	\item The constant $A>0$ is the real-valued integral
	\begin{equation}\label{def:integralA}
		A= \int_0^{\frac{\sqrt{2}-1}{2}} \frac{2}{1-x}\sqrt\frac{x}{3(x+1)(1-4x-4x^2)}  dx\approx 0.177744.
	\end{equation}
	\item The constant $\CInn \in \complexs$ is the Stokes constant associated to the inner equation analyzed in \cite{articleInner} and in Theorem \ref{theorem:innerComputations} below.
	\end{itemize}
\end{theorem}

\begin{remark}\label{remark:sectiontheta}
We can prove the same result for any section 
\[
\Sigma(\tht_*) = \claus{(r,\tht,R,G) 
\in \reals \times \torus \times \reals^2 
\st r>1, \, \tht=\tht_* \,},
\]
with $\tht_* \in
(0,\tht_0)$ and $\tht_0=\arccos\paren{\frac12-\sqrt2}$ (the value of $\mu_0$ depends on how close to the endpoints of the interval $\tht_*$ is).
The section $\tht=\tht_0$ is close to the ``turning point'' of the invariant manifolds (see Figure \ref{fig:L3Outer}).
\end{remark}

The constant $A$ in \eqref{def:integralA} is derived from the values of the complex singularities of the separatrix of certain integrable averaged system, which is studied in the prequel paper \cite{articleInner}. 
The results obtained in  \cite{articleInner} about this separatrix are summarized in Theorem \ref{theorem:singularities} below.

The origin of the constant $\Theta$ appearing in Theorem \ref{theorem:mainTheorem} is explained in Theorem \ref{theorem:innerComputations}, which analyzes the so-called inner equation. This theorem  is also proven in  \cite{articleInner}. 
Moreover, in that paper it is seen, by a numerical computation,  that $\vabs{\CInn}\approx 1.63$. We expect that  one should be able to prove that $\vabs{\CInn}\neq 0$ by  means of  rigorous computer computations (see \cite{BCGS21}). Note that $\vabs{\CInn}\neq 0$ implies that there are not primary (i.e. one round) homoclinic orbits to $L_3$.


%

A fundamental problem in dynamical systems is to prove whether a given model has chaotic dynamics (for instance a Smale horseshoe). 
For many physically relevant models this is usually   remarkably difficult. This is the case of many Celestial Mechanics models, where most of the known chaotic motions have been found in nearly integrable regimes where there is an unperturbed problem which already presents some form of ``hyperbolicity''. This is the case  in the vicinity of collision orbits (see for example \cite{Moe89, BolMac06, Bol06, Moe07}) or close to parabolic orbits (which allows to construct chaotic/oscillatory motions), see~\cite{Sitnikov1960, Alekseev1976, LlibSim80, Moser2001, GMS16, GMSS17, GPSV21}. 
There are also several results in regimes far from integrable which rely on computer assisted proofs \cite{Ari02, WilcZgli03, Cap12, GierZgli19}. The problem tackled in this paper and \cite{articleInner} is radically different. Indeed, if one takes the  limit $\mu\to 0$ in \eqref{def:hamiltonianInitialNotSplit} one obtains the classical integrable Kepler problem in the elliptic regime, where  no hyperbolicity is present. Instead, the (weak) hyperbolicity is created by the $\mathcal{O}(\mu)$ perturbation, which can be captured considering an integrable averaged Hamiltonian along the $1:1$ mean motion resonance\footnote{The $1:1$ averaged Hamiltonian has been also studied to obtain ``good'' approximations for the global dynamics in the $1:1$ resonant zone, see for example \cite{RNP16, AlePou21} and the references therein.}.

One of the classical methods  to construct chaotic dynamics is the Smale-Birkhoff homoclinic theorem   by proving the existence of  transverse homoclinic orbits to invariant objects, most commonly,  periodic orbits.
%
%
Certainly the breakdown of homoclinic orbits to the critical point $L_3$ given by Theorem~\ref{theorem:mainTheorem} does not lead to the existence of chaotic orbits. However, one should expect that Theorem~\ref{theorem:mainTheorem} implies that there exist Lyapunov periodic orbits exponentially close to $L_3$ whose stable and unstable invariant manifolds intersect transversally. This would create chaotic motions ``exponentially close'' to $L_3$ and its invariant manifolds (see \cite{articleChaos}).
As already mentioned, Theorem \ref{theorem:mainTheorem} rules out the existence of primary homoclinic connections to $L_3$ in the RPC$3$BP for $0< \mu\ll 1$. However, it does not prevent the existence of multiround homoclinic orbits, that is homoclinic orbits which pass close to $L_3$ multiple times.
%
%
It has been conjectured (see for instance~\cite{BMO09}, where the authors analyze this problem numerically) that  multi-round homoclinic connections to $L_3$ should exist for a sequence of values $\claus{\mu_k}_{k \in \naturals}$ satisfying $\mu_k\to 0$ as $k \to \infty$. 

%

%
%
%
%

\paragraph{A first step towards proving Arnold diffusion along the $1:1$ mean motion resonance in the $3$-Body Problem?}
Consider the $3$-Body Problem in the planetary regime, that is one massive body (the Sun) and two small bodies (the planets) performing approximate ellipses (including the ``Restricted limit'' when one of planets has mass zero).  A fundamental problem is to assert whether such configuration is stable (i.e. is the Solar system stable?). Thanks to Arnold-Herman-F\'ejoz KAM Theorem, many of such configurations are stable, see  \cite{Arnold63,Fejoz04}. However, it is widely expected that there should be strong instabilities created by Arnold diffusion mechanisms (as conjectured by Arnold in \cite{Arnold64}). In particular, it is widely believed that one of the main sources of such instabilities dynamics are the mean motion resonances, where the period of the two planets is resonant (i.e. rationally dependent) \cite{FGKR16}.

The RPC$3$BP has too low dimension (2 degrees of freedom) to possess Arnold diffusion. However, since it can be seen as a first order for higher dimensional models, the analysis performed in this paper can be seen as a humble first step towards constructing Arnold diffusion in the $1:1$ mean motion resonance. In this resonance,  the RPC$3$BP  has a normally hyperbolic invariant manifold given by the center manifold of the Lagrange point $L_3$. This normally hyperbolic invariant manifold is foliated by the classical Lyapunov periodic orbits. One should expect that the techniques developed in the present paper would allow to prove that the invariant manifolds of these periodic orbits intersect transversally within the corresponding energy level of \eqref{def:hamiltonianInitialNotSplit}. Still, this is a much harder problem than the one considered in this paper and the technicalities involved would be considerable. 

This transversality would not lead to Arnold diffusion due to the low dimension of the RPC3BP. However, if one considers either the Restricted Spatial Circular $3$-Body Problem with small $\mu>0$ which has three degrees of freedom, the Restricted Planar Elliptic $3$-Body Problem with small $\mu>0$ and eccentricity of the primaries $e_0>0$, which has two and a half degrees of freedom, or the ``full'' planar $3$-Body Problem (i.e. all three masses positive, two small) which has three degrees of freedom (after the symplectic reduction by the classical first integrals) one should be able to construct orbits with a drastic change in angular momentum (or inclination in the spatial setting). 

In the Restricted Planar Elliptic $3$-Body Problem the change of angular momentum would imply   the transition of the zero mass body orbit from a close to circular ellipse to a more eccentric one. In the full 3BP, due to total angular momentum conservation, the angular momentum would be transferred from one body to the other changing both osculating ellipses.
This behavior would be analogous to that of \cite{FGKR16} for the $3:1$ and $1:7$ resonances. In that paper, the transversality between the invariant manifolds of the normally hyperbolic invariant manifold was checked numerically for the realistic Sun-Jupiter mass ratio $\mu=10^{-3}$. 
Arnold diffusion instabilities have been analyzed numerically for the  Restricted Spatial Circular $3$-Body Problem  in \cite{SSST14}.

%
%
%
%
%
%
%
%
%
%

\subsection{The strategy to prove Theorem \ref{theorem:mainTheorem}}

The main difficulty in proving Theorem \ref{theorem:mainTheorem} is that the 
distance between the stable and unstable manifolds of $L_3$ is exponentially small with respect to $\sqrt\mu$ (this is also  usually known as a \emph{beyond all orders} phenomenon). This implies that the classical Melnikov Method \cite{GuckenheimerHolmes} to detect the breakdown of homoclinics cannot be applied.

To prove Theorem \ref{theorem:mainTheorem}, we follow the strategy of exponentially small splitting of separatrices (already outlined in \cite{articleInner}) which goes back to the seminal work by Lazutkin \cite{Laz84, Laz05}. See \cite{articleInner} for a list of references on the recent developments in the field of exponentially small splitting of separatrices. In particular, we follow similar strategies of those in  \cite{BFGS12,BCS13}.


In the present work the first order of the difference between manifolds is not given by the Melnikov function.
Instead, we must  derive and analyze an inner equation which provides the dominant term of this distance. As a consequence, we need to ``match'' (i.e. compare) certain  solutions of the inner equation  with the parameterizations of the perturbed invariant manifolds.
%

The first part of the proof, that was completed in the prequel \cite{articleInner}, dealt with the following steps:
\begin{enumerate}[label*=\Alph*.]
	\item
	We perform a change of coordinates to capture the slow-fast dynamics of the system. 
	The first order of the  new Hamiltonian has a saddle point with an homoclinic connection (also known as separatrix) and  a fast harmonic oscillator.
	%
	%
	\item We study the analytical continuation of the time-parametrization of the	separatrix of this first order. 
	In particular, we obtain its maximal strip of analyticity and the singularities at the boundary of this strip.
	%
	%
	\item We derive the inner equation.
	\item We study  two special solutions which will be ``good approximation'' of the perturbed invariant manifolds near the singularities of the unperturbed separatrix (see Step F below).
	%
\end{enumerate}

The remaining steps necessary to complete the proof of Theorem~\ref{theorem:mainTheorem} are the following:
\begin{enumerate}[label*=\Alph*.]
	\item[E] We prove the existence of the analytic continuation of the parametrizations of the invariant manifolds of $L_3$, $W^{\unstable,+}(\de)$ and $W^{\stable,+}(\de)$, in an appropriate complex domain called boomerang domain.
	This domain contains a segment of the real line and intersects a sufficiently small neighborhood of the singularities of the unperturbed separatrix.
	%
	%
	\item[F.] By using complex matching techniques, we show that, close to the singularities of the unperturbed separatrix, the solutions of the inner equation obtained in Step D are ``good approximations'' of the parameterizations of the perturbed invariant manifolds obtained in Step E.
	%
	%
	\item[G.] We obtain an asymptotic formula for the difference between the perturbed invariant manifolds by proving that the dominant term comes from the difference between the solutions of the inner equation.
	%
	%
\end{enumerate}

The structure of this paper goes as follows. 
In Section \ref{section:introductionPoincare} we perform the change of coordinates introduced in Step A and state Theorem \ref{theorem:mainTheoremPoincare}, which is a reformulation of Theorem \ref{theorem:mainTheorem} in this new set of variables.
Then, in Section \ref{section:resultsOuter}, we state the results concerning Steps B, C and D above  (which are proven in \cite{articleInner}) and we carry out Steps E, F and G. These steps lead to the proof of  Theorem \ref{theorem:mainTheoremPoincare}.
%
Sections \ref{section:proofH-existence} and \ref{section:proofG-matching} are devoted to proving the results in Section \ref{section:resultsOuter} which concern Steps E and F.

\section{A singular formulation of the problem}
\label{section:introductionPoincare}

The Lagrange point $L_3$ is a centre-saddle equilibrium point of the Hamiltonian $h$ in \eqref{def:hamiltonianInitialNotSplit} whose eigenvalues,  as $\mu \to 0$, satisfy 
\begin{equation*}
	\mathrm{Spec}  = \claus{\pm \sqrt{\mu} \, \rho(\mu), \pm i  \, \omega(\mu) },
	\quad
	\text{with} \quad \left\{
	\begin{array}{l}
		\rho(\mu)=\sqrt{\frac{21}{8}}  + \OO(\mu),\\[0.4em]
		\omega(\mu)=1 + \frac{7}{8}\mu + \OO(\mu^2).
	\end{array}\right.
\end{equation*}

The center and saddle eigenvalues are found at different time-scales. Moreover, when $\mu=0$, the unstable and stable manifolds of $L_3$ ``collapse'' to a circle of critical points.
%
%
Applying a suitable singular change of coordinates, the Hamiltonian $h$ can be written as a perturbation of a pendulum-like Hamiltonian weakly coupled with a fast oscillator.
The construction of this change of variables is presented in detail in Section 2.1 in \cite{articleInner}. 
In this section, we summarize the most important properties of this set of coordinates.


%
The Hamiltonian $h$ expressed in the classical (rotating) Poincar\'e coordinates,
$
\phi_{\Poi}: (\la,L,\eta,\xi) \to (q,p),
$ 
defines a Hamiltonian system with respect to the symplectic form $d\la \wedge dL + i \h d\eta \wedge d\xi$ and the Hamiltonian
\begin{equation} \label{def:hamiltonianPoincare}
H^{\Poi} = H_0^{\Poi}	+ \mu H_1^{\Poi},
\end{equation}
with
\begin{equation}\label{def:hamiltonianPoincare01}
\begin{split}
H_0^{\Poi}(L,\eta,\xi) &= -\frac{1}{2L^2} - L +  \eta \xi 
\qquad \text{and} \qquad
H_1^{\Poi} = h_1 \circ \phi_{\Poi}.
\end{split}
\end{equation}
Moreover, the critical point $L_3$ satisfies
\begin{equation*}
\la=0, \qquad \quad(L,\eta,\xi) = (1,0,0) + \OO(\mu)
\end{equation*}
and the linearization of the vector field at this point has, at first order, an uncoupled nilpotent and center blocks,
\begin{equation*}
\begin{pmatrix}
0 & -3 & 0 & 0 \\
0 & 0 & 0 & 0 \\
0 & 0 & i & 0 \\
0 & 0 & 0 & -i 
\end{pmatrix}
+ \OO(\mu).
\end{equation*}
Since $\phi_{\Poi}$ is an implicit change of coordinates, there is no explicit expression for $H_1^{\Poi}$.
However, it is possible to obtain series expansion in powers of $(L-1,\eta,\xi)$,
(see Lemma 4.1 in \cite{articleInner} and also Appendix \ref{appendix:proofH-technical}).

To capture the slow-fast dynamics of the system,
renaming
\[
\de= \mu^{\frac{1}{4}},
\]
we perform the singular symplectic scaling 
\begin{equation}\label{def:changeScaling}
	\phi_{\sca}:
	(\la, \La, x,  y) 
	\mapsto 
	(\la,L,\eta,\xi),
	\qquad
	L = 1 + \de^2 \La , \quad
	\eta = \de x , \quad
	\xi = \de y 
\end{equation}
and the time {reparametrization $t = \de^{-2} \tau$}. 
Defining the potential
\begin{equation}\label{def:potentialV}
\begin{split}
V(\la) 
&=  H_1^{\Poi}(\la,1,0,0;0)
= 1 - \cos \la - \frac{1}{\sqrt{2+2\cos \la}},
\end{split}
\end{equation}
the Hamiltonian system associated to $H^{\Poi}$, expressed in scaled coordinates,
defines a Hamiltonian system with respect to the symplectic form $d\la \wedge d\La + i dx \wedge dy$ 
and the Hamiltonian
\begin{equation} \label{def:hamiltonianScaling}
	\begin{split}
		{H} = {H}_{\pend} + {H}_{\osc} + H_1,
	\end{split}
\end{equation}
where
\begin{align} 
	{H}_{\pend}(\la,\La) &= -\frac{3}{2} \La^2  + V(\la), \qquad
	{H}_{\osc}(x,y; \de) = \frac{x y}{\de^2},
	\label{def:HpendHosc} \\
	H_1(\la,\La,x,y;\de) &=
	H_1^{\Poi}(\la,1+\de^2\La,\de x,\de y;\de^4)  - V(\la) +
	\frac{1}{\de^4} F_{\pend}(\de^2\La)
	\label{def:hamiltonianScalingH1}
\end{align}
and
\begin{equation}\label{def:Fpend}
	F_{\pend}(z) = \paren{-\frac{1}{2(1+z)^2}-(1+z)}+\frac{3}{2} + \frac{3}{2}z^2 = \OO(z^3).
\end{equation}
Therefore, we can define the ``new'' first order
\begin{equation}\label{def:hamiltonianScalingH0}
	H_0 = H_{\pend} + H_{\osc}.
\end{equation}
From now on, we refer to $H_0$ as the unperturbed Hamiltonian and we identify $H_1$ as the perturbation.

The next proposition, proven in \cite[Theorem 2.1]{articleInner}, gives some properties of the Hamiltonian $H$.

\begin{proposition}\label{proposition:HamiltonianScaling}
The Hamiltonian $H$,
away from collision with the primaries,
 is real-analytic  in the sense of $\conj{H(\la,\La,x,y;\de)}= H(\conj{\la},\conj{\La},y,x;\conj{\de}).$ 

Moreover, for $\de > 0$ small enough, 
\begin{itemize}
\item The critical point $L_3$ expressed in  coordinates $(\la,\La,x,y)$ is given by
\begin{equation}\label{def:pointL3sca}
	\Ltres(\de)=\paren{0,
		\de^2 \LtresLa(\de),
		\de^3 \Ltresx(\de),
		\de^3 \Ltresy(\de)},
\end{equation}
with $\vabs{\LtresLa(\de)}$, $\vabs{\Ltresx(\de)}$, $\vabs{\Ltresy(\de)} \leq C$, for some constant $C>0$ independent of $\de$.
\item The point $\Ltres(\de)$ is a saddle-center equilibrium point and its linearization is
\begin{equation*}
	\begin{pmatrix}
		0 & -3 & 0 & 0 \\
		-\frac{7}{8} & 0 & 0 & 0 \\
		0 & 0 & \frac{i}{\de^2} & 0 \\
		0 & 0 & 0 & -\frac{i}{\de^2}
	\end{pmatrix} + \OO(\de).
\end{equation*}
Therefore, it possesses a one-dimensional unstable and stable manifolds, $\WW^{\unstable}(\de)$ and  $\WW^{\stable}(\de)$. 
%
\end{itemize}
\end{proposition}	

%
%
%


The unperturbed system given by $H_0$ in \eqref{def:hamiltonianScalingH0} has two homoclinic connections in the $(\la,\La)$-plane associated to the saddle point $(0,0)$ and described by the energy level $H_{\pend}(\la,\La)=-\frac12$ (see Figure \ref{fig:separatrix}).
\begin{figure}[t] 
	\centering
	\begin{overpic}[scale=0.7]{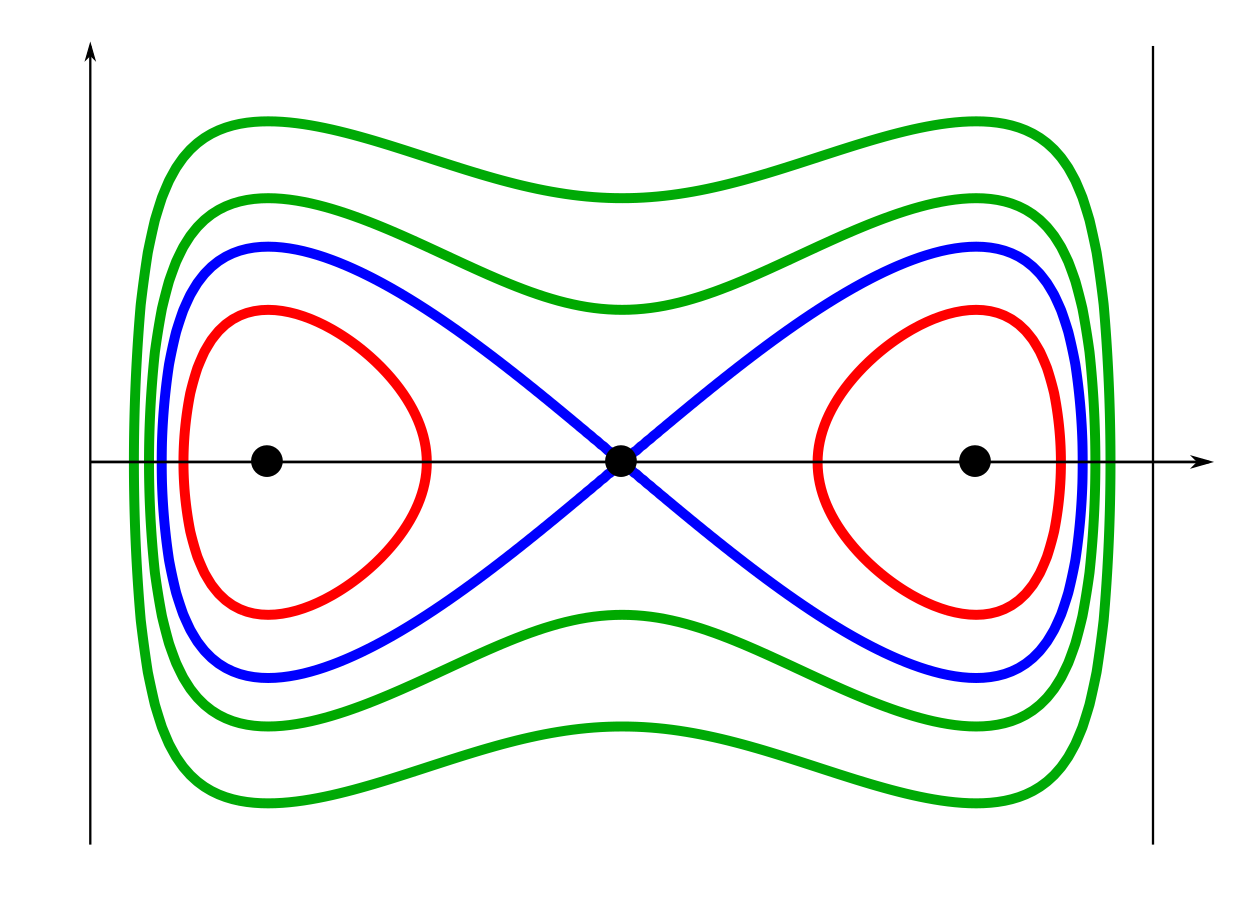}
		\put(93,31){{$\pi$}}
		\put(77,28.5){{$\frac{2}{3}\pi$}}
		\put(49,29){{$0$}}
		\put(17,28.5){{$-\frac{2}{3}\pi$}}
		\put(-2,31){{ $-\pi$}}
		\put(98,34){{$\la$}}
		\put(5,70){{$\La$}}
	\end{overpic}
	\caption{Phase portrait of the system given by Hamiltonian ${H}_{\pend}(\la,\La)$ on~\eqref{def:HpendHosc}. On blue the two separatrices.}
	\label{fig:separatrix}
\end{figure}
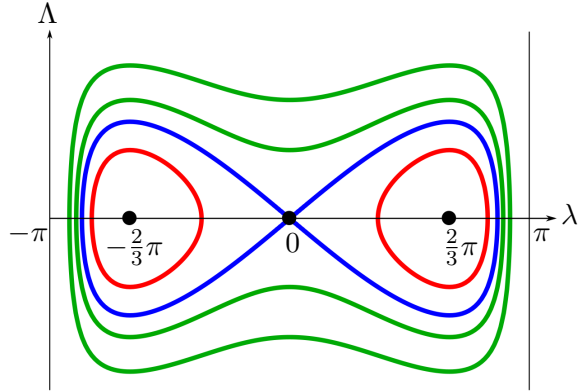
We define 
\begin{equation}\label{def:lambda0}
\la_0 = \arccos\paren{\frac12-\sqrt2},
\end{equation}
which satisfies $H_{\pend}(\la_0,0)=-\frac12$ so that, for the unperturbed system, $\la_0$ is the ``turning point'' in the $(\la,\La)$ variables. 
We will see that, in our regime, $\tht \approx \la$ and thus the value of $\tht_0$ introduced in Remark \ref{remark:sectiontheta} is indeed close to the ``turning point'' of the invariant manifolds (see Figure \ref{fig:L3Outer}).


We rewrite Theorem \ref{theorem:mainTheorem}, in fact the more general result in Remark \ref{remark:sectiontheta}, in the  set of coordinates $(\la,\La,x,y)$.
%
%
%
For $\la_* \in (0,\la_0)$, we consider the $3$-dimensional section
\begin{align*} 
\SSS(\la_*) = \claus{(\la,\La,x,y) \in \reals^2 \times \complexs^2 \st \la = 
\la_*,\, \La>0, \, {x}=\conj{y}
},
\end{align*}
which is transverse to the flow of $H$, and we define the first crossings of the invariant manifolds $\WW^{\unstable,\stable}(\de)$ with this section as
$(\la_*, \La^{\unstable}_{*},x^{\unstable}_{*},y^{\unstable}_{*})$
and 
$(\la_*, \La^{\stable}_{*},x^{\stable}_{*},y^{\stable}_{*})$.
%

\begin{theorem}\label{theorem:mainTheoremPoincare}
Fix an interval $[\la_1,\la_2] \subset (0,\la_0)$ with $\la_0$ as given in \eqref{def:lambda0}.
Then, there exists $\de_0>0$  and $\cttTheorem>0$ such that, 
for $\de \in (0,\de_0)$ 
and $\la_* \in [\la_1,\la_2]$,  the first crossings are analytic with 
respect to $\lambda^*$ and 
\begin{equation}\label{def:estimatesInvMan}
|\La^{\diamond}_{*}|\leq \cttTheorem,
\qquad |x^{\diamond}_{*}|,|y^{\diamond}_{*}|\leq \cttTheorem \delta^3,\qquad 
\diamond=\unstable, \stable.
\end{equation}
Moreover,
\begin{align*}
\vabs{x^{\unstable}_{*}-x^{\stable}_{*}} = 
\vabs{y^{\unstable}_{*}-y^{\stable}_{*}} 
&= 
\sqrt[6]{2} \,
\de^{\frac{1}{3}}  e^{-\frac{A}{\de^2}} 
	\boxClaus{ \vabs{\CInn}
	+ \OO\paren{\frac{1}{\vabs{\log \de}}}},\\
\vabs{\La^{\unstable}_*-\La^{\stable}_{*}} &= \OO(\de^{\frac43}e^{-\frac{A}{\de^2}}),
\end{align*}
where  $A$ and  $\CInn$ are the constants introduced in Theorem \ref{theorem:mainTheorem}.

\end{theorem}

\subsection{Proof of Theorem~\ref{theorem:mainTheorem}}
\label{subection:undoChanges}

To prove Theorem \ref{theorem:mainTheorem} (and Remark \ref{remark:sectiontheta}) from Theorem \ref{theorem:mainTheoremPoincare} we need to ``undo'' the changes of coordinates $\phi_{\Poi}$ and $\phi_{\sca}$ and adjust the section from $\lambda=\text{constant}$ to $\theta=\text{constant}$.

First, we consider the change $\phi_{\sca}$ given by $(\la,L,\eta,\xi)=(\la,1+\de^2\La,\de x, \de y)$, (see \eqref{def:changeScaling}).
For $\la_* \in [\la_1,\la_2]$ we define
\begin{align}\label{proof:poincareEstimatesExistence}
	L^{\diamond}(\la_*;\de) = 1+\de^2 \La_*^{\diamond},
	\qquad
	\eta^{\diamond}(\la_*;\de) = \de x^{\diamond}_*, 
	\qquad
	\xi^{\diamond}(\la_*;\de) = \de y^{\diamond}_*,
	\quad
	\text{for } \diamond=\unstable,\stable.
\end{align}	
Then, by Theorem \ref{theorem:mainTheoremPoincare}, one has
\begin{equation}\label{proof:poincareEstimatesDifference}
\begin{split}	
	|\D L(\la_*;\de)| &= 
	|L^{\unstable}(\la_*;\de) - L^{\stable}(\la_*;\de) |
	= \OO \paren{\de^{\frac{10}3} e^{-\frac{A}{\de^2}}},
	\\
	|\D \eta(\la_*;\de)|
	&=
	|\eta^{\unstable}(\la_*;\de) - \eta^{\stable}(\la_*;\de)| 
	= 	\sqrt[6]{2} \,
\de^{\frac{4}{3}}  e^{-\frac{A}{\de^2}} 
	\boxClaus{ \vabs{\CInn}
	+ \OO\paren{\frac{1}{\vabs{\log \de}}}},\\
	\qquad
	\D \xi(\la_*;\de) & = \conj{\D \eta(\la_*;\de)}.
	\end{split}	
\end{equation}

Next, we study the change $\phi_{\Poi}$. In the following result, we give a series expression of the polar coordinates with respect to the Poincar\'e elements. 
Its proof is a 
direct consequence of the definition of the Poincar\'e variables (see, for 
instance, Section 4.1 in 
\cite{articleInner}).

\begin{lemma}\label{lemma:seriesExpansionPoincare}
Fix $\varrho>0$.
Then, for $\vabs{(L-1,\eta,\xi)} \ll 1$ and $\vabs{\Im \la}\leq \varrho$, the 
polar  coordinates $(r,\tht,R,G)$ introduced  in \eqref{def:changePolars} satisfy
\begin{align*}
	r &= 1 +2(L-1)
	-\frac{e^{-i\la}}{\sqrt2}  \eta
	-\frac{e^{i\la}}{\sqrt2} \xi
	+ \OO(L-1,\eta,\xi)^2, 
	\\
	\tht &= \la + i\sqrt2 e^{-i\la}\eta -i \sqrt2 e^{i\la} \xi + \OO(L-1,\eta,\xi)^2, 
	\\
	R &= 
	{\frac{i e^{-i\la}}{\sqrt2}  \eta
	-
	\frac{i e^{i\la}}{\sqrt2} \xi}
	+
	\OO(L-1,\eta,\xi)^2,
	\qquad
	G = L - \eta \xi.
\end{align*}
\end{lemma}

Since in Theorem \ref{theorem:mainTheoremPoincare} the distance is measured in the section $\la=\la_*$ whereas the Theorem \ref{theorem:mainTheorem}, and more generally Remark \ref{remark:sectiontheta}, measures it in the section  $\tht=\tht^*$, we must ``translate'' the estimates in 
\eqref{proof:poincareEstimatesDifference} to the new section.
%
By Lemma \ref{lemma:seriesExpansionPoincare}, let $g_{\tht}$ be the function such that $\tht=\la+g_{\tht}(\la,L,\eta,\xi)$.
Then, for $\diamond=\unstable,\stable$, we consider 
\[
F^{\diamond}(\la,\tht,\de) =\tht - \la + g_{\tht}\paren{\la,L^{\diamond}(\la;\de),\eta^{\diamond}(\la;\de),\xi^{\diamond}(\la;\de)}.
\]
Applying the Implicit Function Theorem, Lemma 
\ref{lemma:seriesExpansionPoincare} and that, by
\eqref{proof:poincareEstimatesExistence},
$L^{\diamond}(\la;0)=1$ and
$\eta^{\diamond}(\la;0) = 
\xi^{\diamond}(\la;0)=0$,
then there exist function 
$\wh\la^{\diamond}(\tht;\de)$ such that 
$F^{\diamond}(\wh\la^{\diamond}(\tht;\de),\tht,\de)=0$ and
\begin{equation}\label{proof:lausSerie}
\begin{split}	
\wh\la^{\diamond}(\tht;\de) =& \, \tht 
-i \sqrt2 e^{-i\tht}\wh\eta^{\diamond}(\tht;\de)	
+i \sqrt2 e^{i\tht}\wh\xi^{\diamond}(\tht;\de) \\
&+ \OO\paren{\wh L^{\diamond}(\tht;\de)-1,\wh \eta^{\diamond}(\tht;\de), 
\wh\xi^{\diamond}(\tht;\de)}^2,
\end{split}
\end{equation}
with 
$\wh\eta^{\diamond}(\tht;\de)=\eta^{\diamond}(\wh\la^{\diamond}(\tht;\de);\de)$, 
$\wh\xi^{\diamond}(\tht;\de)=\xi^{\diamond}(\wh\la^{\diamond}(\tht;\de);\de)$ 
and $\wh L^{\diamond}(\tht;\de)=L^{\diamond}(\wh\la^{\diamond}(\tht;\de);\de)$.
Notice that, by \eqref{def:estimatesInvMan} (plus Cauchy estimates for their 
derivatives) and \eqref{proof:poincareEstimatesExistence},
\begin{equation*}\label{proof:poincareEstimatesExistencelambda}
\wh\la^{\diamond}(\tht;\de) = \tht + \OO(\de^4).
\end{equation*}
Thus,  for any  $[\tht_1,\tht_2] \subset (0,\la_0)$ and $\delta$ small enough, 
there exists $[\lambda_1,\lambda_2] \subset (0,\la_0)$ such that, for $\tht \in 
[\tht_1,\tht_2]$ one has  $\wh\la^{\unstable,\stable}(\tht;\de) \in 
[\lambda_1,\lambda_2] $.
In addition,
\begin{equation}\label{proof:poincareEstimatesExistenceTheta}
\begin{aligned}	
\wh{L}^{\diamond}(\tht;\de) &= L^{\diamond}(\tht;\de) + \OO(\de^6) = 1 + \OO(\de^2),
\\ 
\wh\eta^{\diamond}(\tht;\de)
&= \eta^{\diamond}(\tht;\de) + \OO(\de^8)
= \OO(\de^4),
\\
\wh\xi^{\diamond}(\tht;\de)
&= \xi^{\diamond}(\tht;\de) + \OO(\de^8)
=\OO(\de^4).
\end{aligned}
\end{equation}
Then, since $\La^{\unstable,\stable}_*>0$, by \eqref{proof:poincareEstimatesExistence} one has that  $\wh{L}^{\unstable,\stable}(\tht;\de)>1$ for $\tht \in [\tht_1,\tht_2]$.
Moreover, by Lemma \ref{lemma:seriesExpansionPoincare} and taking $\de$ small enough, one has $r^{\unstable,\stable}(\tht)-1>0$.

The difference between the invariant manifolds in a section of fixed $\tht \in[\tht_1,\tht_2]$ is given by
\begin{align*}
	\D \wh\la(\tht;\de) =& \,
	\wh\la^{\unstable}(\tht;\de) -
	\wh\la^{\stable}(\tht;\de), 
	&
	\D \wh{L}(\tht;\de) =& \,  
	\wh{L}^{\unstable}(\tht;\de)-
	\wh{L}^{\stable}(\tht;\de), 
	\\
	\D \wh{\eta}(\tht;\de) =& \, 
	\wh\eta^{\unstable}(\tht;\de)-
	\wh\eta^{\stable}(\tht;\de), 
	&
	\D \wh{\xi}(\tht;\de) =& \, 
	\wh\xi^{\unstable}(\tht;\de)-
	\wh\xi^{\stable}(\tht;\de).
\end{align*}
Then, by \eqref{proof:lausSerie} and \eqref{proof:poincareEstimatesExistenceTheta}, one has that
\begin{align*}
\D \wh\la(\tht;\de) = 
-i\sqrt2 e^{-i\tht} \D\wh \eta(\tht;\de)
+i\sqrt2 e^{i\tht} \D\wh  \xi(\tht;\de)+
\OO\paren{\de^2 \D \wh{L}(\tht;\de),
\de^4 \D\wh \eta(\tht;\de),
\de^4\D \wh\xi(\tht;\de)}.
\end{align*} 
Moreover, by the mean value theorem,
\eqref{proof:poincareEstimatesDifference} and
\eqref{proof:poincareEstimatesExistenceTheta},
\begin{equation*}
\begin{split}	
\D \wh{L}(\tht;\de) =&\,
\D L(\wh\la^{\unstable}(\tht;\de);\de) +\wh 
L^\stable(\wh\la^{\unstable}(\tht;\de);\de)-\wh 
L^\stable(\wh\la^{\stable}(\tht;\de);\de)\\
%
 =&\, \OO(\de^{\frac{10}3} e^{-\frac{A}{\de^2}})+\de^2\OO\left(\D \wh
\la(\tht;\de)\right).
\end{split}
\end{equation*}
Analogously,
\begin{equation*}
\begin{split}	
	\D \wh{\eta}(\tht;\de) &= \D\eta(\wh\la^{\unstable}(\tht;\de);\de) + 
\de^ 4\OO\left(\D 
\wh\la(\tht;\de)\right) ,
	\\
	\D \wh{\xi}(\tht;\de) &= \conj{\D\eta(\la^{\unstable}(\tht;\de);\de)} +
	\de^ 4\OO\left(\D \wh
\la(\tht;\de)\right) .
\end{split}	
\end{equation*}
Therefore, using  \eqref{proof:poincareEstimatesDifference},  one can conclude 
that 
 \begin{align*}	
|\D \wh \la(\theta;\de)| &=  \OO \paren{\de^{\frac{4}3} 
e^{-\frac{A}{\de^2}}},
&
|\D\wh \eta(\theta;\de)|
&=\sqrt[6]{2} \,
\de^{\frac{4}{3}}  e^{-\frac{A}{\de^2}} 
\boxClaus{ \vabs{\CInn}
	+ \OO\paren{\frac{1}{\vabs{\log \de}}}},
\\
|\D \wh L(\theta;\de)| &=  \OO \paren{\de^{\frac{10}3} 
e^{-\frac{A}{\de^2}}},
&
\D\wh \xi(\theta;\de) & = \conj{\D \wh\eta(\theta;\de)}.
\end{align*}	
Once we have adjusted the transverse section, it only remains to apply  Lemma 
\ref{lemma:seriesExpansionPoincare}  to translate these differences 
to polar coordinates. That is,
%
\begin{align*}
	r^{\unstable} - r^{\stable} 
	&= -\sqrt2 \cos \tht \, \Re \D \wh{\eta}(\tht;\de)
	-\sqrt2 \sin\tht \, \Im \D \wh{\eta}(\tht;\de)
	+ \OO(\de^{\frac{10}3} e^{-\frac{A}{\de^2}}), 
	\\
	R^{\unstable} - R^{\stable} 
	&= -\sqrt2 \cos \tht \, \Im \D \wh{\eta}(\tht;\de)
	+\sqrt2 \sin\tht \, \Re \D \wh{\eta}(\tht;\de)
	+ \OO(\de^{\frac{16}3} e^{-\frac{A}{\de^2}}),
	\\
	G^{\unstable} - G^{\stable} 
	&= \OO(\de^{\frac{10}3} e^{-\frac{A}{\de^2}}),
\end{align*}
which implies
\begin{align*}
\norm{(r^{\unstable},R^{\unstable},G^{\unstable})-(r^{\stable},R^{\stable},G^{\stable})}
=&\sqrt{2}
\vabs{\D \wh{\eta}(\tht;\de)}
+ \OO(\de^{\frac{10}3} e^{-\frac{A}{\de^2}}) \\
=&
\sqrt[3]{4} \, \de^{\frac43} e^{-\frac{A}{\de^2}} 
\boxClaus{\vabs{\CInn}+\OO\paren{\frac1{\vabs{\log \de}}}}.
\end{align*}
To conclude the proof of  Theorem 
\ref{theorem:mainTheorem}, it is enough to recall that $\de=\mu^{\frac14}$.

 
%

\section{Proof of Theorem~\ref{theorem:mainTheoremPoincare}}
\label{section:resultsOuter}
In this section, we present the main steps necessary to prove Theorem \ref{theorem:mainTheoremPoincare} (see the list in Section \ref{section:introduction}) and complete its proof.
In Section \ref{section:singularitiesOuter} we summarize the results concerning the analysis of the separatrix of the unperturbed Hamiltonian $H_{\pend}$ (see \eqref{def:HpendHosc}) done in \cite{articleInner} (Step B).
In Section \ref{section:outer}, we prove the existence of parametrizations of the perturbed invariant manifolds in suitable complexs domains (Step E).
In Section \ref{section:differenceInner}, we study the difference between the perturbed manifolds near the singularities of the perturbed separatrix.
In particular, in Section \ref{subsection:innerHeuristics}, we summarize the results concerning the derivation (Step C) and analysis (Step D) of the inner equation obtained in \cite{articleInner} and, in Section \ref{subsection:matching}, we
compare certain solutions of the inner equation with the parametrizations of the perturbed manifolds by means of complex matching techniques (Step F).
Finally, in Section \ref{section:difference}, we combine all the previous results to obtain the dominant term of the difference between the invariant manifolds and prove Theorem \ref{theorem:mainTheoremPoincare} (Step G).

%

\subsection{Analytical continuation of the unperturbed separatrix} \label{section:singularitiesOuter}

The unperturbed Hamiltonian 
\[
H_0(\la, \La, x,y) = 
H_{\pend}(\la, \La) + H_{\osc}(x,y)\]
(see \eqref{def:hamiltonianScalingH0}) possesses a saddle with two separatrices 
in the $(\la,\La)$-plane (see 
Figure~\ref{fig:separatrix}).
%
%
%
%
Let us consider the real-analytic time parametrization of the 
 separatrix with $\la\in (0,\pi)$,
\begin{equation}\label{eq:separatrixParametrization}
	\begin{split}
		\s:\reals &\to \torus \times \reals \\
		t &\mapsto \s(t)=(\la_h(t),\La_h(t)),
	\end{split}
\end{equation}
with initial condition $\s(0)=(\la_0,0)$ where $\la_0 =
\arccos\paren{\frac12-\sqrt{2}} \in \left(\frac{2}{3}\pi,\pi\right)$.
%
%
%

The following result (which encompass Theorem 2.2, Proposition 2.3 and 
Corollary~2.4 in \cite{articleInner})
gives the properties of the analytic extension of $\s(t)$ to the
domain 
\begin{equation}\label{def:dominiBow}
	\begin{split}	
		\Pi^{\mathrm{ext}}_{A, \beta} =& \claus{
			t \in \complexs \st 
			\vabs{\Im t} < \tan \beta \, \Re t + A }
		\cup \\
		&\claus{
			t \in \complexs \st 
			\vabs{\Im t} < -\tan \beta \, \Re t + A},
	\end{split}
\end{equation}
with $A$ as given in~\eqref{def:integralA} (see Figure~\ref{fig:dominiBow}).
\begin{figure} 
	\centering
	\begin{overpic}[scale=0.5]{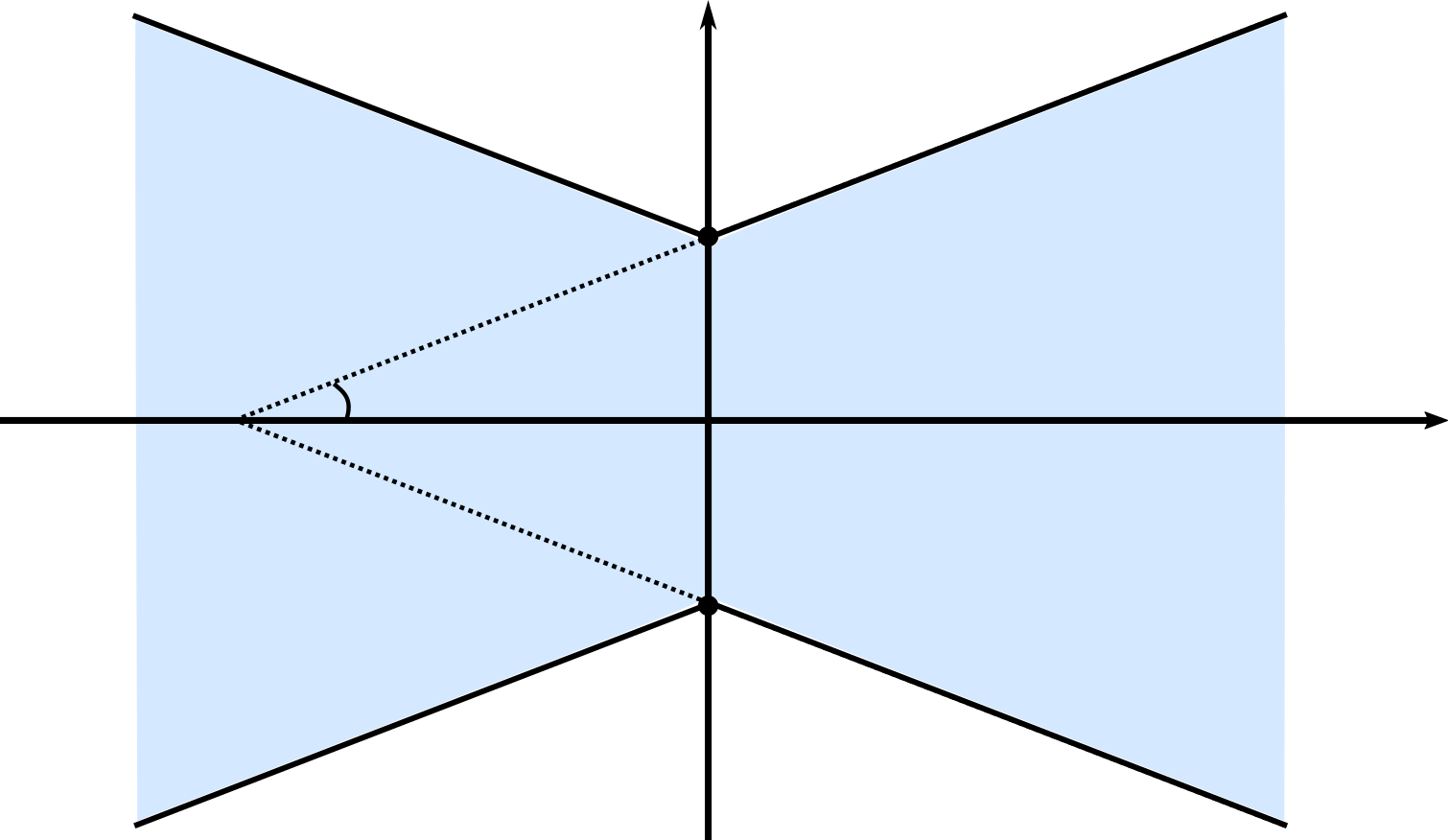}
		\put(101,27.5){\footnotesize$\Re t$}
		\put(45.5,59){\footnotesize$\Im t$}
		\put(28,30){\footnotesize $\beta$}
		\put(42,45){\footnotesize $iA$}
		\put(70,32){$\Pi_{A,\beta}^{\mathrm{ext}}$}
	\end{overpic}
	\bigskip
	\caption{Representation of the domain $\Pi_{A,\beta}^{\mathrm{ext}}$ in~\eqref{def:dominiBow}.}
	\label{fig:dominiBow}
\end{figure}

\begin{theorem}\label{theorem:singularities}
The real-analytic time parametrization $\s$ defined 
in~\eqref{eq:separatrixParametrization} satisfies:
\begin{itemize}
\item There exists $0<\betaBow<\frac{\pi}{2}$ such that $\s(t)$ extends 
analytically to $\Pi_{A,\betaBow}$.
%
\item $\s(t)$ has only two singularities on $\partial \Pi^{\mathrm{ext}}_{A, 
\betaBow}$ at $t=\pm iA$.
\item There exists $\upsilon>0$ such that,
for $t \in \complexs$ with $\vabs{t -iA}<\upsilon$ and $\arg(t -iA) \in (-\frac{3\pi}{2},\frac{\pi}{2})$, 
\begin{equation*}
	\begin{split}
		&\la_h(t) =  \pi + 3\al_{+} (t- iA)^{\frac{2}{3}} + \OO(t- iA)^{\frac{4}{3}}, \\[0.5em]
		&\La_h(t) =  -\frac{ 2 \al_{+}}{3} \frac{1}{(t- iA)^{\frac{1}{3}}} + 
		\OO(t- iA)^{\frac{1}{3}}, 
	\end{split}
\end{equation*}
with $\al_{+} \in \complexs$ such that $\al_{+}^3=\frac{1}{2}$. 

An analogous result holds for $\vabs{t+iA}<\upsilon$, $\arg(t+iA) \in(-\frac{\pi}{2},\frac{3\pi}{2})$ and $\al_- = \conj{\al_+}$.
\item  $\La_h(t)$ has only one zero in $\ol{\Pi^{\mathrm{ext}}_{A, \betaBow}}$ at $t=0$.
\end{itemize}
\end{theorem}

\subsection{The perturbed invariant manifolds}
\label{section:outer}

In this section, following the approach described in~\cite{BFGS12, BCS13, 
GMS16}, we study the analytic continuation of the parametrizations of the 
perturbed one-dimensional stable and unstable manifolds, $\WW^{\unstable}(\de)$ 
and $\WW^{\stable}(\de)$. 
%
%

Since we  measure the distance between the invariant manifolds in the 
section $\la=\la_*$ (see Theorem \ref{theorem:mainTheoremPoincare}), 
we parameterize them as graphs with respect to $\la$ (whenever is possible) or, 
more conveniently, 
%
%
%
with respect 
to the independent variable $u$ defined by $\la=\la_h(u)$.

To define these suitable parameterizations we first translate the 
equilibrium point $\Ltres(\de)$ to $\mathbf{0}$ by the change of coordinates
\begin{equation}\label{def:changeEqui}
	\phi_{\equi}:(\la,\La,x,y) \mapsto (\la,\La,x,y) + \Ltres(\de).
\end{equation}
Second, we consider the symplectic change of coordinates 
\begin{equation}\label{def:changeOuter}
	\phi_{\out}:(u,w,x,y) \to (\la,\La,x,y),
	\quad 
	{\la}= \la_h(u), \hh  {\La}= \La_h(u) - \frac{w}{3\La_h(u)}.
\end{equation}
We refer to $(u,w,x,y)$ as the \emph{separatrix coordinates}.

Let us remark that $\phi_{\out}$ is not defined for $u=0$ since $\La_h(0)=0$ (see Theorem \ref{theorem:singularities}).
We deal with this fact later when considering the domain of definition for $u$.

%
%
%
%
 
After these changes of variables, we look for the perturbed invariant manifolds 
as a graph with respect to $u$.
%
In other words, we look for functions
\[
\zdOut(u) = \left(\wdOut(u),\xdOut(u),\ydOut(u)\right)^T,
\quad \text{ for } \diamond=\unstable,\stable,
\]
such that the 
invariant manifolds given in Proposition~\ref{proposition:HamiltonianScaling} 
can be expressed as
\begin{equation}\label{eq:invariantManifoldsExpression}
\WW^{\diamond}(\de)= \left\{ \paren{\la_h(u), \La_h(u)-\frac{\wdOut(u)}{3\La_h(u)}, \xdOut(u), \ydOut(u)} + \Ltres(\de)\right\}, \quad \text{for } \diamond=\unstable,\stable,
\end{equation}
with $u$ belonging to an appropriate domain contained in $\Pi^{\mathrm{ext}}_{A, \betaBow}$ (see \eqref{def:dominiBow}).
%
%
The graphs $\zuOut$ and $\zsOut$
must satisfy the asymptotic conditions 
\begin{equation}\label{eq:asymptoticConditionsOuter}
	\begin{split}
		\lim_{\Re u \to -\infty} \left(\frac{\wuOut(u)}{\La_h(u)},\xuOut(u), \yuOut(u) \right) = 
		\lim_{\Re u \to +\infty} \left(\frac{\wsOut(u)}{\La_h(u)},\xsOut(u), \ysOut(u) \right) = 0.
	\end{split}
\end{equation}

\begin{remark}\label{remark:realAnalytic}
Since the Hamiltonian $H$ is real-analytic in the sense of $\conj{H(\la,\La,x,y;\de)}=H(\conj{\la},\conj{\La},y,x;\conj{\de})$ (see Proposition \ref{proposition:HamiltonianScaling}), 
then we say that $\zOut(u)=(\wOut(u),\xOut(u),\yOut(u))^T$ is real-analytic if it satisfies
\begin{align*}
	\wOut(\conj{u}) = \conj{\wOut(u)}, \qquad
	\xOut(\conj{u}) = {\yOut(u)}, \qquad
	\yOut(\conj{u}) = {\xOut(u)}.
\end{align*}
\end{remark}

The classical way to study exponentially small  splitting of separatrices, in this setting, is to look for solutions $\zuOut$ and $\zsOut$ in a certain complex common domain containing a segment of the real line and intersecting a $\OO(\de^2)$ neighborhood of the singularities $u=\pm iA$ of the separatrix.

Recall that the invariant manifolds can not be expressed as a graph in a neighborhood of $u=0$.
To overcome this technical problem, we find solutions $\zuOut$ and $\zsOut$ defined in a complex domain, which we call \emph{boomerang domain} due to its shape (see Figure~\ref{fig:dominiBoomerang}).
\begin{figure}[t] 
	\centering
	\begin{overpic}[scale=1]{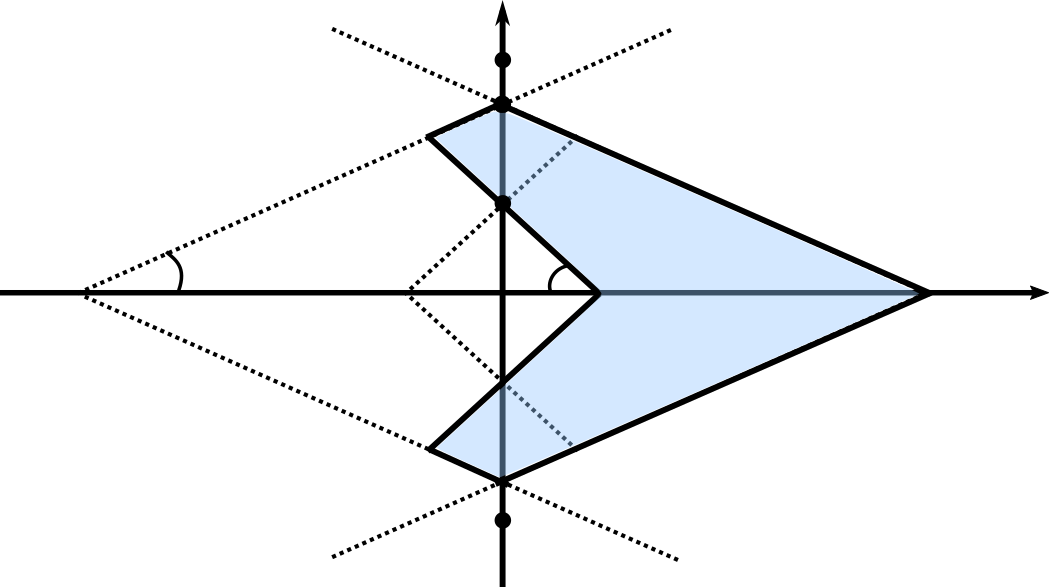}
		\put(60,31){$\DBoomerang$}
		\put(19,29.5){\footnotesize $\betaOutA$}
		\put(48.5,29.5){\footnotesize $\betaOutB$}
		\put(42,49.5){\footnotesize $iA$}
		\put(52,44.5){\footnotesize $i(A-\kappa \de^2)$}
		\put(39.5,4){\footnotesize $-iA$}
		\put(41,35){\footnotesize $\dBoomerang A$}
		\put(102,27){\footnotesize$\Re u$}
		\put(45,58){\footnotesize$\Im u$}
	\end{overpic}
	\bigskip
	\caption{The boomerang domain $\DBoomerang$ defined in~\eqref{def:dominiBoomerang}.}
	\label{fig:dominiBoomerang}
\end{figure}
Namely,
\begin{equation}\label{def:dominiBoomerang}
\begin{split}
\DBoomerang = \left\{ u \in \complexs \right. \st 
&\vabs{\Im u} < A - \kappa \de^2 + \tan \betaOutA \Re u,
\vabs{\Im u} < A - \kappa \de^2 - \tan \betaOutA \Re u,
\\
&\vabs{\Im u} > \left. \dBoomerang A - \tan \betaOutB \Re u 
\right\},
\end{split}
\end{equation}
where $\kappa>0$ is such that $A-\kappa \de^2>0$,
$\betaOutA$ is the constant given in Theorem~\ref{theorem:singularities}
and $\betaOutB \in [\betaOutA,\frac{\pi}{2})$ and $\dBoomerang \in 
(\frac{1}{4},\frac{1}{2})$
are independent of $\de$.

%

\begin{theorem}\label{theorem:existence}
%
Fix a constant $\dBoomerang \in (\frac{1}{4},\frac{1}{2})$.
Then, there exists $\de_0, \kappaBoomerang>0$ such that,
for $\de \in (0,\de_0)$, 
$\kappa\geq\kappaBoomerang$,
the graph parameterizations $\zuOut$ and $\zsOut$ introduced 
in~\eqref{eq:invariantManifoldsExpression} can be extended real-analytically to 
the domain $\DBoomerang$.

Moreover, there exists a real constant $\cttOuterA>0$ independent of $\de$ and $\kappa$ such that, for $u \in \DBoomerang$ we have that
\begin{align*}
|\wdOut(u)| \leq \frac{\cttOuterA\de^2}
{\vabs{u^2 + A^2}}
+ \frac{\cttOuterA\de^4}{\vabs{u^2 + A^2}^{\frac{8}{3}}}, \quad
|\xdOut(u)| \leq \frac{\cttOuterA\de^3}
{\vabs{u^2 + A^2}^{\frac{4}{3}}}, \quad
|\ydOut(u)| \leq \frac{\cttOuterA\de^3}
{\vabs{u^2 + A^2}^{\frac{4}{3}}}.
\end{align*}
\end{theorem}


Notice that the asymptotic conditions \eqref{eq:asymptoticConditionsOuter} do 
not have any meaning in the domain $\DBoomerang$ since it is bounded.
Therefore, to prove the existence of $\zuOut$ and $\zsOut$ in $\DBoomerang$ one has to start with different domains where these asymptotic conditions make sense and then find a way to extend  them real-analytically to $\DBoomerang$.
We describe the details of these process in the  following Sections \ref{subsection:outerBasic} and \ref{subsection:outerExtension}.

\subsubsection{Analytic extension of the stable and unstable manifolds}
\label{subsection:outerBasic}

The Hamiltonian $H$ written in separatrix coordinates 
(see~\eqref{def:changeEqui} and \eqref{def:changeOuter}) becomes
\begin{equation}\label{def:hamiltonianOuter}
	H^{\out} =  H^{\out}_0 +  H^{\out}_1,
\end{equation}
with 
\begin{equation}\label{def:hamiltonianOuterSplit}
	\begin{split}
		H^{\out}_0 = 	w  + \frac{xy}{\de^2}, \qquad
		H^{\out}_1 =
		{H} \circ \paren{\phi_{\equi} \circ \phi_{\out}} -H^{\out}_0.
	\end{split}
\end{equation}
%
%
%
Introducing the notation $z=(w,x,y)^T$ and defining
\begin{equation}\label{def:matrixAAAOuter}
	\AAA^{\out}= \frac{i}{\de^2} \begin{pmatrix}
		0 & 0 & 0 \\
		0 & 1 & 0 \\
		0 & 0 & -1
	\end{pmatrix},
\end{equation}
the equations associated to the Hamiltonian $H^{\out}$ can be written as
\begin{align}\label{eq:systemEDOsOuter}
	\left\{ \begin{array}{l}
		\dot{u} = 1 + g^{\out}(u,z),\\
		\dot{z} = \AAA^{\out} z + f^{\out}(u,z),
	\end{array} \right.
\end{align}
where $g^{\out} = \partial_w H_1^{\out}$  and $f^{\out}=\paren{-\partial_u 
H_1^{\out}, i\partial_y H_1^{\out}, -i\partial_x H_1^{\out} }^T$. Consequently, 
the parameterizations~$\zuOut(u)$ and $\zsOut(u)$ given in 
\eqref{eq:invariantManifoldsExpression} satisfy the invariance equation
\begin{equation}\label{eq:invariantEquationOuter}
	\partial_u \zdOut = 
	\AAA^{\out} \zdOut + \RRR^{\out}[\zdOut], \hh
	\text{ for } \diamond=\unstable,\stable,
\end{equation}
with
\begin{equation}\label{def:operatorRRROuter}
	\RRR^{\out}[\varphi](u)= 
	\frac{f^{\out}(u,\varphi)- g^{\out}(u,\varphi) \AAA^{\out} \varphi }{1+g^{\out}(u,\varphi)}.
\end{equation}

%
\begin{remark}\label{remark:derivadaH1w}
Note that one can use this invariance equation whenever 
\[
 1+g^{\out}(u,\varphi)=1+ \partial_w H_1^{\out}(u,\varphi)\neq 0
\]
This condition is satisfied in the different domains that are considered in this section and in the forthcoming ones and it is checked in Appendix \ref{appendix:proofH-technical}  (see \eqref{proof:boundsInftyg} and  \eqref{def:Fitag}). 
This fact is also used later in Section \ref{section:differenceInner}.
\end{remark}

The first step is to look for solutions of this 
equation in the  domains
\begin{equation}\label{def:dominisInfty}
\begin{split}
\DuInfty = \claus{u \in \complexs \st \Re u < -\rhoInfty}, \qquad
\DsInfty = \claus{u \in \complexs \st \Re u > \rhoInfty},
\end{split}
\end{equation}
for some $\rhoInfty>0$, which allows us to take into account the asymptotic 
conditions \eqref{eq:asymptoticConditionsOuter}. 
%
%
%

\begin{proposition}\label{proposition:existenceInfty}
Fix $\rhoInfty>0$.
Then, there exists $\de_0>0$ such that, for $\de \in (0,\de_0)$,
the equation~\eqref{eq:invariantEquationOuter} has a unique real-analytic 
solution $\zdOut =(\wdOut,\xdOut,\ydOut)^T$ in  $\DdInfty$ 
(for $\diamond=\unstable, \stable$)
satisfying the corresponding asymptotic condition \eqref{eq:asymptoticConditionsOuter}.

Moreover, there exists  $\cttOuterB>0$ independent of $\de$ such that, for $u 
\in \DdInfty$, 
\begin{align*}
	|\wdOut(u) e^{-2\vap u}| \leq \cttOuterB\de^2, 
	\qquad
	|\xdOut(u) e^{-\vap u}| \leq \cttOuterB\de^3, 
	\qquad
	|\ydOut(u) e^{-\vap u}| \leq \cttOuterB\de^3.
\end{align*}
with $\nu=\sqrt{\frac{21}{8}}$ for $\diamond=\unstable$ and
$\nu=-\sqrt{\frac{21}{8}}$ for $\diamond=\stable$.
\end{proposition}	

This proposition is proved in Section \ref{subsection:proofH-existenceInfinite}.

To extend analytically the invariant manifolds to reach the boomerang domain $\DBoomerang$ we have to face the problem that these parameterizations become undefined at $u=0$. 
To overcome it, first we extend the solutions $\zuOut$ and $\zsOut$ of 
Proposition~\ref{proposition:existenceInfty} to the \emph{outer domains} (see 
Figure~\ref{fig:dominisComecocos})
\begin{equation}\label{def:dominisComecocos}
\begin{split}
\DuOut = \left\{ u \in \complexs \right. \st  
&\vabs{\Im u} < A - \kappa \de^2 - \tan \betaOutA \Re u,
\\
&\vabs{\Im u} > \dOuter A + \tan \betaOutB \Re u  , \left. \,
{\Re u} > - \rhoOuter \right\}, \\
\DsOut = \big\{u \in \complexs \st & 
-u \in \DuOut\big\},
\end{split}
\end{equation}
where
$d_1 \in (\frac{1}{4},\frac{1}{2})$ and $\rhoOuter>\rhoInfty$
are fixed independent of $\de$,
and  $\kappa>0$ is such that $A-\kappa \de^2>0$.
%

\begin{figure}
\centering
\begin{overpic}[scale=0.7]{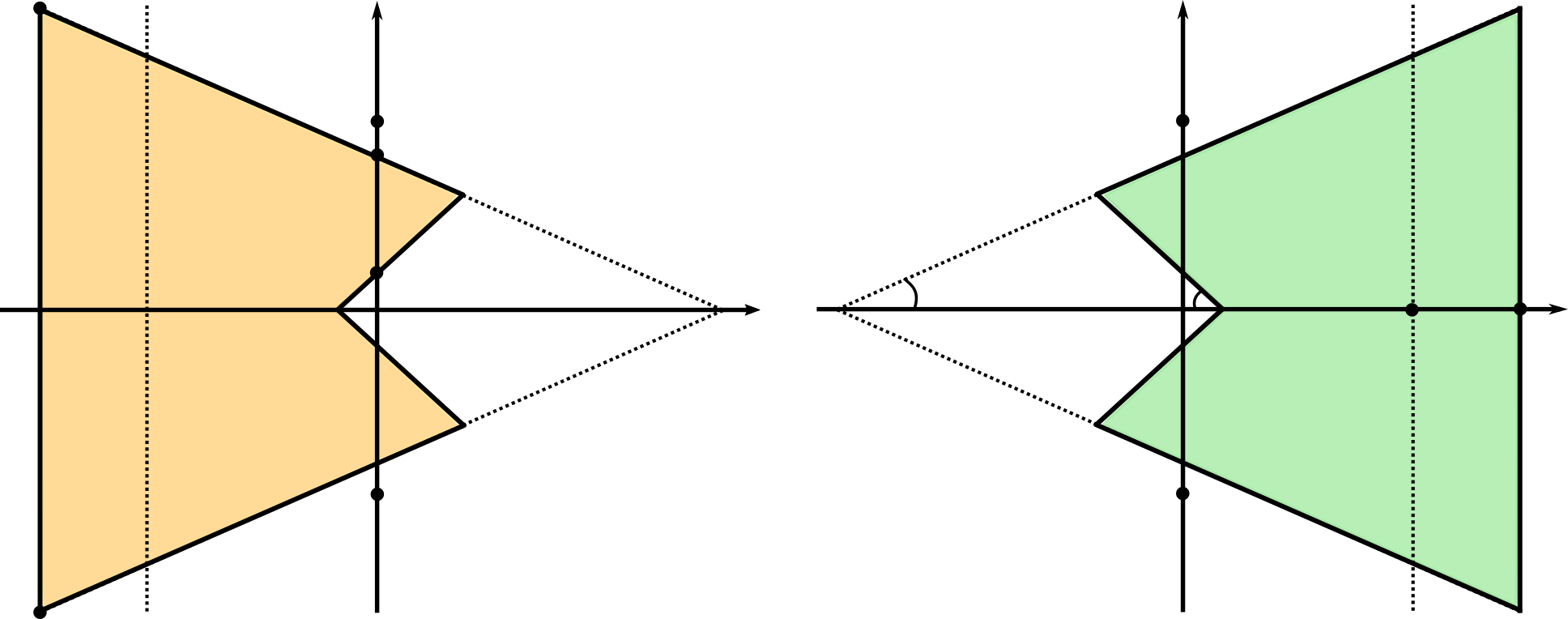}
	\put(10,24){$\DuOut$}
	\put(77,24){$\DsOut$}
	\put(3,-2){\footnotesize $\DuInftyRho$}
	\put(90,-2){\footnotesize $\DsInftyRho$}
	\put(87,18){\footnotesize $\rhoInfty$}
	\put(94,18){\footnotesize $\rhoOuter$}
	\put(59,20.3){\footnotesize $\betaOutA$}
	\put(72.5,20.3){\footnotesize $\betaOutB$}
	\put(72,31){\footnotesize $iA$}
	\put(25,29.2){\footnotesize $iA - \kappa \de^2$}
	\put(25,21){\footnotesize $d_1 A$}
	\put(69,7){\footnotesize $-iA$}
	\put(-1.5,0){\footnotesize $\conj{u_1}$}
	\put(-1.5,39){\footnotesize $u_1$}
	\put(100,19){\footnotesize $\Re u$}
	\put(21,40){\footnotesize $\Im u$}
	\put(73,40){\footnotesize $\Im u$}
\end{overpic}
\bigskip
\caption{The outer domains $\DuOut$ and $\DsOut$ defined in~\eqref{def:dominisComecocos}} 
\label{fig:dominisComecocos}
\end{figure}

\begin{proposition}\label{proposition:existenceComecocos}
Consider the functions $\zuOut$, $\zsOut$ and the constant $\rhoInfty>0$ 
obtained in Proposition~\ref{proposition:existenceInfty}.
Fix constants $\rhoOuter>\rhoInfty$ and $\dOuter \in (\frac{1}{4},\frac{1}{2})$.
%
%
Then, there exist $\de_0,\kappaOuter>0$
such that, for $\de \in (0,\de_0)$,
$\kappa\geq\kappaOuter$,
the functions $\zdOut=(\wdOut,\xdOut,\ydOut)^T$, 
$\diamond=\unstable, \stable$, can be extended analytically to 
the domain $\DdOut$

Moreover, there exists  $\cttOuterC>0$ independent of $\de$ and $\kappa$ such 
that, for $u \in \DdOut$,
\begin{align*}
|\wdOut(u)| \leq \frac{\cttOuterC\de^2}
{\vabs{u^2 + A^2}}
+ \frac{\cttOuterC\de^4}{\vabs{u^2 + A^2}^{\frac{8}{3}}}, \quad
|\xdOut(u)| \leq \frac{\cttOuterC\de^3}
{\vabs{u^2 + A^2}^{\frac{4}{3}}}, \quad
|\ydOut(u)| \leq \frac{\cttOuterC\de^3}
{\vabs{u^2 + A^2}^{\frac{4}{3}}}.
\end{align*}
\end{proposition}	

This proposition is proved in Section \ref{subsection:proofH-existenceBounded}.

Notice that taking $\rhoOuter$ big enough, $\dOuter\leq \dBoomerang$ and $\kappaOuter\leq\kappaBoomerang$ we have $\DBoomerangKappa \subset \DsOutKappa$.
Therefore, for the stable manifold $\zsOut$,  Proposition~\ref{proposition:existenceComecocos} implies Theorem~\ref{theorem:existence}.
However, we still need to extend further $\zuOut$ in order to reach $\DBoomerangKappa$.

%
%

%
%

\subsubsection{Further analytic extension of the unstable manifold}
\label{subsection:outerExtension}

Since by Proposition~\ref{proposition:existenceComecocos} the unstable solution $\zuOut$ is defined in $\DuOutKappa$,
To prove Theorem \ref{theorem:existence}
it only remains to extend it to the points in the boomerang domain $\DBoomerangKappa$ which do not belong to the outer unstable domain. 
Namely, we extend $\zuOut$ to
\begin{equation}\label{def:domainBoomerangTilde}
\begin{split}
\DBoomerangTilde = \left\{ u \in \complexs \right. \st 
&\vabs{\Im u} < A - \kappa \de^2 - \tan \betaOutA \Re u,
\\
&\vabs{\Im u} < \dBoomerang A + \tan \betaOutB \Re u, \, \,
\vabs{\Im u} > \left. \dBoomerang A - \tan \betaOutB \Re u 
\right\},
\end{split}
\end{equation}
for suitable $\kappa$ and $d$ (see Figure~\ref{fig:dominiBoomerangTilde}).
Notice that $\DBoomerangTilde \subset \DBoomerang$ and that $\DBoomerangTilde$ only contains points at distance of $u=\pm iA$ of order $1$ with respect to $\de$.
\begin{figure} 
\centering
\begin{overpic}[scale=1]{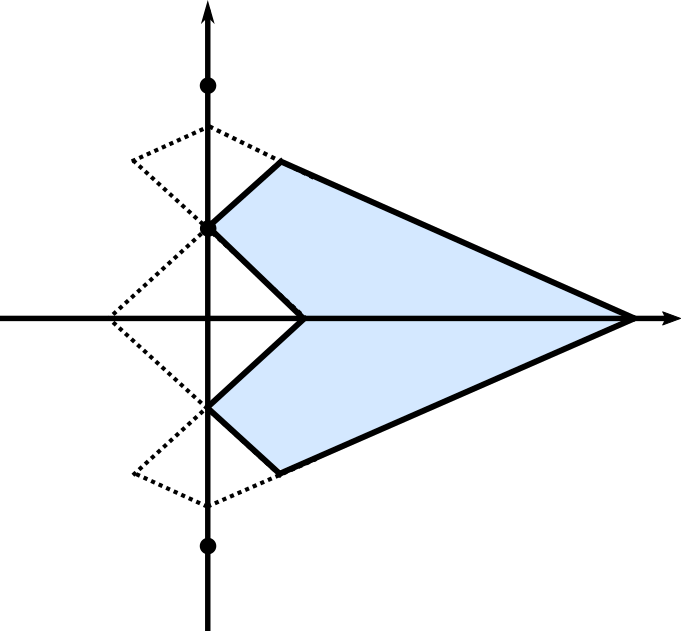}
	\put(50,50){$\DBoomerangTilde$}
	\put(12,30){$\DBoomerang$}
	\put(33,79){\footnotesize $iA$}
	\put(33,10){\footnotesize $-iA$}
	\put(19,57.5){\footnotesize $\dBoomerang A$}
	\put(101,44.5){\footnotesize$\Re u$}
	\put(27,93.5){\footnotesize$\Im u$}
\end{overpic}
\bigskip
\caption{The domain $\DBoomerangTilde$ defined in~\eqref{def:domainBoomerangTilde}.}
\label{fig:dominiBoomerangTilde}
\end{figure}

As we have mentioned, to measure  the difference between the invariant manifolds $\WW^{\unstable}(\de)$ and $\WW^{\stable}(\de)$ it is  convenient to parameterize them as graphs (see \eqref{eq:invariantManifoldsExpression}). 
However, these graph parametrizations are not defined at $u=0$.
Moreover, since all the fixed point arguments that we apply to obtain the graph parameterizations rely on complex path integration, we are not able to extend them to domains which are not simply connected.
Therefore, to reach $\DBoomerangTilde$ from $\DuOut$, we need to switch to a different parametrization that is well defined at $u=0$.

The auxiliary parametrization we consider is the classical time-parametrization which is associated to the Hamiltonian $H$ in \eqref{def:hamiltonianScaling}.
(Recall that the graph parametrization $\zuOut$ was associated to the Hamiltonian $H^{\out} = H \circ \phi_{\equi} \circ \phi_{\out}$).


This analytic extension procedure 
has three steps:
\begin{enumerate}
\item We consider the \emph{outer transition domain} (see Figure \ref{fig:dominisTransition})
\begin{equation}\label{def:DOuterTilde}
\begin{split}	
	\DOuterTilde = \left\{ v \in \complexs \right. \st  
	&\vabs{\Im v} < A - \kappaOuterTilde \de^2 - \tan \betaOutA \Re v,
	\\
	&\vabs{\Im v} > \dOuterA A + \tan \betaOutB \Re v, 
	\\ &\left.
	\vabs{\Im v} < \dOuterB A + \tan \betaOutB \Re v
	\right\},
\end{split}
\end{equation}
where
$\dOuter < \dOuterA < \dOuterB < \frac12$
are independent of $\de$
and  $\kappaOuterTilde>\kappaOuter$ is such that $A-\kappaOuterTilde \de^2>0$.
Notice that $\DOuterTilde \subset \DuOutKappa$.

Since $\dot{u}=1+o(1)$ (see \eqref{eq:systemEDOsOuter}),
we look for a real-analytic and close to the identity change of coordinates $u=v + \uOut(v)$ defined in $\DOuterTilde$ such that the time-parametrization
\begin{equation}\label{eq:equationuOutA}
\Gu(v) = \phi_{\equi} \circ \phi_{\out}(v+\uOut(v),\zuOut(v+\uOut(v)))
\end{equation}	
is a solution of  the Hamiltonian $H$ in \eqref{def:hamiltonianScaling}.
That is, $\dot{v}=1$ and $\Gu(v)\in\WW^{\unstable}(\de)$ for $v\in \DOuterTilde$.
See the details in Proposition \ref{proposition:changeuOut} and Corollary \ref{corollary:changeuOut} below.
%
%
\item We  extend analytically the time-parametrization $\Gu(v)$ to reach the domain $\DBoomerangTilde$.
In particular, we extend $\Gu$ to the \emph{flow domain}
\begin{equation}\label{def:dominiFlow}
\begin{split}
	\DFlow = \left\{ v \in \complexs \st
	\right. &
	\vabs{\Im v} < A - \kappa \de^2 - \tan \betaOutA \Re v, 
	\\ & \left.
	\vabs{\Im v} < \dFlow A + \tan \beta_1 \Re v
	\right\},
\end{split}
\end{equation}
where
$\dFlow \in (\dOuterA,\dOuterB)$
is independent of $\de$
and  $\kappaFlow>\kappaOuterTilde$ is such that $A-\kappaFlow \de^2>0$.
Notice that, 
\[
\DOuterTilde \cap \DFlow\neq\emptyset,
\qquad \text{and} \qquad
\DBoomerangTildeProp \subset \DFlow, 
\]
for $\dBoomerangTilde\in(\dOuter,\dFlow)$ and 
$\kappaBoomerangTilde>\kappaFlow$.
See the details in Proposition \ref{proposition:existenceFlow}.

\item We prove that there exists a real-analytic close to the identity change of variables of the form $v=u+\vOut(u)$,  $u\in\DBoomerangTildeProp$, such that the function $\zuOut(u)$ defined by
\begin{equation}\label{eq:equationvOut}
	(u,\zuOut(u)) =
	(\phi_{\equi} \circ \phi_{\out})^{-1} \Big( \Gu(u + \vOut(u))\Big) 	
\end{equation}
gives an invariant graph of $H^{\out}$ in \eqref{def:hamiltonianOuter}.
See the details in Proposition \ref{proposition:changevOut} and Corollary \ref{corollary:changevOut} below.
\end{enumerate}

As a consequence, we have extended analytically $\zuOut$ to $\DBoomerangTildeProp$.

\begin{figure}[t] 
	\centering
	\begin{overpic}[scale=1]{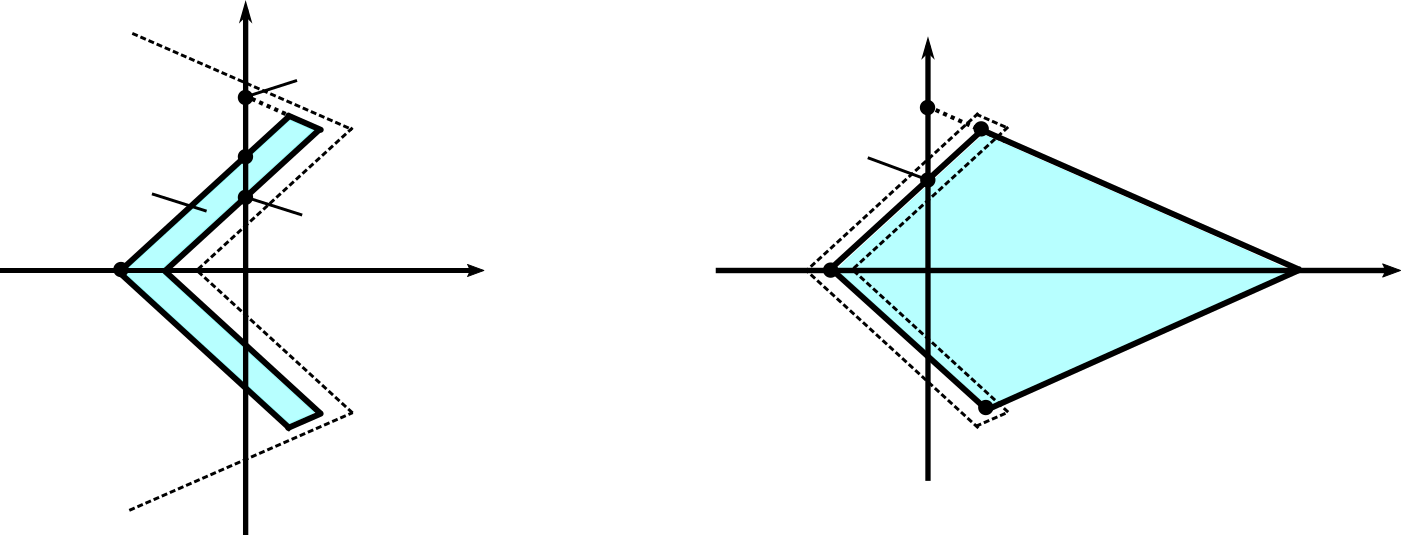}
		\put(-1,25){$\DOuterTilde$}
		\put(5,6){\footnotesize $\DuOut$}
		\put(71,21){$\DFlow$}
		\put(55,9){\footnotesize $\DOuterTilde$}
		\put(22,22){\footnotesize $\dOuterA A$}
		\put(12,27){\footnotesize $\dOuterB A$}
		\put(21.5,32){\footnotesize $iA-\kappaOuterTilde\de^2$}
		\put(6,17){\footnotesize $\rhoOuterTilde$}
		\put(57,26){\footnotesize $\dFlow A$}
		\put(54,30){\footnotesize $iA-\kappaFlow\de^2$}
		\put(71,30){\footnotesize $v_1$}
		\put(71,6){\footnotesize $\conj{v_1}$}
		\put(56,17){\footnotesize $v_0$}
		\put(35,18){\footnotesize $\Re u$}
		\put(101,18){\footnotesize $\Re v$}
		\put(15,39){\footnotesize $\Im u$}
		\put(64,37){\footnotesize $\Im v$}
	\end{overpic}
	\bigskip
	\caption{The domain $\DOuterTilde$ given in \eqref{def:DOuterTilde} (left) and $\DFlow$ in \eqref{def:dominiFlow} (right).}
	\label{fig:dominisTransition}
\end{figure}

%
%
%

For the first step, we look for a function $\uOut$ such that $\paren{v+\uOut(v), \zuOut(v+\uOut(v))}$ is a solution of the differential equations given by the Hamiltonian $H^{\out}$ in \eqref{def:hamiltonianOuter}.
Therefore, $\uOut$ satisfies 
\begin{equation}\label{eq:equationuOutB}
	\partial_v \, {\uOut}(v) = \partial_w H_1^{\out}
	\paren{v + \uOut(v), \zuOut(v+\uOut(v))}.
\end{equation}
The next proposition ensures that $\uOut$ exists and it is well defined for $v\in\DOuterTilde$.

\begin{proposition}\label{proposition:changeuOut}
Let the function $\zuOut$ and the constants $\rhoOuter$, $\dOuter$ and $\kappaOuter$ be as obtained in Proposition~\ref{proposition:existenceComecocos}
and consider constants $\dOuterA,\dOuterB \in(\dOuter,\frac{1}{2})$ such that $\dOuterA<\dOuterB$ and 
$\kappaOuterTilde>\kappaOuter$. 
Then, there exists $\de_0$ such that, for $\de \in (0,\de_0)$,
the equation \eqref{eq:equationuOutB} has  a real-analytic solution $\uOut: \DOuterTilde \to  \complexs$.

Moreover, for some  constant $\cttOuterD>0$ independent of $\de$ and  for $v \in \DOuterTilde$, $\uOut$ satisfies
\[
\vabs{\uOut(v)} \leq \cttOuterD \de^{2}
\qquad \text{and} \qquad
v + \uOut(v) \in \DuOutKappa.
\]
\end{proposition}

This proposition is proved in Section \ref{subsection:proofH-changeuOut}.
%
%
%
%
Together with Proposition \ref{proposition:changeuOut} implies the following corollary.

\begin{corollary}\label{corollary:changeuOut}
Under the hypothesis of Proposition \ref{proposition:changeuOut}, 
%
there exists $\de_0>0$ such that, for $\de \in (0,\de_0)$,  
the function $\Gu$ in \eqref{eq:equationuOutA} is well defined and real-analytic in $\DOuterTilde$.	
\end{corollary}

On the following, we use without mention that $\Gu(v)$ can be split as
\begin{align}\label{eq:splitParametrizationGu}
	\Gu(v) = \G_h(v) + \Gh(v),
	\quad \text{with} \quad
	\left\{
	\begin{aligned}
		\G_h &  = (\la_h,\La_h,0,0)^T, \\
		\Gh  & = (\lah, \Lah, \xh,\yh)^T.
	\end{aligned} 
	\right.
\end{align}
The next proposition extends the parametrization $\Gu$ to the domain $\DFlow$ (see \eqref{def:dominiFlow}).

\begin{proposition}\label{proposition:existenceFlow}
Let the function $\Gu$ and the constants
$\dOuterA, \dOuterB$ and $\kappaOuterTilde$ be as obtained in Corollary~\ref{corollary:changeuOut} and Proposition \ref{proposition:changeuOut}
and fix
$\dFlow \in (\dOuterA,\dOuterB)$ and $\kappaFlow>\kappaOuterTilde$.
Then, there exists $\de_0>0$ such that,
for $\de \in (0,\de_0)$, $\Gu$ 
can be real-analytically extended to $\DFlow$. 

Moreover, there exists a constant $\cttOuterF>0$ independent of $\de$ such that, for $v \in \DFlow$,
\begin{align*}
	\vabsSmall{\lah(v)} \leq \cttOuterF \de^2, \qquad
	\vabsSmall{\Lah(v)} \leq \cttOuterF \de^2, \qquad
	\vabsSmall{\xh(v)} \leq \cttOuterF \de^3, \qquad
	\vabsSmall{\yh(v)} \leq \cttOuterF \de^3.
\end{align*}	
\end{proposition}

This proposition is proved in Section \ref{subsection:proofH-existenceFlow}.


For the third step, we ``go back'' to the graph parametrization $\zuOut(u)$ by looking for a change $v=u+\vOut(u)$ for $u \in \DBoomerangTilde$.
Notice that, in order to satisfy equation \eqref{eq:equationvOut} and recalling \eqref{def:pointL3sca}, $\vOut$ must be a solution of
\begin{equation}\label{eq:equationsvOutA}
	\lah(u+\vOut(u)) = 
	\la_h(u)-\la_h(u+\vOut(u)).
\end{equation}
%
%
Then, one can easily recover the graph parametrization $(\wuOut(u),\xuOut(u),\yuOut(u))$ using the equations
\begin{align}
\begin{aligned}\label{eq:equationsvOutB}
\La_h(u) 
-\La_h(u+\vOut(u))
- \frac{\wuOut(u)}{3\La_h(u)} + \de^2 \LtresLa(\de) 
&=
\Lah(u+\vOut(u)), \\
\xuOut(u) + \de^3 \Ltresx(\de) &=\xh(u+\vOut(u)), \\
\yuOut(u) + \de^3 \Ltresy(\de)
&=	\yh(u+\vOut(u)).
\end{aligned}
\end{align}
%
%

The next proposition ensures that $\vOut$ exists and it is well defined in $\DBoomerangTilde$ (see \eqref{def:domainBoomerangTilde}).

\begin{proposition}\label{proposition:changevOut}
Let the function $\Gu$ and the constants
$\dFlow$ and $\kappaFlow$ be as obtained in Proposition~\ref{proposition:existenceFlow}
and the constant $\dOuter$ as obtained in
Proposition~\ref{proposition:existenceComecocos}.
Let us consider constants
$\dBoomerangTilde\in(\dOuter,\dFlow)$ and 
$\kappaBoomerangTilde>\kappaFlow$.
Then, there exists $\de_0>0$ such that,
for $\de \in (0,\de_0)$, equation \eqref{eq:equationsvOutA} has a real-analytic solution $\vOut: \DBoomerangTildeProp \to \complexs$ 
satisfying 
\[
\vabs{\vOut(u)} \leq \cttOuterG \de^{2}
\qquad \text{and} \qquad
u + \vOut(u) \in \DFlow.
\]
for some  constant $\cttOuterG>0$ independent of $\de$  and  $u \in \DBoomerangTildeProp$.
\end{proposition}

Proposition \ref{proposition:changevOut} is proved in Section \ref{subsection:proofH-changevOut}. Summarizing all the previous results we obtain the following result.

\begin{corollary}\label{corollary:changevOut}
Let the function $\vOut$ and the constants $\dBoomerangTilde$ and $\kappaBoomerangTilde$ be as obtained in Proposition \ref{proposition:changevOut}. 
Then, there exists $\de_0>0$ such that,
for $\de \in (0,\de_0)$, 
equation \eqref{eq:equationsvOutB} has a unique solution $\zuOut=(\wuOut,\xuOut,\yuOut)^T: \DBoomerangTildeProp \to \complexs^3$.

Moreover, there exists a constant $\cttOuterH>0$ independent of $\de$ such that, for $u \in \DBoomerangTildeProp$,
\begin{align*}
	\vabsSmall{\wuOut(u)} \leq \cttOuterH \de^2, \qquad
	\vabsSmall{\xuOut(u)} \leq \cttOuterH \de^3, \qquad
	\vabsSmall{\yuOut(u)} \leq \cttOuterH \de^3.
\end{align*}
\end{corollary}

To finish this section, notice that, taking $\rhoOuter$ big enough,
$\dBoomerang \geq \dBoomerangTilde$ and $\kappaBoomerang\geq\kappaBoomerangTilde$
we have that
\[
\DBoomerangKappa \subset \DuOutKappa \cup \DBoomerangTildeProp,
\qquad \text{with} \qquad
\DuOutKappa \cap \DBoomerangTildeProp \neq \emptyset,
\] 
%
and then, Corollary
\ref{corollary:changevOut} and Proposition \ref{proposition:existenceComecocos} 
imply the statements of  Theorem~\ref{theorem:existence} referring to  the unstable manifold $\zuOut$.

\subsection{A first order of the invariant manifolds near the singularities}
\label{section:differenceInner}

Let us consider the difference 
\begin{equation*}
	\dzOut = (\dwOut,\dxOut,\dyOut)^T = \zuOut - \zsOut,
\end{equation*}
where $\zuOut$ and $\zsOut$ are the perturbed invariant graphs given in Theorem \ref{theorem:existence}.
%
%
Since $\zuOut$ and $\zsOut$ satisfy the invariance equation \eqref{eq:invariantEquationOuter}, the difference $\dzOut$ satisfies the linear equation
\begin{equation}\label{eq:invariantEquationDifference1}
	\partial_u \dzOut(u)  = 
	\AAA^{\out} \dzOut(u) + 
	\wt{\BB}^{\spl}(u) \dzOut(u),
\end{equation}
where $\AAA^{\out}$ is as given in \eqref{def:matrixAAAOuter} and
\begin{equation}\label{def:Bspl1}
	\wt{\BB}^{\spl}(u) = \int_0^1 
	D_z\RRR^{\out}[\sigma\zuOut + (1-\sigma) \zsOut](u) d \sigma.
\end{equation}
Since  $\zuOut$ and $\zsOut$ are already defined in $\DBoomerang$,  $\wt{\BB}^{\spl}(u)$ can be considered as a  ``known'' function.

In addition, since the graphs of $\zuOut$ and $\zsOut$ belong to the same energy level of $H^{\out}$ (see \eqref{def:hamiltonianOuter}), we have that
\begin{align*}
	H^{\out}(u,\zuOut(u);\de) - H^{\out}(u,\zsOut(u);\de) =0,
	\quad \text{for } u \in \DBoomerang.
\end{align*}
Therefore, we can reduce~\eqref{eq:invariantEquationDifference1} to a two dimensional equation.
Indeed, defining $\Ups = 
\paren{\Ups_1,\Ups_2,\Ups_3}$ such that
\begin{align}\label{def:operatorPPdifference}
	\Ups(u) = 
	\int_0^1 D_z H^{\out} 
	\paren{u, \sigma \zuOut(u) + (1-\sigma) \zsOut(u)} d \sigma,
\end{align}
and applying the mean value theorem we have that
\begin{align*}
	\Ups_1(u) \dwOut(u) +
	\Ups_2(u) \dxOut(u) +
	\Ups_3(u) \dyOut(u) = 0.
\end{align*}
Notice that $\Upsilon_1(u)=1+\int_0^1 \partial_w H_1^{\out} \paren{u, \sigma \zuOut(u) + (1-\sigma) \zsOut(u)} d\sigma$ and therefore $\Upsilon_1(u)\neq 0$ for $u \in \DBoomerang$ (see Remark \ref{remark:derivadaH1w}).
%
%
Therefore, writing
\begin{equation}\label{eq:defDwOut}
	\dwOut(u) = 
	-\frac{\Ups_2(u)}{\Ups_1(u)} \dxOut(u) 
	-\frac{\Ups_3(u)}{\Ups_1(u)} \dyOut(u)
\end{equation}
and defining $\dzHat=(\dxOut,\dyOut)^T$, the last two components of \eqref{eq:invariantEquationDifference1} are equivalent to
\begin{equation}\label{eq:invariantEquationDifference2}
	\partial_u {\dzHat}(u) = 
	{\AAA}^{\spl}(u) \dzHat(u)
	+ 
	{\BB}^{\spl}(u) \dzHat(u),
\end{equation} 
where
\begin{align} \label{def:operatorsDifferenceAABB}
	&\AAA^{\spl}=
	\begin{pmatrix}
		\frac{i}{\de^2} + \wt{\BB}^{\spl}_{2,2} & 0 \\
		0 & -\frac{i}{\de^2} + \wt{\BB}^{\spl}_{3,3}
	\end{pmatrix}, 
	&{\BB}^{\spl}=
	\begin{pmatrix} 
		-\frac{	\Ups_2}{	\Ups_1} \wt{\BB}^{\spl}_{2,1} &
		\wt{\BB}^{\spl}_{2,3} 
		-\frac{	\Ups_3}{	\Ups_1} \wt{\BB}^{\spl}_{2,1} \\[0.6em]
		\wt{\BB}^{\spl}_{3,2} 
		-\frac{	\Ups_2}{	\Ups_1} \wt{\BB}^{\spl}_{3,1} &
		-\frac{	\Ups_3}{	\Ups_1} \wt{\BB}^{\spl}_{3,1}
	\end{pmatrix}. 
\end{align}

Next, we give an heuristic idea of how to obtain an exponentially small bound for $\dyOut(u)$ for $u \in \DBoomerang$.
The case for $\dxOut$ is analogous. 
If we omit the influence of $\wt{\BB}^{\spl}$,
then there exists  $c_y\in\mathbb{C}$ such that $\dyOut$ is of the form
\[
\dyOut(u) = c_y \, e^{-\frac{i}{\de^2}u}.
\]
Evaluating this function at the points
\begin{align*}
	u_+ = i(A-\kappa \de^2), \qquad
	u_- =-i(A-\kappa \de^2),
\end{align*}
one has $\dyOut(u_+) \sim c_y e^{\frac{A}{\de^2}-\kappa}$. 
Then, since $\dyOut(u_+) \sim 1$, it implies that
$
c_y \sim e^{-\frac{A}{\de^2}+\kappa}
$
and, as a consequence, $\dyOut$ is exponentially small for  $u \in \reals$.
However, we are not interested in an upper bound of $\dyOut$ but in an asymptotic formula. 
Thus we have to find the constant $c_y$, or more precisely a good approximation of it. 

To this end, we need to give the main terms of $\dyOut$ at $u=u_+$.
%
Likewise we need to analyze $\dxOut(u)\sim c_x \, e^{\frac{i}{\de^2}u}$ at $u=u_-$.
%
%
%
%
To perform this analysis we proceed as follows:

\begin{enumerate}
	\item We provide suitable solutions $\ZusInn(U)$ of the so-called inner equation.
	The inner equation, see \cite{Bal06, BalSea08}, describes the dominant behavior of the functions $\zuOut$ and $\zsOut$ close to (one of) the singularities $u=\pm iA$.
	In particular, it involves the first order of the Hamiltonian $H^{\out}$ close to a singularity and it is independent of the small parameter $\de$.
	See Section \ref{subsection:innerHeuristics}.
	\item We check how well $\zusOut(u)$ are approximated by $\ZusInn(U)$ around the singularities $u=\pm iA$ by means of a complex matching procedure.
	See Section \ref{subsection:matching}.
\end{enumerate}

%
%

\subsubsection{The inner equation}
\label{subsection:innerHeuristics}

In this section we summarize the results on the derivation and study of the inner equation obtained in \cite{articleInner}.
We focus on the inner equation around the singularity $u=iA$, but analogous results hold near $u=-iA$. 
%

%

To derive the inner equation, we look for a new Hamiltonian which is a good approximation of $H^{\out}$, given in~\eqref{def:hamiltonianOuter}, in a suitable neighborhood of $u=iA$.
First, we scale the variables $(u,w,x,y)$ so that the graphs $\zusOut(u)$ become $\OO(1)$-functions when $u-iA=\OO(\de^2)$.
Since, by Theorem \ref{theorem:existence}, we have that
\begin{align*}
\wdOut(u) = \OO(\de^{-\frac43}), 
\qquad
\xdOut(u) = \OO(\de^{\frac13}),
\qquad
\ydOut(u) = \OO(\de^{\frac13}), 
\qquad 
\text{for } \diamond=\unstable,\stable,
\end{align*}
we consider the symplectic scaling
%
$
	\phi_{\Inner}:(U,W,X,Y) \to (u,w,x,y),
$
given by
\begin{equation}\label{def:changeInner}
	U=\frac{u-iA}{\de^2}, \hh
	W=\de^{\frac{4}{3}}\frac{w}{2 \al_+^2}, \hh
	X=\frac{x}{\de^{\frac{1}{3}} \sqrt{2} \al_+ } , \hh
	Y=\frac{y}{\de^{\frac{1}{3}}\sqrt{2} {\al_+}},
\end{equation}
where  $\al_+ \in \complexs$ is the constant given by in Theorem~\ref{theorem:singularities}, which is added to avoid the dependence of the inner equation on it.
Moreover,  we also perform the time scaling $\tau=\de^2 T$.
We refer to $(U,W,X,Y)$ as the \emph{inner coordinates}.


\begin{proposition}\label{proposition:innerDerivation}
The Hamiltonian system associated to~\eqref{def:hamiltonianOuter} expressed in the inner coordinates  is Hamiltonian with respect to the symplectic form $dU \wedge dW + i dX \wedge dY$ and
\begin{equation}\label{def:hamiltonianInnerComplete}
	H^{\Inner} = \HH + H_1^{\Inner},
\end{equation}
where 
\begin{equation*}
\HH(U,W,X,Y) = 
H^{\inn}(U,W,X,Y;\de) |_{\de=0} =
W + XY + \KK(U,W,X,Y),
\end{equation*}
with
\begin{align*}
\KK(U,W,X,Y) =& \,
-\frac{3}{4}U^{\frac{2}{3}} W^2 
- \frac{1}{3 U^{\frac{2}{3}}}
\paren{\frac{1}{\sqrt{1+\JJ(U,W,X,Y)}} - 1 }, 
\\
\begin{split}	
\JJ(U,W,X,Y) =& \,
\frac{4 W^2}{9 U^{\frac{2}{3}} }
-\frac{16 W}{27 U^{\frac{4}{3}}} 
+\frac{16}{81 U^{2}}
+\frac{4(X+Y)}{9 U} \paren{W 
-\frac{2}{3 U^{\frac23}}} \\[0.6em]
&-\frac{4i(X-Y)}{3 U^{\frac{2}{3}}} 
-\frac{X^2 + Y^2}{3 U^{\frac{4}{3}}}
+\frac{10 X Y}{9 U^{\frac{4}{3}}} .
\end{split} \nonumber
\end{align*}
Moreover, if 
$\cttInnDerA^{-1} \leq \vabs{U} \leq \cttInnDerA$ and 
$\vabs{(W,X,Y)} \leq \cttInnDerB$ 
for some $\cttInnDerA>1$  and $0<\cttInnDerB<1$,
there exist
$\cttInnDerC,\cttInnDerAA,\cttInnDerBB>0$
independent of 
$\de, \cttInnDerA, \cttInnDerB$ such that
\begin{equation}\label{eq:boundsH1Inn}
|H_1^{\Inner}(U,W,X,Y;\de)|
\leq 
\cttInnDerC \cttInnDerA^{\cttInnDerAA}
\cttInnDerB^{\cttInnDerBB}
\de^{\frac{4}{3}}.
\end{equation} 
\end{proposition}

This result is proven in \cite{articleInner} in Proposition 2.5.

Now, we present the study of the inner Hamiltonian $\HH$.
%
Denoting $Z=(W,X,Y)^T$, the equations associated to the Hamiltonian $\HH$, can be written as
\begin{equation*}
	\left\{ \begin{array}{l}
		\dot{U} = 1 + g^{\Inner}(U,Z),\\
		\dot{Z} = \AAA^{\Inn} Z + f^{\Inner}(U,Z),
	\end{array} \right.
\end{equation*}
where 
\begin{equation}\label{def:matrixAAA}
	\AAA^{\Inner}= \begin{pmatrix}
		0 & 0 & 0 \\
		0 & i & 0 \\
		0 & 0 & -i
	\end{pmatrix},
\end{equation}
and
$f^{\Inn} = \paren{-\partial_U \KK, 
	i \partial_Y \KK, -i\partial_X \KK }^T$ 
and
$g^{\inn} = \partial_{W} \KK$.
%
%
We look for invariant graphs $Z=\ZuInn(U)$ and $Z=\ZsInn(U)$ of this equation, that  satisfy the invariance equation also called \emph{inner equation},
\begin{equation}\label{eq:invariantEquationInner}
	\partial_U \ZdInn(U) = 
	\AAA^{\Inner} \ZdInn + \RRR^{\Inner}[\ZdInn](U),
	\qquad \text{for } 
	\diamond=\unstable, \stable,
\end{equation}
with
\begin{equation}\label{def:operatorRRRInner}
\RRR^{\Inn}[\varphi](U)= 
	\frac{f^{\Inner}(U,\varphi)- g^{\Inner}(U,\varphi) \AAA^{\Inner} \varphi }{1+g^{\Inner}(U,\varphi)}.
\end{equation}
%
%
\begin{figure}[t] 
	\centering
	\vspace{5mm}
	\begin{overpic}[scale=0.8]{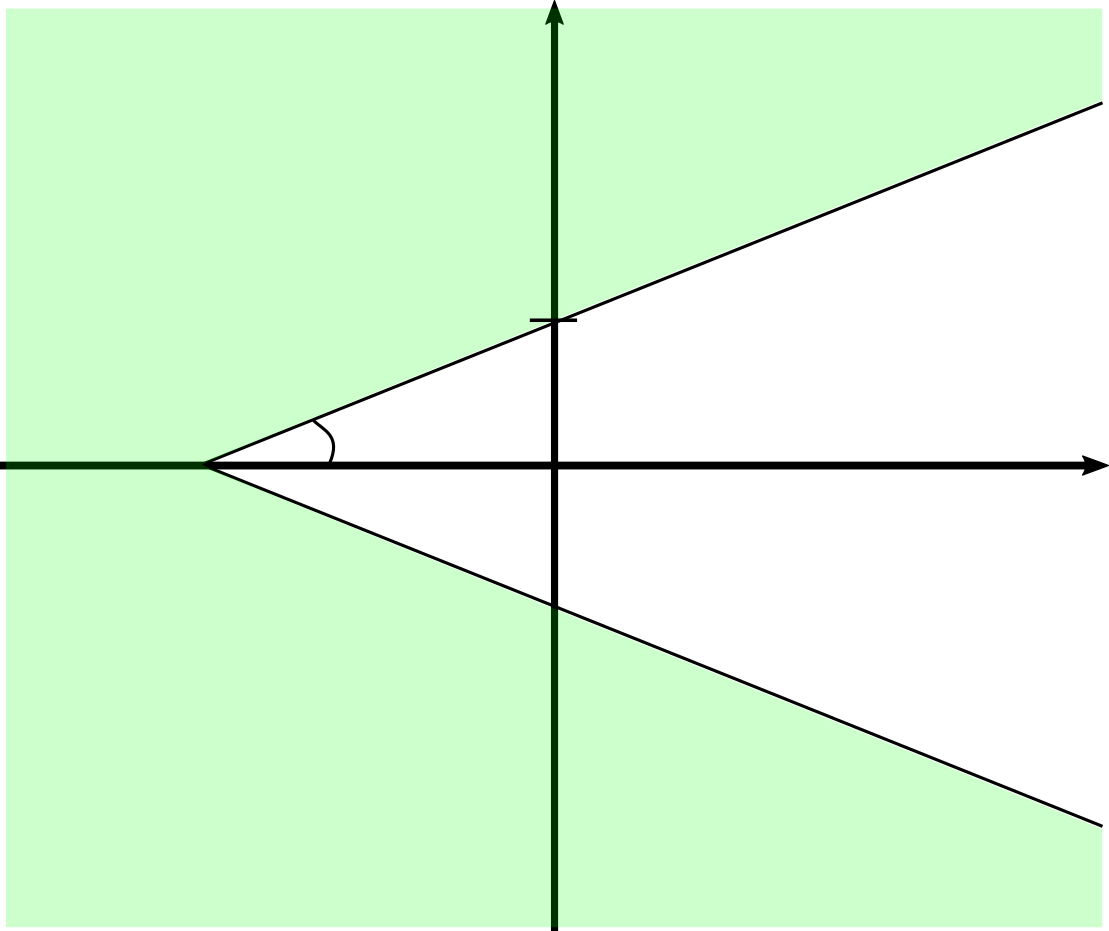}
		\put(10,60){$\DuInn$}
		\put(32,44){{$\betaInner$}}
		\put(44,54){{$\rhoInn$}}
		\put(102,41){$\Re U$}
		\put(46,86){$\Im U$}
	\end{overpic}
	\bigskip
	\caption{The inner domain $\DuInn$ for the unstable case.}
	\label{fig:dominiInnerUnstable}
\end{figure} 
These functions $\ZuInn$ and $\ZsInn$ will be defined 
in the domains
\begin{equation*}
\DuInn = \claus{ U \in \complexs \text{ : }
|\Im U| \geq \tan \betaInner \Re U + \rhoInn}, 
\qquad
 \DsInn = -\DuInn,
\end{equation*}
respectively, for some $\rhoInn>0$ and 
with $\betaBow$ as given in Theorem~\ref{theorem:existence} (see Figure~\ref{fig:dominiInnerUnstable}).
%
%
Moreover, we analyze the difference $\DZInn=\ZuInn-\ZsInn$ in the overlapping domain
\begin{equation*}
\EInn = \DuInn \cap \DsInn \cap 
	\claus{ U \in \complexs \st \Im U < 0}.
\end{equation*}


\begin{theorem}\label{theorem:innerComputations}
There exist $\kappaInner, \cttInnExist>0$
such that for $\rhoInn\geq\kappaInner$,	
the equation~\eqref{eq:invariantEquationInner} has analytic solutions
$
\ZdInn(U) =(\WdInn(U),\XdInn(U),\YdInn(U))^T, 
$ for $U \in \DdInn$, $\diamond=\unstable,\stable$, satisfying
\begin{equation*}
	| U^{\frac{8}{3}} \WdInn(U)| \leq \cttInnExist, \qquad
	| U^{\frac{4}{3}} \XdInn(U) | \leq \cttInnExist, \qquad
	| U^{\frac{4}{3}} \YdInn(U) | \leq \cttInnExist.
\end{equation*}
In addition, there exist $\CInn \in \complexs$, $\cttInnDiff>0$ independent of $\rhoInn$, and an analytic function $\chi=(\chi_1,\chi_2,\chi_3)^T$ 
such that, for $U \in \EInn$,
\begin{equation*}
	\DZInn(U) = \ZuInn(U)-\ZsInn(U) =
	\CInn e^{-iU} \Big(
	(0,0,1)^T + \chi(U) \Big),
\end{equation*}
with $|(U^{\frac{7}{3}} \chi_1(U), 
U^{2} \chi_2(U),U \chi_3(U))|
\leq \cttInnDiff$.
%
\end{theorem}

This result is  Theorem 2.7 of \cite{articleInner}.

\begin{remark}\label{remark:innerComputationConjugats}
To obtain the analogous result to Theorem \ref{theorem:innerComputations} near the singularity $u=-iA$, one must perform the change of coordinates
\begin{equation*}
{V}=\frac{u+iA}{\de^2}, \qquad
\widehat{W}=\de^{\frac{4}{3}}\frac{w}{2 \al_-^2}, \qquad
\widehat{X}=\frac{x}{\de^{\frac{1}{3}} \sqrt{2} \al_- }, 
\qquad
\widehat{Y}=\frac{y}{\de^{\frac{1}{3}}\sqrt{2} {\al_-}},
\end{equation*}
where  $\al_- \in \complexs$ is $\al_- =\conj{\al_+}$ (see Theorem \ref{theorem:singularities}). 	
Then, for $V \in \conj{\DdInn}$, one can prove the existence of the corresponding solutions 
\[
\ZdInnN(V)=(\WdInnN(V),\XdInnN(V),\YdInnN(V))^T, \qquad
\text{where } \diamond = \mathrm{u, s}.
\]
Due to the real-analyticity of the problem (see Remark \ref{remark:realAnalytic}) we have that $\widehat{X}^{\diamond}(V) = \conj{Y^{\diamond}(U)}$. 
Therefore, the difference $\DZInnN = \ZuInnN - \ZsInnN$, is given asymptotically for $U \in \conj{\EInn}$ by
\begin{equation*}
	\DZInnN(V) = \conj{\CInn} e^{iV} \Big(
	(0,1,0)^T + \zeta(V) \Big),
\end{equation*}
where $\zeta=(\zeta_1,\zeta_2,\zeta_3)^T$ satisfies 
$|(V^{\frac{7}{3}} \zeta_1(V),
V \zeta_2(V),
V^{2} \zeta_3(V)| \leq C$, for a constant $C$ independent of $\kappa$.
\end{remark}
%
%

\subsubsection{Complex matching estimates} \label{subsection:matching}

We now study how well the solutions of the inner equation approximate the solutions of the original system given by by Proposition \ref{proposition:existenceComecocos} in an appropriate domain.
%
%
%
%
%
%
%
As in the previous section, we focus on the singularity $u=iA$, but analogous results can be proven for $u=-iA$ (see Remark \ref{remark:innerComputationConjugats}).
Let us recall that the functions $\zusOut$ are expressed in the separatrix coordinates (see \eqref{def:changeOuter}) 
while the functions $\ZusInn$ are expressed in inner coordinates (see \eqref{def:changeInner}).
%
%
%

We first define the matching domains in separatrix coordinates and, later, we translate them to the inner coordinates. Let us consider $\betaMchA$, $\betaMchB$,
and $\gamma$  independent of $\de$ and $\kappa$, such that
\begin{equation*}
0< \betaMchA < \betaOutA < \betaMchB < \frac{\pi}{2},
\qquad \text{and} \qquad
\g \in \left[\frac35 ,1\right),
\end{equation*}
with $\betaOutA$ as given in Theorem \ref{theorem:singularities}.
Then, we define $u_j \in \complexs$ $j=2,3$ (see Figure \ref{fig:matchingDomains}), as the points satisfying:
\begin{itemize}
	\item $\Im u_j = -\tan \beta_{j} \Re u_j + A - \kappa \de^2$.
	\item $\vabs{u_j- u_+} = \de^{2\gamma}$, where $u_+=i(A-\kappa\de^2)$.
	\item $\Re u_2<0$ and $\Re u_3>0$.
\end{itemize}
\begin{figure}[t] 
	\centering
	\vspace{5mm}
	\begin{overpic}[scale=0.9]{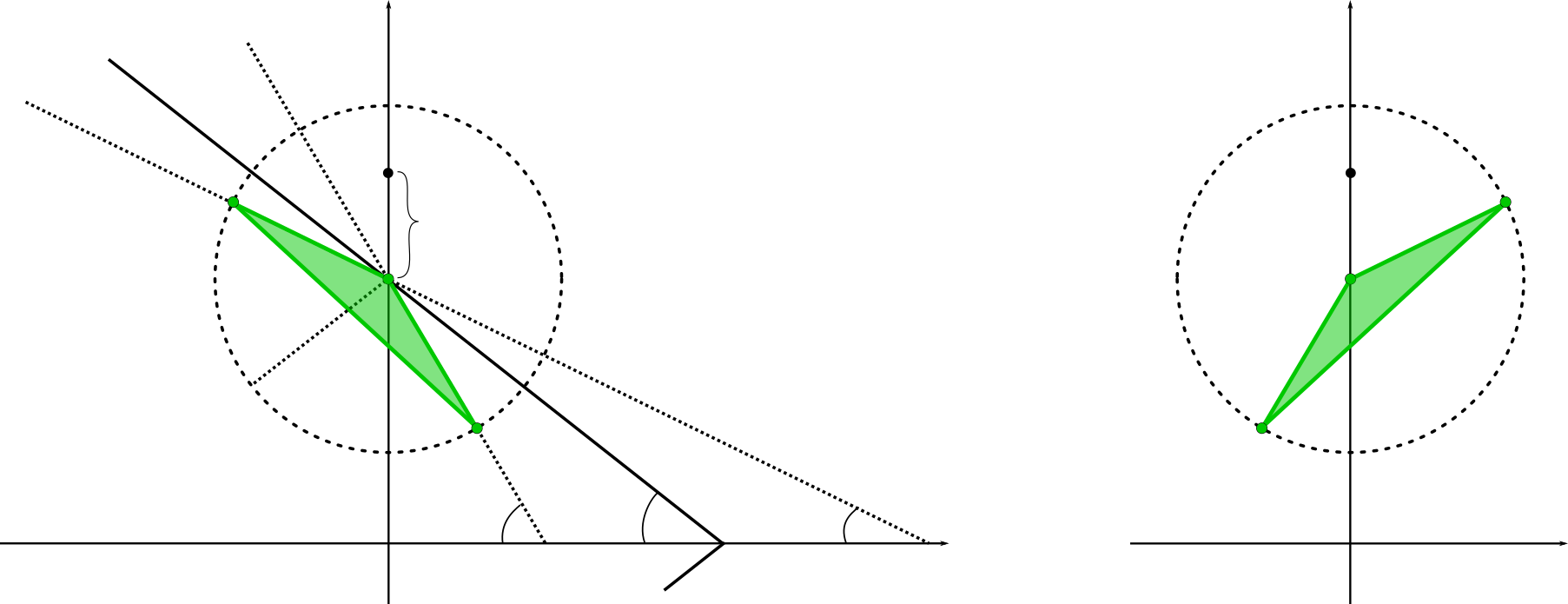}
		\put(11,8){\textcolor{myGreen}{{\footnotesize $\DuMchOut$}}}
		\put(91,7){\textcolor{myGreen}{{\footnotesize $\DsMchOut$}}}
		\put(29,4.5){{\footnotesize $\betaMchB$}}
		\put(38.5,4.5){{\footnotesize $\betaOutA$}}
		\put(50.5,4.5){{\footnotesize $\betaMchA$}}
		\put(21.5,26.5){{\footnotesize $iA$}}
		\put(83,26.5){{\footnotesize $iA$}}
		\put(27,24){{\footnotesize $\kappa \de^2$}}
		\put(16.5,16.5){{\footnotesize $\de^{2\g}$}}
		\put(11.5,24.5){\textcolor{myGreen}{{\footnotesize $u_2$}}}
		\put(31,10){\textcolor{myGreen}{{\footnotesize $u_3$}}}
		\put(26.25,20.25){\textcolor{myGreen}{{\footnotesize $u_+$}}}
		\put(22,39){{\footnotesize $\Im u$}}
		\put(84,39){{\footnotesize $\Im u$}}
		\put(61,3){{\footnotesize$\Re u$}}
		\put(101,3){{\footnotesize$\Re u$}}
		\put(97,25){\textcolor{myGreen}{{\footnotesize $-\conj{u}_2$}}}
		\put(76,10){\textcolor{myGreen}{{\footnotesize $-\conj{u}_3$}}}
		\put(83,21){\textcolor{myGreen}{{\footnotesize $u_+$}}}
	\end{overpic}
	\caption{The matching domains $\DuMchOut$ and $\DsMchOut$ in the outer variables.}
	\label{fig:matchingDomains}
\end{figure} 
We define the matching domains in the separatrix coordinates as the triangular domains
\begin{align*}
	\DuMchOut = 
	\that{u_+ \, u_2 \, u_3},
	\qquad
	\DsMchOut =
	\that{u_+ \, (-\conj{u_2}) \, (-\conj{u_3})}.
\end{align*}
%
%
%
%
Let $\dOuter$, $\rhoOuter$ and $\kappaOuter$ be as given in Proposition \ref{proposition:existenceComecocos}. 
Then, for $\kappa\geq \kappaOuter$ and $\de>0$ small enough, the  matching domains satisfy
\begin{equation}\label{eq:domainsMatchingOuterSubset}
\DuMchOut \subset \DuOut \qquad \text{ and } \qquad
\DsMchOut \subset \DsOut,
\end{equation}
and, as a result, $\zuOut$ and $\zsOut$ are well defined in $\DuMchOut$ and $\DsMchOut$, respectively. 

%
%
%

The  matching domains in inner variables are defined by
\begin{equation}\label{def:matchingDomainInner}
\DdMchInn = \left\{
U \in \complexs \st \de^2 U + iA \in \DdMchOut
\right\}, \qquad
\text{ for } \diamond=\unstable,\stable,
\end{equation} 
with
\begin{equation}\label{def:U12}
U_j = \frac{u_j - iA}{\de^2},
\hh
\text{ for } j=2,3.
\end{equation}
%
%
%
%
Therefore, for $U \in \DdMchInn$, 
\begin{equation*}
	\kappa \cos \betaMchA  \leq 
	\vabs{U} \leq 
	\frac{C}{\de^{2(1-\g)}}.
\end{equation*}
By definition, 
\begin{equation*}
\DuMchInn \subset \DuInn \qquad\text{ and } \qquad
\DsMchInn \subset \DsInn,
\end{equation*}
for $\kappa \geq \kappaInner$ (see Theorem \ref{theorem:innerComputations}). Thus, $\ZusInn$ is well defined in $\DusMchInn$.

In order to compare $\zusOut(u)$ and $\ZusInn(U)$, we translate $\zusOut$ to inner coordinates
\begin{equation}\label{def:ZdMchO}
\ZdMchO(U) =
\big(\WdMchO,\XdMchO, \YdMchO\big)^T(U) =
 \paren{
	\de^{\frac43}\frac{ \wdOut}{2 \al_+^2}, \,
	\frac{\xdOut}{\de^{\frac13} \sqrt{2}\al_+}, \,
	\frac{\ydOut}{\de^{\frac13} \sqrt{2}\al_+}
}^T(\de^2 U + iA),
\end{equation}
with $\diamond=\unstable,\stable$ and $\zdOut=(\wdOut, \xdOut, \ydOut)^T$ are given in Proposition \ref{proposition:existenceComecocos}.
%
%
Therefore, by \eqref{eq:domainsMatchingOuterSubset},
$\ZdMchO$ is well defined in the  matching domain $\DdMchInn$ (which is expressed in inner variables).

Next theorem  gives estimates for $\ZusMchO-\ZusInn$.

%
%

\begin{theorem}\label{theorem:matching}
Consider $\kappaOuter$ and $\kappaInner$ as obtained in Proposition~\ref{proposition:existenceComecocos}
and  	
Theorem \ref{theorem:innerComputations}, respectively.
Then, there exist $\g^*\in [\frac35,1)$, 
$\kappaMch \geq \max\claus{\kappaOuter,\kappaInner}$ 
and $\de_0>0$ such that, 
for  $\g \in (\g^*,1)$,
there exists $\cttMch>0$ satisfying that, for $U \in \DdMchInn$,
$\kappa\geq \kappaMch$
and 
$\de \in (0,\de_0)$,
\begin{equation*}
|U^{\frac{4}{3}}\WdMchU(U)| \leq \cttMch \de^{\frac{2}{3}(1-\g)}, \qquad
\vabs{U \h \XdMchU(U)} \leq \cttMch \de^{\frac{2}{3}(1-\g)}, \qquad
\vabs{U \h \YdMchU(U)} \leq \cttMch  \de^{\frac{2}{3}(1-\g)},
\end{equation*}
with 
$
(\WdMchU,\XdMchU,\YdMchU)^T=\ZdMchU
=\ZdMchO - \ZdInn$ and $\diamond=\unstable,\stable$.
\end{theorem}

This theorem is proven in Section \ref{section:proofG-matching}.

\subsection{The asymptotic formula for the difference}
\label{section:difference} 

We look for an asymptotic expression for the difference 
\[
\dzHat=(\dxOut,\dyOut)^T
=(\xuOut-\xsOut,\yuOut-\ysOut)^T,
\] 
where $(\xuOut,\yuOut)$ and $(\xsOut,\ysOut)$ are components of the perturbed invariant graphs given in Theorem \ref{theorem:existence}.
Recall that, by \eqref{eq:invariantEquationDifference2}, $\dzHat$ satisfies
\begin{equation}\label{eq:invariantEquationDifference3}
	\partial_u {\dzHat}(u) = 
	{\AAA}^{\spl}(u) \dzHat(u)
	+ 
	{\BB}^{\spl}(u) \dzHat(u),
\end{equation} 
with $\AAA^{\spl}$ and $\BB^{\spl}$ as given in \eqref{def:operatorsDifferenceAABB}.
%
%
%
The equation is split as a dominant part, given by the matrix $\AAA^{\spl}$ and a small perturbation corresponding to the the matrix ${\BB}^{\spl}$.
Therefore, it makes sense to look for  $\dzHat$ as $\dzHat = \dzHatO + \textit{h.o.t}$ with a suitable dominant term  $\dzHatO=(\dxOutO,\dyOutO)^T$  satisfying
\begin{equation}\label{eq:invariantEquationDifferenceHomogeneous}
\partial_u {\dzHat_0}(u) = 
\AAA^{\spl}(u)\dzHatO(u).
\end{equation}
A fundamental matrix of \eqref{eq:invariantEquationDifferenceHomogeneous}, for $u \in \DBoomerang$, is given by
\begin{equation}\label{def:fundamentalMatrixDifference}
	\MM(u) = \begin{pmatrix}
		m_x(u) & 0 \\
		0 & m_y(u)
	\end{pmatrix},
\end{equation}
with 
\begin{equation}\label{def:mxmyalxaly}
	\begin{aligned}
		m_x(u) &= e^{\frac{i}{\de^2}u} \bb_x(u), 
		&\bb_x(u)=\exp\paren{\int_{u_*}^u \wt{\BB}^{\spl}_{2,2}(s) d s}, \\
		m_y(u)& = e^{-\frac{i}{\de^2}u} \bb_y(u), 
		&\bb_y(u)=\exp\paren{\int_{u_*}^u \wt{\BB}^{\spl}_{3,3}(s) d s},
	\end{aligned} 
\end{equation}
and a fixed $u_* \in \DBoomerang \cap \reals$. Then, $\Delta\Phi_0$ must be of form 
\begin{equation}\label{def:dzHatO:0}
	\dzHatO(u)= \begin{pmatrix}
		\Delta x_0(u)\\[0.5em]
		\Delta y_0(u)
	\end{pmatrix}=\begin{pmatrix}
		c_x^0 m_x(u)\\[0.5em]
		c_y^0 m_y(u)
	\end{pmatrix},
\end{equation}
for suitable constants $c_x^0,c_y^0\in\mathbb{C}$ which we now determine.

By Theorems~\ref{theorem:innerComputations} and~\ref{theorem:matching} and using the inner change of coordinates in \eqref{def:changeInner},
we have a  good approximation of $\dyOut (u)$
near the singularity $u=iA$ given by
\[
\dyOut(u) \approx 
\sqrt{2} \al_+ \de^{\frac{1}{3}} \DYInn\paren{\frac{u-iA}{\de^2}}.
\]
Then, taking  $u=u_+=i(A-\kappa\de^2)$,  we have that
\begin{equation*}
\dyOut(u_+) \approx\dyOut_0(u_+) \approx
\sqrt{2} \al_+ \de^{\frac{1}{3}} \DYInn\paren{\frac{u_+-iA}{\de^2}} =
\sqrt{2} \al_+ \de^{\frac{1}{3}} e^{-\kappa} \CInn (1+\chi_3(-i\kappa)).
\end{equation*}
Then, using that $\dyOut(u_+) \approx\dyOut_0(u_+)=c_y^0 m_y(u_+)$,
and proceeding analogously for the component $\dxOut$ at the point $u_-=-i(A-\kappa\de^2)$ (see Remark~\ref{remark:innerComputationConjugats}),
%
we take 
\begin{equation}\label{def:dzHatO}
		c_x^0 = \de^{\frac{1}{3}} e^{-\frac{A}{\de^2}}
		\conj{\CInn} \,  \sqrt{2} \al_-  \bb^{-1}_x(u_-)\qquad \text{and}\qquad
		c_y^0 =	\de^{\frac{1}{3}} e^{-\frac{A}{\de^2}}
		{\CInn} \sqrt{2} \al_+  \bb^{-1}_y(u_+).
\end{equation}
To prove Theorem \ref{theorem:mainTheoremPoincare}, we check that $\dzHatO(u)$ is the leading term of  $\dzHat(u)$, for $u \in \reals \cap \DBoomerang$, by estimating  the remainder $\dzHatU=\dzHat-\dzHatO$.
%

%
%
In order to simplify the notation, throughout the rest of the document, we denote by $C$ any positive constant independent of $\de$ and $\kappa$ to state estimates.

\subsubsection{End of the proof of Theorem \ref{theorem:mainTheoremPoincare}}
We look for $\dzHatU$ as the unique solution of an integral equation.
%
Since $\dzHat$ satisfies~\eqref{eq:invariantEquationDifference3}, by the variations of constants formula
\begin{align}\label{eq:dzOutDiff}
	\dzHat(u) = 
	\begin{pmatrix}
		c_x m_x(u)\\
		c_y m_y(u)
	\end{pmatrix}
	+
	\begin{pmatrix}
		\displaystyle m_x(u)
		\int_{u_-}^{u} m_x^{-1}(s) \,
		\pi_1 \paren{{\BB}^{\spl}(s) \dzHat(s)} ds \\
		\displaystyle m_y(u) 
		\int_{u_+}^{u} m_y^{-1}(s) \, 
		\pi_2 \paren{{\BB}^{\spl}(s) \dzHat(s)} ds
	\end{pmatrix},
\end{align}
where $\MM(u)$ is the fundamental matrix \eqref{def:fundamentalMatrixDifference}, $s$ belongs to some integration path in $\DBoomerang$ and $c_x$ and $c_y$ are defined as
\begin{align}\label{def:cxcyDifference}
c_x = \dxOut(u_-) m_x^{-1}(u_-), \qquad
c_y = \dyOut(u_+) m_y^{-1}(u_+).
\end{align}
For $k_1, k_2 \in \complexs$, we define 
\begin{equation}\label{def:operatorIIdiff}
\II[k_1,k_2](u) =
\big(k_1 \, m_x(u), k_2 \, m_y(u)\big)^T,
\end{equation}
and the operator
\begin{align}\label{def:operatorEEdiff}
	\EE[\phiA](u) = 
	\begin{pmatrix}
		\displaystyle
		m_x(u) \int_{u_-}^{u} m_x^{-1}(s) \,
		\pi_1 \paren{{\BB}^{\spl}(s) \phiA(s)} ds \\
		\displaystyle
		m_y(u) \int_{u_+}^{u} m_y^{-1}(s) \,
		\pi_2 \paren{{\BB}^{\spl}(s) \phiA(s)} ds
	\end{pmatrix}.
\end{align}
Then, with this notation, $\dzHatO = \II[c_x^0,c_y^0]$ (see \eqref{def:dzHatO}) and equation~\eqref{eq:dzOutDiff} is equivalent to $\dzHat = \II[c_x,c_y]+\EE[\dzHat]$. 
Since $\EE$ is a linear operator, $\dzHatU = 
\dzHat-\dzHatO$ satisfies
\begin{equation}\label{eq:dzOutU}
	\dzHatU(u) = 
	\II[c_x-c_x^0, c_y-c_y^0](u) + 
	\EE[\dzHatO](u) + \EE[\dzHatU](u).
\end{equation}

To obtain estimates for $\dzHatU$, we first prove that $\Id-\EE$ is invertible in the Banach space $\XSplTotal= \XSpl \times \XSpl$, with 
%
\begin{align*}
	\XSpl = \left\{ \phiA: \DBoomerang \to \complexs \st \normSpl{\phiA}
	= \sup_{u\in \DBoomerang} \vabs{e^{\frac{A-\vabs{\Im u}}{\de^2}} 
	\phiA(u)}
	<+\infty \right\},
\end{align*}
endowed with the norm
\begin{align}\label{def:normexp}
	\normSplTotal{\phiA} = 
	\normSpl{\phiA_1} + \normSpl{\phiA_2},
\end{align}
for $\phiA=(\phiA_1,\phiA_2)$.
Therefore, to prove Theorem \ref{theorem:mainTheoremPoincare} it is enough to see that $\dzHatU$  satisfies that $\normSplTotal{\dzHatU} \leq C\de^{\frac13}\vabs{\log \de}^{-1}$.

First, we state a lemma  whose proof is postponed to Appendix \ref{subappendix:proofH-technicalFirst}.

\begin{lemma}\label{lemma:boundsBspl}
Let $\kappaBoomerang, \de_0$ be the constants given in Theorem \ref{theorem:existence}.
Then, there exists a constant $C>0$ such that, for $\kappa\geq\kappaBoomerang$, $\de \in(0,\de_0)$ and  $u \in \DBoomerang$, 
the function $\Ups$ in \eqref{def:operatorPPdifference}, 
the matrix $\BB^{\spl}$ in \eqref{def:operatorsDifferenceAABB} and the functions $\bb_x$, $\bb_y$ in \eqref{def:mxmyalxaly} satisfy
for $\kappa\geq\kappaBoomerang$, $\de \in(0,\de_0)$ and $u \in \DBoomerang$,
\begin{align}
	&\vabs{\Ups_1(u)-1}\leq \frac{C}{\kappa^2}, \qquad
	\vabs{\Ups_2(u)}\leq  \frac{C\de}{\vabs{u^2+A^2}^{\frac{4}{3}}}, \qquad
	\vabs{\Ups_3(u)}\leq  \frac{C\de}{\vabs{u^2+A^2}^{\frac{4}{3}}}, \label{proof:boundsUpsilon}
	\\[0.4em]
{C}^{-1} &\leq \vabs{\bb_*(u)} \leq C, 
\quad *=x,y, 
\quad \text{and} \quad
|{\BB}^{\spl}_{i,j}(u)| \leq 
\frac{C \, \de^2}{\vabs{u^2 + A^2}^{2}}, 
\quad i,j \in \claus{1,2}. \nonumber
\end{align}
\end{lemma}

In the next lemma we obtain estimates for the linear operator $\EE$ (see \eqref{def:operatorEEdiff}).

\begin{lemma}\label{lemma:operatorEEdiff}
%
Let $\kappaBoomerang, \de_0$ be the constants as given in Theorem \ref{theorem:existence}.
There exists $\cttDiffA>0$ such that for $\de\in(0,\de_0)$ and 
$\kappa\geq \kappaBoomerang$,
the operator $\EE: \XSplTotal \to \XSplTotal$ in \eqref{def:operatorEEdiff}  is well defined and satisfies that, for $\phiA \in \XSplTotal$,
\begin{align*}
\normSplTotal{\EE[\phiA]} \leq \frac{\cttDiffA}{\kappa} 
\normSplTotal{\phiA}.
\end{align*}
In particular, $\Id - \EE$ is invertible and 
\begin{align*}
\normSplTotal{(\Id - \EE)^{-1}[\phiA]} \leq 2\normSplTotal{\phiA}.
\end{align*}

\end{lemma}

\begin{proof}
Let us consider $\EE=(\EE_1,\EE_2)^T$, $\phiA \in \XSplTotal$ and $u \in \DBoomerang$.
We only prove the estimate for $\EE_2[\phiA](u)$.
The corresponding one for $\EE_1[\phiA](u)$ follows analogously.

By the definition of $m_y$ in~\eqref{def:mxmyalxaly} and Lemma \ref{lemma:boundsBspl}, we have that
\begin{align*}
\vabs{\EE_2[\phiA](u)} 
%
&\leq
C \de^2 e^{\frac{\Im u}{\de^2}} 
\vabs{\int_{u_+}^{u} 
e^{-\frac{\Im s}{\de^2}} 
\frac{
\vabs{\phiA_1(s)}+\vabs{\phiA_2(s)}
}{\vabs{s^2+A^2}^2} d s } 
\\
&\leq
C \de^2 e^{\frac{\Im u - A}{\de^2}} 
\normSplTotal{\phiA}
\vabs{\int_{u_+}^{u} 
e^{\frac{\vabs{\Im s}-\Im s}{\de^2}} 
\frac{d s}{\vabs{s^2+A^2}^2}}.
\end{align*}
%
Let us consider the case $\Im u < 0$. Then, for a fixed $u_0 \in \reals \cap \DBoomerang$, we define the integration path $\rho_t \subset \DBoomerang$ as
\begin{align*}
	\rho_t =	
	\begin{cases}
		u_+ + 2t(u_0-u_+)
		&\quad \text{for } t \in (0,\frac12), \\
		u_0 + (2t-1)(u-u_0)
		&\quad \text{for }  t \in [\frac12,1).
	\end{cases}	
\end{align*}
Then,
\begin{align*}
\vabs{\EE_2[\phiA](u)} 
%
&\leq 
C \de^2  e^{-\frac{\vabs{\Im u}+A}{\de^2}} 
\normSplTotal{\phiA}
\vabs{\int_{0}^{\frac12} \frac{dt}{\vabs{\rho_t-iA}^2}
+ \int_{\frac12}^{1} 
\frac{e^{\frac{2\vabs{\Im \rho_t}}{\de^2}}}{\vabs{\rho_t+iA}^2} dt}  
%
%
\leq \frac{C}{\kappa} 
e^{\frac{\vabs{\Im u}-A}{\de^2}}
\normSplTotal{\phiA}.
\end{align*}
If $\Im u \geq 0$, we consider the integration path $\rho_t = u_+ + t(u-u_+)$ for $t\in[0,1]$ and we obtain
\begin{align*}
\vabs{\EE_2[\phiA](u)} &\leq 
C \de^2  e^{\frac{\vabs{\Im u}-A}{\de^2}} 
\normSplTotal{\phiA}
\vabs{\int_{0}^{1} \frac{\vabs{u-u_+}}{\vabs{\rho_t-iA}^2} dt} 
%
\leq \frac{C}{\kappa} 
e^{\frac{\vabs{\Im u}-A}{\de^2}}
\normSplTotal{\phiA}.
\end{align*}
Therefore,
$
\normSpl{\EE_2[\phiA]} \leq \frac{C}{\kappa}\normSplTotal{\phiA}.
$
\end{proof}

Notice that, by  \eqref{eq:dzOutU}, $\dzHatU$ satisfies
\begin{equation}\label{eq:dzOutUInverse}
(\Id - \EE )\dzHatU(u) = \II[c_x-c_x^0, c_y-c_y^0](u) + \EE[\dzHatO](u).
\end{equation}
Since, by Lemma \ref{lemma:operatorEEdiff}, $\Id-\EE$ is invertible in $\XSplTotal$ we have an explicit formula for $\dzHatU$.
Nevertheless, we still need good estimates for  the right hand side with respect to the norm \eqref{def:normexp}.

%

%


\begin{lemma}\label{lemma:IconstantsDiff}
There exist $\kappa_*, \de_0, \cttDiffB>0$ such that, for
$\kappa=\kappa_*\vabs{\log \de}$ and $\de\in(0,\de_0)$, 
%
\begin{align*}
\normSplTotal{\II[c_x-c_x^0,c_y-c_y^0]}
\leq \frac{\cttDiffB \, \de^{\frac{1}{3}}}{\vabs{\log \de}}\qquad \text{and}\qquad
\normSplTotal{\EE[\dzHatO](u)}\leq \frac{\cttDiffB \, \de^{\frac{1}{3}}}{\vabs{\log \de}},
\end{align*}
with  $\II$,  $(c_x^0,c_y^0)$, $(c_x,c_y)$, $\EE$ and $\dzHatO$ defined  in \eqref{def:operatorIIdiff}, \eqref{def:dzHatO}, \eqref{def:cxcyDifference}, \eqref{def:operatorEEdiff}
and
\eqref{def:dzHatO:0}, respectively.

\end{lemma}

\begin{proof}
By the definition of the function $\II$,
\[
\normSplTotal{\II[c_x-c_x^0,c_y-c_y^0]} = \vabs{c_x-c_x^0}\normSpl{m_x} + \vabs{c_y-c_y^0}\normSpl{m_y},
\]
where $m_x$ and $m_y$ are given in \eqref{def:mxmyalxaly}.
Then, by Lemma \ref{lemma:boundsBspl},
\begin{align*}
	\normSpl{m_x} = 
	e^{\frac{A}{\de^2}} 
	\sup_{u \in \DBoomerang} 
	\boxClaus{e^{-\frac{\Im u+\vabs{\Im u}}{\de^2}}  \vabs{\bb_x(u)}}
	\leq C e^{\frac{A}{\de^2}},
	\qquad
	\normSpl{m_y} \leq 
	%
	%
	%
	C e^{\frac{A}{\de^2}},
\end{align*}
and, as a result,
\begin{equation}\label{proof:operatorIIdiff}
\normSplTotal{\II[c_x-c_x^0,c_y-c_y^0]} \leq
C e^{\frac{A}{\de^2}}
\paren{|c_x-c_x^0| + |c_y-c_y^0|}.
\end{equation}
We now obtain an estimate for $|c_y-c_y^0|$.
The estimate for $|c_x-c_x^0|$ follows analogously.

By the definition of $m_y$ (see \eqref{def:mxmyalxaly}), one has
\begin{equation}\label{proof:cyMenyscyO}
\begin{split}	
	\vabs{c_y - c_y^0} 
	&= 	e^{-\frac{A}{\de^2}+\kappa} 
	\vabs{\bb_y^{-1}(u_+)}
	\vabs{\dyOut(u_+) - \dyOutO(u_+)}.
\end{split}
\end{equation}
%
%
%
Let us denote $\DYInnC = \YuMchO -\YsMchO$ where $\YusMchO$ are given on~\eqref{def:ZdMchO}. 
Recall that $\YusMchO=\YusInn + \YusMchU$ where $\YusInn$ is the third component of $\ZusInn$, the solutions of the inner equation (see Theorems \ref{theorem:innerComputations} and \ref{theorem:matching}).
We write,
\begin{align*}
\dyOut(u_+) &= \sqrt{2} \al_+ \de^{\frac{1}{3}} 
\DYInnC\paren{\frac{u_+ - iA}{\de^2}} 
=\sqrt{2} \al_+ \de^{\frac{1}{3}} \left[
\DYInn \paren{-i\kappa} +
\YuMchU \paren{-i\kappa} - \YsMchU \paren{-i\kappa}
\right].
\end{align*}
%
%
By the definition of $\dyOutO$ in \eqref{def:dzHatO:0} (see also \eqref{def:dzHatO}), we have
$
	\dyOutO(u_+) = \sqrt{2}\al_+ \de^{\frac{1}{3}}
	\CInn e^{-\kappa} .
$
Then, by  \eqref{proof:cyMenyscyO} and Lemma~\ref{lemma:boundsBspl},
\begin{align*}
	\vabs{c_y-c_y^0} \leq
	C \de^{\frac{1}{3}} e^{-\frac{A}{\de^2}+\kappa} 
	\Big[
	\vabs{\DYInn \paren{-i\kappa} - \CInn e^{-\kappa}}
	+ 
	\vabs{\YuMchU \paren{-i\kappa}} + 
	\vabs{\YsMchU \paren{-i\kappa}}
	\Big],
\end{align*}
and, applying Theorems~\ref{theorem:innerComputations} and
\ref{theorem:matching}, we obtain
\begin{equation*}
\begin{split}	
\vabs{c_y-c_y^0}
&\leq
C \de^{\frac{1}{3}} 
e^{-\frac{A}{\de^2}+\kappa} 
\boxClaus{
\vabs{\chi_3(-i\kappa)  e^{-\kappa} }
+
\frac{C}{\kappa} \de^{\frac23(1-\g)} } 
\leq 
\frac{C}{\kappa} \de^{\frac13} e^{-\frac{A}{\de^2}} 
\paren{1 + 
\de^{\frac23(1-\g)} e^{\kappa}},
\end{split}
\end{equation*}
where $\g \in (\g^*,1)$ with $\g^*\in [\frac{3}{5},1)$ given in Theorem \ref{theorem:matching}.
%
%
%
Taking $\kappa=\kappa_* \vabs{\log \de}$ with $0<\kappa_*<\frac{2}{3}(1-\g)$, we obtain
\begin{equation*}
\begin{split}	
\vabs{c_y-c_y^0}
&\leq 
\frac{C\de^{\frac13} }{\vabs{\log \de}} 
e^{-\frac{A}{\de^2}} 
\paren{1 +	\de^{\frac23(1-\g)-\kappa_*}}
\leq
\frac{C \de^{\frac13} }{\vabs{\log \de}}
e^{-\frac{A}{\de^2}}.
\end{split}
\end{equation*}
This bound and  \eqref{proof:operatorIIdiff} prove the first estimate of the lemma.

For the second estimate, it only remains to bound  $\dzHatO$ and apply Lemma~\ref{lemma:operatorEEdiff}.
Indeed, by the definition of $\dzHatO$  in \eqref{def:dzHatO}, Lemma \ref{lemma:boundsBspl} and \eqref{proof:operatorIIdiff}, we have that
\begin{align*}
\normSplTotal{\dzHatO} 
=
\normSplTotal{\II[c_x^0,c_y^0]} 
\leq
C e^{\frac{A}{\de^2}} \paren{\vabs{c_x^0}+\vabs{c_y^0}}
\leq
C \de^{\frac13}.
\end{align*}	
%
%
Since $\kappa=\kappa_* \vabs{\log \de}$ with  $0<\kappa_*<\frac{2}{3}(1-\g)$,
 Lemma \ref{lemma:operatorEEdiff} implies $\normSplTotal{\EE[\dzHatO]} \leq \frac{C  \de^{\frac13}}{\vabs{\log \de}}$. 
\end{proof}

With this lemma, we can give sharp estimates for $\dzHatU$ by using equation
\eqref{eq:dzOutUInverse}. 
Indeed, since the right hand side of this equation belongs to  $\XSplTotal$, by Lemma \ref{lemma:operatorEEdiff},
\[
\dzHatU(u) = (\Id - \EE )^{-1}\left( \II[c_x-c_x^0, c_y-c_y^0](u) + \EE[\dzHatO](u)\right).
\]
Then, Lemmas \ref{lemma:operatorEEdiff} and \ref{lemma:IconstantsDiff}	imply
\begin{align}\label{def:fitaDeltaphi1}
\normSplTotal{\dzHatU}
\leq \frac{C \de^{\frac{1}{3}}}{\vabs{\log \de}}.
\end{align}
%
%
To prove Theorem \ref{theorem:mainTheoremPoincare}, it only remains to analyze  $\bb_x(u_-)$ and $\bb_y(u_+)$.

\begin{lemma}\label{lemma:boundsBonsBeta}
Let $\kappa_*$ be as given in Lemma \ref{lemma:IconstantsDiff}.
Then, there exists $\de_0>0$ such that, for $\de \in (0,\de_0)$ and $\kappa=\kappa_*\vabs{\log \de}$, the  functions $\bb_{x}, \bb_y$ defined in \eqref{def:mxmyalxaly} satisfy
	\begin{align*}
		\bb_x^{-1}(u_-) &= e^{-\frac{4i}9(\pi-\la_h(u_*))}
		\paren{1+\OO\paren{\frac1{\vabs{\log \de}}}},
		\\	
		\bb_y^{-1}(u_+) &= e^{\frac{4i}9(\pi-\la_h(u_*))}
		\paren{1+\OO\paren{\frac1{\vabs{\log \de}}}},
	\end{align*}
	where  $u_{\pm}=\pm i(A-\kappa\de^2)$.
\end{lemma}

This lemma is proven in Appendix \ref{subappendix:proofH-technicalSecond}.


Let $u_* \in \DBoomerang \cap \reals$.
We compute the first order of $\dzHatO(u_*)=(\dxOutO(u_*),\dyOutO(u_*))^T$.
Since, by Theorem \ref{theorem:singularities}, $(\al_+)^3=(\al_-)^3=\frac12$, and applying Lemma \ref{lemma:boundsBonsBeta} and \eqref{def:dzHatO}, we obtain
\begin{align*}
\vabs{\D x_0(u_*)} =
\vabs{\D y_0(u_*)} =
\sqrt[6]{2}
\vabs{\CInn}
\de^{\frac13} e^{-\frac{A}{\de^2}} 
\paren{1+\OO\paren{\frac1{\vabs{\log \de}}}}.
\end{align*}
%
Moreover, by \eqref{def:fitaDeltaphi1}, 
\begin{align*}
\vabs{\D x(u_*) - \D x_0(u_*)},
\vabs{\D y(u_*) - \D y_0(u_*)} 
\leq 
 \frac{C \de^{\frac13} e^{-\frac{A}{\de^2}}}{\vabs{\log \de}}.
\end{align*}
Finally, notice that the section $u=u_* \in \DBoomerang \cap \reals$  translates to $\la= \la^* :=\la_h(u_*)$ (see \eqref{def:changeOuter}).
Moreover, since $\lad_h=-3\La_h$ (see \eqref{eq:separatrixParametrization}), one deduces that $\La_h(u)>0$ for $u>0$.
Therefore, by the change of coordinates \eqref{def:changeOuter}, Theorem \ref{theorem:existence} and taking $\de$ small enough,
\begin{align*}
	\La_*^{\diamond} = \La_h(u_*) - \frac{\wdOut(u_*)}{3\La_h(u_*)}
	= 
	\La_h(u_*) + \OO(\de^2) > 0,
	\qquad
	\text{with }\diamond=\unstable,\stable, 
\end{align*}
and, therefore using formula \eqref{eq:defDwOut} for $\dwOut$ and Lemma \ref{lemma:boundsBspl}, we obtain that
\[
\vabs{\La_*^{\unstable}-\La_*^{\stable}} \leq 
C \vabs{\dwOut(u_*)} 	
\leq 
C \de \vabs{\dxOut(u_*)} +
C \de \vabs{\dyOut(u_*)}
\leq 
C \de^{\frac43} e^{-\frac{A}{\de^2}}.
\]



\section{The perturbed invariant manifolds}
\label{section:proofH-existence}

In this section, we prove Theorem \ref{theorem:existence} by following the 
scheme detailed in Sections \ref{subsection:outerBasic} and 
\ref{subsection:outerExtension}.


Throughout this section and the following ones, we denote the components of all 
the functions and operators by a numerical sub-index $f=(f_1,f_2,f_3)^T$, unless 
stated otherwise. 
%

\subsection{The invariant manifolds in the infinity domain}
\label{subsection:proofH-existenceInfinite}

The first step is to prove 
Proposition~\ref{proposition:existenceInfty}, which deals with the proof of the existence of 
parameterizations $\zuOut$ and $\zsOut$ satisfying the invariance equation 
\eqref{eq:invariantEquationOuter} and the asymptotic conditions 
\eqref{eq:asymptoticConditionsOuter}.
We only consider the $-\unstable-$ case, being the $-\stable-$ case analogous.

Consider the invariance equation \eqref{eq:invariantEquationOuter},
$\partial_u \zuOut = 	\AAA^{\out} \zuOut + \RRR^{\out}[\zuOut]$,
with $\AAA^{\out}$ and $\RRR^{\out}$ defined in \eqref{def:matrixAAAOuter} and \eqref{def:operatorRRROuter}, respectively.
%
This equation can be written as
\begin{equation}\label{eq:invariantEquationInftyB}
	\LL \zuOut = \RRR^{\out}[\zuOut],
	\qquad \text{with} \qquad
	\LL \phiA = (\partial_u-\AAA^{\out})\phiA.
\end{equation}
In order to obtain a fixed point equation from 
\eqref{eq:invariantEquationInftyB}, we look for a left 
inverse of $\LL$ in a suitable Banach space.
To this end, for a fixed $\rhoInfty>0$ and a  given $\al \in \reals$, we introduce
\[
\XcalInfty_{\al} = \bigg\{ \varphi:\DuInfty  \to \complexs \st \varphi \text{ real-analytic, } 
\normInfty{\varphi}_{\al} := \sup_{u \in \DuInfty} 
|e^{-\al u} \varphi(u)| < \infty \bigg\},
\]
and the product space
$
\XcalInftyTotal= 
\XcalInfty_{2\vap} \times 
\XcalInfty_{\vap} \times 
\XcalInfty_{\vap},
$
with $\nu = \sqrt\frac{21}{8}$
endowed with the weighted product norm
\[
\normInftyTotal{\varphi}= 
\de \normInfty{\varphi_1}_{2\vap} + \normInfty{\varphi_2}_{\vap} + 
\normInfty{\varphi_3}_{\vap}.
\]

Next lemmas, proven in \cite{BFGS12}, give some properties of these Banach 
spaces and provide a left inverse operator of $\mathcal L$. 
\begin{lemma}\label{lemma:sumNormsOutInf}
Let $\al, \beta \in \reals$. Then, the following statements hold:
\begin{enumerate}
	\item If $\al>\beta\geq 0$, then $\XcalInfty_{\al} \subset \XcalInfty_{\beta}$. Moreover
	$
	\normInfty{\varphi}_{\beta} \leq \normInfty{\varphi}_{\al}.
	$
	\item If  $\varphi \in \XcalInfty_{\al}$ and $\zeta \in \XcalInfty_{\beta}$, then
	$\varphi \zeta  \in \XcalInfty_{\al+\beta}$ and
	$
	\normInfty{\varphi\zeta}_{\al+\beta} \leq \normInfty{\varphi}_{\al} \normInfty{\zeta}_{\beta}.
	$
\end{enumerate}
\end{lemma}
%

%

\begin{lemma}\label{lemma:GlipschitzInfinity}
The linear operator $\GG: \XcalInftyTotal \to \XcalInftyTotal$ given by
\[
\GG[\varphi](u) = \left(\int_{-\infty}^u \varphi_1(s) ds, 
\int_{-\infty}^u e^{-\frac{i}{\de^2} (s-u)} \varphi_2(s) ds , 
\int_{-\infty}^u e^{\frac{i}{\de^2} (s-u)} \varphi_3(s) ds \right)^T
\]
is continuous, injective and is a left inverse of the operator $\LL$. 

Moreover, there exists a constant $C$ independent of $\de$ and $\rhoInfty$ such that, for $\varphi \in \XcalInftyTotal$, 
\begin{align*}
\normInftyTotal{\GG[\phiA]} \leq C
\paren{\normInfty{\varphi_1}_{2\vap} +
\de^2 \normInfty{\varphi_2}_{\vap} +
\de^2 \normInfty{\varphi_3}_{\vap}}.
\end{align*}
\end{lemma}

Notice that the eigenvalues of the saddle point $(0,0)$ of $H_{\pend}(\la,\La)$ (see \eqref{def:HpendHosc}) are 
$
\pm \sqrt{\frac{21}{8}}.
$
Then, the parametrization of the separatrix  $\sigma =(\la_h,\La_h)$ (see 
\eqref{eq:separatrixParametrization}) satisfies 
\begin{equation}\label{eq:separatrixBanachSpace}
	\la_h \in \XcalInfty_{\vap}
	\quad \text{ and } \quad
	\La_h \in \XcalInfty_{\vap}.
\end{equation}
Therefore, $\zuOut$ is a solution of \eqref{eq:invariantEquationInftyB} 
satisfying the asymptotic conditions \eqref{eq:asymptoticConditionsOuter} if 
and only if $\zuOut \in \XcalInftyTotal$ and satisfies the fixed point equation
\begin{equation*}
	\phiA = \FF[\phiA] = \GG \circ \RRR^{\out}[\phiA].
\end{equation*}
%
%
Thus, Proposition~\ref{proposition:existenceInfty} is a 
straightforward consequence of the following proposition.

\begin{proposition}\label{proposition:existenceInftyBanach}	
There exists $\de_0>0$ such that, for $\de \in (0,\de_0)$, equation $\phiA = 
\FF[\phiA]$  has a  solution $\zuOut \in \XcalInftyTotal$. 	
Moreover, there exists a real constant $\cttOutInftyA>0$ independent of $\de$ such that
$
\normInftyTotal{\zuOut} \leq \, \cttOutInftyA \de^3.
$
\end{proposition}

To see that $\FF$ is a contractive operator, we have to pay attention to the 
nonlinear terms $\RRR^{\out}$.
%

\begin{lemma}\label{lemma:computationsRRRInf}
Fix $\varrho>0$ and let $\RRR^{\out}$ be the operator defined in~\eqref{def:operatorRRROuter}.
Then, for $\de>0$ small enough\footnote{To simplify the exposition, in this 
lemma and in the  technical lemmas from now on,  we avoid referring 
to the existence of $\de_0$ and just mention that $\delta$ must be small 
enough. We follow the same convention for $\kappa$ whenever is needed.}
and $\normInftyTotal{\varphi}\leq \varrho 
\de^3$, there exists a constant $C>0$ such that
\begin{align*}
\normInfty{\RRR^{\out}_1[\varphi]}_{2\vap} \leq C \de^2, \qquad
\normInftySmall{\RRR^{\out}_j[\varphi]}_{\vap} \leq C \de, \qquad
j= 2,3,
\end{align*}
and
\begin{align*}
\normInfty{\partial_w {\RRR^{\out}_1} [\varphi]}_0 &\leq C \de^2, &
\normInfty{\partial_x {\RRR^{\out}_1} [\varphi]}_{\vap} &\leq C \de, &
\normInfty{\partial_y {\RRR^{\out}_1} [\varphi]}_{\vap} &\leq C \de, \\
\normInftySmall{\partial_w {\RRR^{\out}_j} [\varphi]}_{-\vap} &\leq C \de, &
\normInftySmall{\partial_x {\RRR^{\out}_j} [\varphi]}_0 &\leq C, &
\normInftySmall{\partial_y{\RRR^{\out}_j} [\varphi]}_0 &\leq C, \quad j=2,3.
\end{align*}
\end{lemma}

The proof of this lemma is postponed to 
Appendix~\ref{subsubsection:ProofComputationsRRRInf}.

\begin{proof}[Proof of Proposition~\ref{proposition:existenceInftyBanach}]

Consider the closed ball
\begin{equation*}
	\ballInftyDef= \claus{\varphi \in \XcalInftyTotal \st \normInftyTotal{\varphi} \leq \varrho}.
\end{equation*}
First,  we obtain an estimate for $\FF[0]$. By Lemmas 
\ref{lemma:GlipschitzInfinity} and \ref{lemma:computationsRRRInf}, if $\delta$ 
is small enough,
\begin{align}\label{eq:firstIterationInfty}
\normInftyTotal{\FF[0]} \leq
C \de   \normInfty{\RRR^{\out}_1[0]}_{2\nu} +
C \de^2 \normInfty{\RRR^{\out}_2[0]}_{\nu} +
C \de^2 \normInfty{\RRR^{\out}_3[0]}_{\nu}
\leq \frac12 \cttOutInftyA \de^3,
\end{align}
for some $\cttOutInftyA>0$.

Then, it only remains to check that the operator $\FF$ is contractive in 
$B(\cttOutInftyA\de^3)$.
Let $\phiA, \phiB \in B(\cttOutInftyA\de^3)$.  
Then, by the mean value theorem, 
\begin{align*}
\RRR_j^{\out}[\varphi]-\RRR_j^{\out}[\wt{\varphi}] 
= \boxClaus{\int_0^1 D{\RRR_j^{\out}}[s {\varphi} + (1-s) \wt{\varphi}] d s} ({\varphi} - \wt{\varphi}),
\qquad j=1,2,3.
\end{align*}
Applying Lemmas
\ref{lemma:sumNormsOutInf} and 
\ref{lemma:computationsRRRInf} and the above equality, we obtain
\begin{align*}
\begin{multlined}
\normInfty{\RRR_1^{\out}[{\varphi}]-\RRR_1^{\out}[\wt{\varphi}]}_{2\vap}
\leq
\sup_{\zeta \in B(\cttOutInftyA\de^3)} \Big[
\normInfty{{\varphi}_1-\wt{\varphi}_1}_{2\vap}
 \normInfty{\partial_w{\RRR^{\out}_1}[\zeta]}_{0} \\
+
\normInfty{{\varphi}_2-\wt{\varphi}_2}_{\vap}   
\normInfty{\partial_x{\RRR^{\out}_1}[\zeta]}_{\vap} 
+
\normInfty{{\varphi}_3-\wt{\varphi}_3}_{\vap}
 \normInfty{\partial_y{\RRR^{\out}_1}[\zeta]}_{\vap}
 \Big] 
\leq
C \de \normInftyTotal{{\varphi} -\wt{\varphi}},
\end{multlined}
\\
\begin{multlined}
\normInftySmall{\RRR_j^{\out}[{\varphi}]-\RRR_j^{\out}[\wt{\varphi}]}_{\vap}
\leq
\sup_{\zeta \in B(\cttOutInftyA\de^3)} \Big[
\normInfty{{\varphi}_1-\wt{\varphi}_1}_{2\vap}
\normInftySmall{\partial_w{\RRR^{\out}_j}[\zeta]}_{-\vap} \\ 
+
\normInfty{{\varphi}_2-\wt{\varphi}_2}_{\vap} \normInftySmall{\partial_x{\RRR^{\out}_j}[\zeta]}_{0} 
+
\normInfty{{\varphi}_3-\wt{\varphi}_3}_{\vap} \normInftySmall{\partial_y{\RRR^{\out}_j}[\zeta]}_{0} 
\Big] 
\leq
C \normInftyTotal{{\varphi} -\wt{\varphi}}.
\end{multlined}
\end{align*}
for $j=2,3$. Then, by Lemma~\ref{lemma:GlipschitzInfinity}
and taking $\de$ small enough,
\begin{equation}\label{eq:FFcontractivaInfty}
\begin{split}
\normInftyTotal{\FF[{\varphi}]-\FF[\wt{\varphi}]} \leq& \,
C \de
\normInfty{\RRR_1^{\out}[{\varphi}]-\RRR_1^{\out}[\wt{\varphi}]}_{2\vap}+
C\de^2 \sum_{j=2}^3
\normInftySmall{\RRR_j^{\out}[{\varphi}]-\RRR_j^{\out}[\wt{\varphi}]}_{\vap} \\
%
\leq& \, C\de^2 \normInftyTotal{{\varphi}-\wt{\varphi}}
\leq \frac{1}{2} \normInftyTotal{{\varphi}-\wt{\varphi}}.
\end{split}
\end{equation}
Then, by the definition of $\varrho$ in  \eqref{eq:firstIterationInfty} and 
\eqref{eq:FFcontractivaInfty},
$\FF: B(\cttOutInftyA\de^3) \to B(\cttOutInftyA\de^3)$
is well defined and contractive.
Therefore, $\FF$ has a fixed point $\zuOut\in B(\cttOutInftyA\de^3)$.
\end{proof}

\subsection{The invariant manifolds in the outer domain}
\label{subsection:proofH-existenceBounded}

To prove Proposition~\ref{proposition:existenceComecocos}, we must extend 
analytically the parameterizations  $\zuOut$ and $\zsOut$ given in Proposition 
\ref{proposition:existenceInfty} to the outer domains, $\DuOut$ and $\DsOut$, 
respectively.
Again, we only deal with the unstable -$\unstable$- case, being the -$\stable$- case analogous.
We prove the existence of $\zuOut$ by means of a fixed point argument in a suitable Banach space.

%


Given $\al, \beta \in \reals$, we consider the norm
\begin{align*}
	&\normOut{\varphi}_{\al,\beta}= 
	\sup_{u\in \DuOut} 
	\vabs{g_{\de}^{-\al}(u) \paren{u^2+A^2}^{\beta}  \varphi(u)},
	\quad
	g_{\de}(u) = \frac{1}{\vabs{u^2+A^2}} + \frac{\de^2}{\vabs{u^2+A^2}^{\frac{8}{3}}},
\end{align*}
and the  associated Banach space
\begin{align}\label{def:XcalOut}
&\XcalOut_{\al, \beta}= \left\{ \varphi: \DuOut \to \complexs  \st  \varphi \text{ real-analytic, } \normOut{\varphi}_{\al,\beta}
 < \infty \right\}. 
\end{align}

These Banach spaces have the following properties, which we use without 
mentioning along the section. Their proof follows the same lines as the proof 
of Lemma 7.1 in \cite{BFGS12}.

\begin{lemma}\label{lemma:sumNormsOuter}
The following statements hold:
\begin{enumerate}
\item If $\varphi \in \XcalOut_{\al,\beta_1}$, then $\varphi \in 
\XcalOut_{\al,\beta_2}$ for any $\beta_2 \in \reals$ and
%
\begin{align*}
\begin{cases}
\normOut{\varphi}_{\al,\beta_2} \leq C \normOut{\varphi}_{\al,\beta_1}, 
&\quad \text{for } \beta_2 - \beta_1 > 0, \\
\normOut{\varphi}_{\al,\beta_2} \leq C 
(\kappa \de^2)^{\beta_2-\beta_1} \normOut{\varphi}_{\al,\beta_1},
&\quad \text{for } \beta_2-\beta_1 \leq 0.
\end{cases}
\end{align*}
\item If $\varphi \in \XcalOut_{\al,\beta_1}$, then $\varphi \in 
\XcalOut_{\al-1,\beta_2}$ for any $\beta_2 \in \reals$ and 
%
\begin{align*}
\begin{cases}
\normOut{\varphi}_{\al-1,\beta_2} \leq C \normOut{\varphi}_{\al,\beta_1}, 
&\quad \text{for } \beta_2 - \beta_1 > \frac53, \\[0.5em]
\normOut{\varphi}_{\al-1,\beta_2} \leq 
C \de^2 (\kappa \de^2)^{(\beta_2-\beta_1)-\frac83} 
 \normOut{\varphi}_{\al,\beta_1},
&\quad \text{for } \beta_2 - \beta_1 \leq \frac53.
\end{cases}
\end{align*}
\item If $\varphi \in \XcalOut_{\al_1,\beta_1}$ and 
$\zeta \in \XcalOut_{\al_2,\beta_2}$, then 
$\varphi \zeta \in \XcalOut_{\al_1+\al_2,\beta_1+\beta_2}$ and
\[
\normOut{\varphi\zeta}_{\al_1+\al_2,\beta_1+\beta_2} \leq \normOut{\varphi}_{\al_1,\beta_1} \normOut{\zeta}_{\al_2,\beta_2}.
\]
\item If 
$\varphi \in \XcalOut_{0,\beta+1}$ and $\zeta \in 
\XcalOut_{0,\beta+\frac{8}{3}}$, then
$\varphi + \de^2\zeta \in \XcalOut_{1,\beta}$ and 
%
\begin{align*}
	\normOutSmall{\varphi + \de^2\zeta}_{1,\beta}
	\leq 
	\normOut{\varphi}_{0,\beta+1} +
	\normOut{\zeta}_{0,\beta + \frac{8}{3} }.
\end{align*}
\end{enumerate}
\end{lemma}
%

Let us recall that, by Proposition~\ref{proposition:existenceInfty}, 
the invariance equation~\eqref{eq:invariantEquationOuter}
has a unique solution $\zuOut$ 
in the domain $\DuInfty$ satisfying the asymptotic condition~\eqref{eq:asymptoticConditionsOuter}.
Our objective is to extend analytically $\zuOut$  to the outer domain 
$\DuOut$. Notice that, since $\rhoInfty<\rhoOuter$,   $\DuInfty \cap \DuOut 
\neq \emptyset$ (see definitions \eqref{def:dominisInfty} and 
\eqref{def:dominisComecocos} of $\DuInfty$ and $\DuOut$).

As explained in Section \ref{subsection:proofH-existenceInfinite}, equation \eqref{eq:invariantEquationOuter} is equivalent to  $\LL \zuOut = \RRR^{\out}[\zuOut]$ with $\LL\phiA=(\partial_u-\AAA^{\out})\phiA$ and $\RRR^{\out}$ given in \eqref{def:operatorRRROuter}.
%
%
%
In the following lemma we introduce a right-inverse operator of $\LL$ defined on $\XcalOut_{\al,\beta}$. 

\begin{lemma}\label{lemma:GlipschitzOuter}
Let us consider the operator $\GG[\varphi]=\paren{\GG_1[\phiA_1],\GG_2[\phiA_2],\GG_3[\phiA_3]}^T$, such that
\begin{equation*}
\GG[\varphi](u) = \paren{
\int_{-\rhoOuter}^u \varphi_1(s) ds,
\int_{\conj{u}_1}^u e^{-\frac{i}{\de^2}(s-u) } \varphi_2(s) ds , 
\int_{u_1}^u e^{\frac{i}{\de^2}(s-u) } \varphi_3(s) ds }^T,
\end{equation*}
where $u_1$ and $\conj{u}_1$ are the vertices of the domain $\DuOut$ (see Figure~\ref{fig:dominisComecocos}).
%
Fix $\beta > 0$. There exists a constant $C$ such that:
\begin{enumerate}
\item If $\varphi \in \XcalOut_{1,\beta}$, then $\GG_1[\varphi] \in 
\XcalOut_{1,\beta-1}$ and
$
\normOut{\GG_1[\varphi]}_{1,\beta-1} \leq C \normOut{\varphi}_{1,\beta}$.
\item If $\varphi \in \XcalOut_{0,\beta}$, then $\GG_j[\varphi] \in 
\XcalOut_{0,\beta}$, $j=2,3$, and
$
\normOut{\GG_j[\varphi]}_{0,\beta} \leq C \de^2 \normOut{\varphi}_{0,\beta}$.
%
\end{enumerate}
\end{lemma}

The proof of this lemma follows the same lines as the proof of Lemma 7.3 in 
\cite{BFGS12}.


Consider $u_1$ and $\conj{u_1}$ as in Figure \ref{fig:dominisComecocos} and the function
\begin{equation*}
	F^0(u) = \paren{\wuOut(-\rhoOuter),\,
		\xuOut(\conj{u}_1)
		e^{-\frac{i}{\de^2}(\conj{u}_1-u)}, \,
		\yuOut({u_1})
		e^{\frac{i}{\de^2}({u}_1-u)}
	}^T.
\end{equation*}
Notice that, since $0<\rhoInfty<\rhoOuter$, we have
$\claus{-\rhoOuter, u_1,\conj{u}_1} \in \DuInfty$.
Therefore, by Proposition~\ref{proposition:existenceInfty}, 
$\zuOut$ is already defined at these points.
%
%
We define the fixed point operator
\begin{equation}\label{eq:operatorFFouter}
\FF[\phiA]= F^0 + \GG \circ \RRR^{\out}[\phiA],
\end{equation}
where the operator $\RRR^{\out}$ is given in~\eqref{def:operatorRRROuter}.
Since $\LL(F^0)=0$, 
by Lemma~\ref{lemma:GlipschitzOuter}, a solution $\zuOut=\FF[\zuOut]$ satisfies $\LL \zuOut = \RRR^{\out}[\zuOut]$ and by construction is the real-analytic continuation of the function $\zuOut$ obtained in Proposition \ref{proposition:existenceInfty}.

We rewrite Proposition~\ref{proposition:existenceComecocos} in terms of the operator $\FF$ defined in the Banach space
\[
\XcalOutTotal = \XcalOut_{1,0} \times \XcalOut_{0,\frac{4}{3}} \times \XcalOut_{0,\frac{4}{3}},
\] 
endowed with the norm
\[
\normOutTotal{\varphi}= 
\de \normOut{\varphi_1}_{1,0} + 
\normOut{\varphi_2}_{0,\frac{4}{3}} + 
\normOut{\varphi_3}_{0,\frac{4}{3}}.
\]

\begin{proposition}\label{proposition:existenceComecocosBanach}
There exist $\de_0, \kappaOuter>0$ such that, for $\de \in (0,\de_0)$ and $\kappa\geq \kappaOuter$, the fixed point equation $\zuOut = {\FF}[\zuOut]$ has a unique solution $\zuOut \in \XcalOutTotal$. 
Moreover, there exists a real constant $\cttOutOutA>0$  independent of $\de$ and $\kappa$ such that
$
\normOutTotal{\zuOut} \leq \cttOutOutA \de^3.
$
\end{proposition}

We prove this proposition through a 
fixed point argument.
First, we state a technical lemma, whose proof is postponed until 
Appendix~\ref{subsubsection:proofComputationsRRROut}. Fix $\varrho>0$ and define
\begin{equation*}
\ballOuterDef= \claus{\varphi \in \XcalOutTotal \st \normOutTotal{\varphi} \leq \varrho}.
\end{equation*}

\begin{lemma}\label{lemma:computationsRRROut}
Fix $\varrho>0$ and let $\RRR^{\out}$ be the operator defined in~\eqref{def:operatorRRROuter}.
For $\de>0$ small enough and $\kappa>0$ big enough, 
there exists a constant $C>0$
such that, for $\varphi \in \ballOuter$,
\begin{align*}
\normOut{\RRR^{\out}_1[\phiA]}_{1,1} \leq C \de^2, \qquad
\normOutSmall{\RRR^{\out}_j[\phiA]}_{0,\frac{4}{3}} \leq C \de, \qquad
j= 2,3,
\end{align*}
and,
\begin{align*}
&\normOut{\partial_w {\RRR^{\out}_1} [\phiA]}_{1,\frac13} 
\leq C \de^2,  &
&\normOut{\partial_x {\RRR^{\out}_1} [\phiA]}_{0,\frac73} 
\leq C\de, &
&\normOut{\partial_y {\RRR^{\out}_1} [\phiA]}_{0,\frac{7}{3}} 
\leq C{\de}, \\
&\,\normOutSmall{\partial_w {\RRR^{\out}_2} [\phiA]}_{0,\frac{2}{3}} 
\leq C\de, &
&\,\normOutSmall{\partial_x {\RRR^{\out}_2} [\phiA]}_{1,-\frac23} 
\leq C, &
&\,\normOutSmall{\partial_y{\RRR^{\out}_2} [\phiA]}_{0,2} 
\leq C\de^2, \\
&\,\normOutSmall{\partial_w {\RRR^{\out}_3} [\phiA]}_{0,\frac{2}{3}} 
\leq C\de, &
&\,\normOutSmall{\partial_x {\RRR^{\out}_3} [\phiA]}_{0,2} 
\leq C\de^2, &
&\,\normOutSmall{\partial_y {\RRR^{\out}_3} [\phiA]}_{1,-\frac23} 
\leq C.
\end{align*}

\end{lemma}

The next lemma gives properties of the operator $\FF$.
%

\begin{lemma}\label{lemma:FpropertiesComecococs}
Fix $\varrho>0$ and let $\FF$ be the operator defined in~\eqref{eq:operatorFFouter}.
Then, for $\de>0$ small enough and $\kappa>0$ big enough,  there exist 
constants $\cttOutOutB, \cttOutOutC>0$ independent of $\de$ and $\kappa$ such 
that
\[
\normOutTotal{\FF[0]} \leq \cttOutOutB \de^3.
\]
Moreover,
for $\phiA, \phiB \in \ballOuter$,
\begin{align*}
\de \normOut{{\FF_1}[\phiA]-{\FF_1}[\phiB]}_{1,0}
&\leq \cttOutOutC \paren{ 
\frac{\de}{\kappa^2} 
\normOut{\phiA_1 - \phiB_1}_{1,0} 
+ \normOut{\phiA_2 - \phiB_2}_{0,\frac{4}{3}}
+ \normOut{\phiA_3 - \phiB_3}_{0,\frac{4}{3}} }, \\
\normOut{{\FF_j}[\phiA]-{\FF_j}[\phiB]}_{0,\frac{4}{3}} 
&\leq \frac{\cttOutOutC}{\kappa^2} \normOutTotal{\phiA - \phiB},
\qquad \text{for} \quad j=2,3.
\end{align*}

\end{lemma}

\begin{proof}
First, we obtain the estimates for $\FF[0]$. By 
Proposition~\ref{proposition:existenceInfty}, we have that
\begin{align*}
	\vabs{\wuOut(-\rhoOuter)} \leq C \de^2,
	\quad
	\vabs{\xuOut(\conj{u_1})} \leq C \de^3,
	\quad
	\vabs{\yuOut({u_1})} \leq C \de^3,
\end{align*}
and, as a result, 
$\normOutTotalSmall{F^0} \leq C \de^3$.
Then, applying Lemmas~\ref{lemma:GlipschitzOuter} and~\ref{lemma:computationsRRROut}, we obtain
\begin{align*}
\normOutTotal{\FF[0]} \leq
\normOutTotal{F^0} +
C \de \normOut{\RRR^{\out}_1[0]}_{1,1} +
C \de^2 \textstyle \sum_{j=2}^3 \normOutSmall{\RRR^{\out}_j[0]}_{0,\frac43} 
\leq C \de^3.
\end{align*}

For the second statement, since $\FF=F^0+\GG \circ \RRR^{\out}$ and $\GG$ is linear, we need to compute estimates for $\RRR^{\out}[\phiA]-\RRR^{\out}[\phiB]$.
Then, by the mean value theorem, 
\begin{align*}
	\RRR_j^{\out}[\phiA]-\RRR_j^{\out}[\phiB] 
	= \boxClaus{\int_0^1 D{\RRR_j^{\out}}[s {\phiA} + (1-s) \phiB] d s} (\phiA - \phiB),
	\qquad j=1,2,3.
\end{align*}	
In addition, by Lemmas \ref{lemma:sumNormsOuter} and \ref{lemma:computationsRRROut}, for $j=2,3$, we have the estimates
\begin{equation*}
\begin{aligned}	
\normOut{\partial_w {\RRR^{\out}_1} [\phiA]}_{0,1} 
&\leq \frac{C}{\kappa^2}, &
\normOut{\partial_x {\RRR^{\out}_1} [\phiA]}_{1,-\frac{1}{3}} 
&\leq \frac{C}{\de}, &
\normOut{\partial_y {\RRR^{\out}_1} [\phiA]}_{1,-\frac{1}{3}} 
&\leq \frac{C}{\de}, \\
\normOutSmall{\partial_w {\RRR^{\out}_j} [\phiA]}_{-1,\frac{4}{3}} 
&\leq \frac{C}{\kappa^2 \de}, &
\normOutSmall{\partial_x {\RRR^{\out}_j} [\phiA]}_{0,0} 
&\leq \frac{C}{\kappa^2 \de^2}, &
\normOutSmall{\partial_y{\RRR^{\out}_j} [\phiA]}_{0,0} 
&\leq \frac{C}{\kappa^2 \de^2}.
\end{aligned}
\end{equation*}
We estimate each component separately.
For $j=1$, we have that
\begin{align*}	
\MoveEqLeft[5] \de \lVert \RRR_1^{\out}[\phiA]-	\RRR_1^{\out}[\phiB] \rVert^{\mathrm{out}}_{1,1}
\leq
\sup_{\zeta \in \ballOuter} \de \Big[
\normOut{\phiA_1-\phiB_1}_{1,0}
\normOut{\partial_w{\RRR^{\out}_1}[\zeta]}_{0,1} \\
&+
\normOut{\phiA_2-\phiB_2}_{0,\frac{4}{3}}  \normOut{\partial_x{\RRR^{\out}_1}[\zeta]}_{1,-\frac{1}{3}}  
+
\normOut{\phiA_3-\phiB_3}_{0,\frac{4}{3}} \normOut{\partial_y{\RRR^{\out}_1}[\zeta]}_{1,-\frac{1}{3}} \Big] 
\\
\leq& \,
\frac{C \de}{\kappa^2} 
\normOut{{\phiA}_1-{\phiB}_1}_{1,0} 
+		
C \normOut{{\phiA}_2 -{\phiB}_2}_{0,\frac{4}{3}}
+
C \normOut{{\phiA}_3 -{\phiB}_3}_{0,\frac{4}{3}}.
\end{align*}
Analogously, for $j=2,3$, we obtain
\begin{align*}
\MoveEqLeft[7]
\normOutSmall{\RRR_j^{\out}[\phiA]-\RRR_j^{\out}[\phiB]}_{0,\frac{4}{3}}
\leq
\sup_{\zeta \in \ballOuter} \Big[
\normOut{\phiA_1-\phiB_1}_{1,0}
\normOutSmall{\partial_w{\RRR^{\out}_j}[\zeta]}_{-1,\frac{4}{3}} \\
&+
\normOut{\phiA_2-\phiB_2}_{0,\frac{4}{3}} \normOutSmall{\partial_x{\RRR^{\out}_j}[\zeta]}_{0,0}
+
\normOut{\phiA_3-\phiB_3}_{0,\frac{4}{3}}
\normOutSmall{\partial_y{\RRR^{\out}_j}[\zeta]}_{0,0} \Big] \\
\leq& \,
\frac{C}{\kappa^2 \de^2}
\normOutTotal{{\phiA} -{\phiB}},
\end{align*}
and, using Lemma~\ref{lemma:GlipschitzOuter}, we obtain the estimates for the 
second statement. 
\end{proof}

Lemma~\ref{lemma:FpropertiesComecococs} shows that, by assuming $\kappa$ big enough, operators $\FF_2$ and $\FF_3$ have Lipschitz constant less than $1$ with the norm in $\XcalOutTotal$.
However, we are not able to control the Lipschitz constant of $\FF_1$.
To overcome this problem, we apply a Gauss-Seidel argument to define a new operator
\begin{equation*} 
\wt{\FF}[\zOut]
= \wt{\FF}[(\wOut,\xOut,\yOut)]
= \begin{pmatrix} 
\FF_1[\wOut,\FF_2[\zOut],\FF_3[\zOut]] \\ 
\FF_2[\zOut] \\
\FF_3[\zOut] \end{pmatrix},
\end{equation*}
which turns out to be contractive in a suitable ball and has the same fixed points as $\FF$.

\begin{proof}[End of the proof of Proposition~\ref{proposition:existenceComecocosBanach}]
We look for a fixed point of $\wt{\FF}$.
First, we obtain an estimate for $\normOutTotalSmall{\wt{\FF}[0]}$. 
We rewrite it as
\begin{align*}
\wt{\FF}[0] = \FF[0] + 
\Big(\FF_1[0,\FF_2[0],\FF_3[0]] - \FF_1[0],
0,0\Big)^T,
\end{align*}
and we notice that, by Lemma~\ref{lemma:FpropertiesComecococs}, 
$
\normOutTotalSmall{(0,\FF_2[0],\FF_3[0])} \leq
\normOutTotal{\FF[0]} \leq C \de^3.
$
Then, applying Lemma~\ref{lemma:FpropertiesComecococs}, there exists constant $ 
\cttOutOutA>0$ such that
\begin{equation}\label{proof:operadorJJ0}
\begin{split}	
\normOutTotalSmall{\wt{\FF}[0]} &\leq 
\normOutTotal{\FF[0]} + 
\normOut{\FF_1[0,\FF_2[0],\FF_3[0]] - \FF_1[0]}_{1,0} \\
&\leq 
\normOutTotal{\FF[0]} +
C \normOut{\FF_2[0]}_{0,\frac{4}{3}} + 
C \normOut{\FF_3[0]}_{0,\frac{4}{3}}
\leq \frac12  \cttOutOutA \de^3.
\end{split}
\end{equation}
%

Now, we prove that the operator $\wt{\FF}$ is contractive in $B( \cttOutOutA 
\delta^3)$.
Indeed, by Lemma~\ref{lemma:FpropertiesComecococs}, we have that, for $\phiA, 
\phiB \in B( \cttOutOutA 
\delta^3)$,
\begin{align*}
\de \normOutSmall{\wt{\FF}_1[\phiA]-\wt{\FF}_1[\phiB]}_{1,0} &\leq C \paren{
\frac{\de}{\kappa^2} 
\normOut{\phiA_1 - \phiB_1}_{1,0}
+
\normOut{\FF_2[\phiA]-\FF_2[\phiB]}_{0,\frac{4}{3}}
+
\normOut{\FF_3[\phiA]-\FF_3[\phiB]}_{0,\frac{4}{3}}
} \\
&\leq  
\frac{C \de}{\kappa^2} 
\normOut{\phiA_1 - \phiB_1}_{1,0} + 
\frac{2 C}{\kappa^2} 
\normOutTotal{\phiA-\phiB}
\leq 
\frac{C}{\kappa^2} 
\normOutTotal{\phiA-\phiB}, \\
\normOutSmall{\wt{\FF}_j[\phiA]-\wt{\FF}_j[\phiB]}_{0,\frac{4}{3}}
&= 
\normOut{\FF_j[\phiA]-\FF_j[\phiB]}_{0,\frac{4}{3}}
\leq
\frac{C}{\kappa^2} 
\normOutTotal{\phiA-\phiB}, 
\qquad \text{for } j=2,3.
\end{align*}
Then, for $\kappa>0$ big enough, we have that
$
\normOutTotalSmall{\wt{\FF}[\varphi]-\wt{\FF}[\phiB]} \leq \frac{1}{2} \normOutTotal{\phiA-\phiB}.
$
Together with \eqref{proof:operadorJJ0}, this implies that 
$\wt{\FF}:B( \cttOutOutA 
\delta^3) \to B( \cttOutOutA 
\delta^3)$ is well defined and contractive.
Therefore, $\wt{\FF}$ has a fixed point $\zuOut \in B( \cttOutOutA 
\delta^3)$.
\end{proof}

\subsection{Switching to the time-parametrization}
\label{subsection:proofH-changeuOut}

In this section, by means of a fixed point argument, we prove Proposition 
\ref{proposition:changeuOut}. That is, we obtain a change of 
variables $\uOut$ satisfying \eqref{eq:equationuOutB}, that is
\begin{equation}\label{eq:InvU}
\partial_v \uOut = R[\uOut]\quad \text{ where }\quad 
  R[\uOut]= \partial_w H_1^{\out}
	\paren{v + \uOut(v), \zuOut(v+\uOut(v))}.
\end{equation}
To this end, we consider the Banach space
\begin{equation}\label{def:Yout}
\YcalOuter = \bigg\{\phiA: \DOuterTilde \to \complexs \st \phiA \text{ real-analytic, }
\normSup{\phiA}:=\sup_{v \in \DOuterTilde} \vabs{\uOut(v)}<\infty \bigg\}.
\end{equation}
%

%

First, we state a technical lemma. Its proof is a  direct consequence of the 
proof of Lemma \ref{lemma:computationsRRROut} (see also Remark \ref{rmk:R} 
in Appendix~\ref{subsubsection:proofComputationsRRROut}).
%

\begin{lemma}\label{lemma:computationsRRRtransitionOuter}
Fix $\varrho>0$.
%
For $\de>0$ small enough and $\varphi \in \YcalOuter$ such that 
$\normSup{\varphi} \leq \varrho \de^2$, 
there exists a constant $C>0$  such that
		$\normSup{R[\varphi]}\leq C \de^2$ and  
		$\normSup{D R[\varphi]} \leq C \de^2$. 
\end{lemma}

%

Let us define the operators
\begin{equation}\label{def:operatorGGtransitionFlow}
G[\varphi](v) = 
\int_{\rhoOuterTilde}^v \varphi(s) ds,
\qquad \text{and} \qquad
F = G \circ R,
\end{equation}
where $\rhoOuterTilde \in \reals$ is the rightmost vertex of the domain $\DOuterTilde$ (see Figure~\ref{fig:dominisTransition}).
Then, a solution $\uOut = F[\uOut]$  satisfies  equation 
\eqref{eq:InvU} and the initial condition $\uOut(\rhoOuterTilde)=0$.


\begin{proof}[Proof of Proposition \ref{proposition:changeuOut}]

The  operator $G$ in \eqref{def:operatorGGtransitionFlow} satisfies that, for 
$\varphi \in \YcalOuter$, 
\begin{equation}\label{eq:GlipschitzTransitionFlow}
\normSup{G[\phiA]}\leq C \normSup{\phiA}.
\end{equation}
Then, by Lemma
\ref{lemma:computationsRRRtransitionOuter}, 	
there exists $\cttOuterD>0$ independent of $\de$ such that
\begin{align}\label{proof:firstIterationTransitionOuter}
	\normSup{F[0]} \leq C \normSup{R[0]} 	\leq \frac12 \cttOuterD \de^2.
\end{align}
Moreover, for $\phiA,\phiB \in B( \cttOuterD \de^2)
= \claus{\varphi \in 
\YcalOuter \st \normSup{\varphi} \leq \cttOuterD \de^2}$,
by  the mean value theorem and Lemma 
\ref{lemma:computationsRRRtransitionOuter},  
\begin{align*}
	\normSup{R[\phiA]-R[\phiB]} = 
	\normSupBig{\int_0^1 DR[s\phiA+(1-s)\phiB]ds} \normSup{\phiA-\phiB}
	\leq C \de^2 
	\normSup{\phiA-\phiB}.
\end{align*}
Then, by Lemma \ref{lemma:computationsRRRtransitionOuter}, 
\eqref{eq:GlipschitzTransitionFlow},  
\eqref{proof:firstIterationTransitionOuter}
and taking $\de$ small enough, 
$F$ is well defined and contractive in $B( \cttOuterD \de^2)$ and, as a result, 
has a fixed point $\uOut \in B( \cttOuterD \de^2)$.

It only remains to check that 
$v+\uOut(v) \in \DuOutKappa$ for $v \in \DOuterTilde$.
Indeed, since $\normSup{\uOut}  \leq  \cttOuterD\de^2$ and ${\DOuterTilde} 
\subset \DuOut$, taking $\de$ small enough the statement is proved.

\end{proof}	

	

\subsection{Extending the time-parametrization}
\label{subsection:proofH-existenceFlow}

In this section, we  extend analytically the parametrization $\Gu$ given in Corollary \ref{corollary:changeuOut} from the transition domain $\DuOut$ to the flow domain  $\DFlow$ (see \eqref{def:dominiFlow}).

Since $\Gu$ satisfies the equations given by $H$ in \eqref{def:hamiltonianScaling}, $\Gh=\Gu-\G_h$ (see \eqref{eq:splitParametrizationGu}) satisfies
\begin{align*}
\left\{
\begin{aligned}	
\partial_v \lah &= -3\Lah 
+ \partial_{\La} H_1 (\G_h + \Gh;\de), \\
\partial_v \Lah &= 
-V'(\la_h+\lah)
+V'(\la_h)
- \partial_{\la} H_1 (\G_h + \Gh;\de), \\
\partial_v \xh &= 
i \frac{\xh}{\de^2}
+ i\partial_{y} H_1 (\G_h + \Gh;\de), \\
\partial_v \yh &= 
-i \frac{\yh}{\de^2}
- i\partial_{x} H_1 (\G_h + \Gh;\de),
\end{aligned}
\right.
\end{align*} 
%
which can be rewritten as
 $\LL^{\flow} \, \Gh = \RRR^{\flow}[\Gh]$,
where
\begin{equation}\label{def:operatorLLFlow}
\LL^{\flow} \phiA = 
\paren{\partial_v - \AAA^{\flow}(v)} \phiA,
\qquad
\AAA^{\flow}(v) = \begin{pmatrix}
	0 & -3 & 0 & 0 \\
	-V''(\la_h(v)) & 0 & 0 & 0 \\
	0 & 0 & \frac{i}{\de^2} & 0  \\
	0 & 0 & 0 & -\frac{i}{\de^2} 
\end{pmatrix},
\end{equation}
and
\begin{equation}\label{def:operatorRFlow}
\RRR^{\flow}[\phiA](v) = 
\begin{pmatrix}
	\partial_{\La} H_1 (\G_h(v)+\phiA(v);\de) \\ 
	T[\phiA_1](v)-\partial_{\la} H_1 (\G_h(v)+\phiA(v);\de) 
	\\
	i\partial_{y} H_1 (\G_h(v)+\phiA(v);\de) 
	\\
	-i\partial_{x} H_1 (\G_h(v)+\phiA(v);\de) \\
\end{pmatrix},	
\end{equation}
with
$T[\phiA_1] = -V'(\la_h+\phiA_1) +V'(\la_h)
+ V''(\la_h)\phiA_1 $.
%
%
%

We look for $\Gh$ through fixed point argument in the Banach space $\XcalFlowTotal = \paren{\XcalFlow}^4
$, where 
%
\begin{align*}
\XcalFlow = \bigg\{\phiA: \DFlow \to \complexs \st
 \phiA \text{ real-analytic, }
 \normFlow{\phiA}:=\sup_{v\in\DFlow}\vabs{\phiA(v)} < \infty \bigg\},
\end{align*}
endowed with the norm
\[
\normFlowTotal{\phiA} = \de \normFlow{\phiA_1}
+ \de \normFlow{\phiA_2}
+ \normFlow{\phiA_3} + \normFlow{\phiA_4}.
\]
%
%
A fundamental matrix of the linear equation $\dot{\xi} = \AAA^{\flow}(v) \xi$ is
\begin{align*}
\MoveEqLeft 	
\Phi(v) = \begin{pmatrix}
	3\La_h(v) & 3f_h(v) & 0 & 0  \\
	-\dot{\La}_h(v) & -\dot{f}_h(v) & 0 & 0 \\
	0 & 0 & e^{\frac{i}{\de^2} v} & 0 \\
	0 & 0 & 0 & e^{-\frac{i}{\de^2} v}
\end{pmatrix}
\, \text{with } \,
f_h(v) = \La_h(v) \int_{v_0}^{v} \frac{1}{\La_h^2(s)} ds.
\end{align*}
Note that $f_h(v)$ is analytic at $v=0$.

To look for a right inverse of operator $\LL^{\flow}$ in \eqref{def:operatorLLFlow}, 
let us consider the linear operator
\begin{align*}
\GG^{\flow}[\phiA](v) = \paren{
	\int_{v_0}^v \phiA_1(s) ds,
	\int_{v_0}^v \phiA_2(s) ds,
	\int_{\conj{v}_1}^v \phiA_3(s) ds,
	\int_{{v}_1}^v \phiA_4(s) ds
}^T,
\end{align*}
where $v_0$, $v_1$ and $\conj{v}_1$ are the vertexs of the domain $\DFlow$ 
(see Figure \ref{fig:dominisTransition}).
Then, the linear operator
$\GGh[\phiA] =  \Phi \GG[\Phi^{-1}\phiA]$
is a right inverse of the operator $\LL^{\flow}$,
and, 
for $\phiA \in \XcalFlowTotal$, satisfies 
\begin{align}\label{eq:operatorGGhFlow}
\normFlow{\GGh_1[\phiA]} + \normFlow{\GGh_2[\phiA]} \leq C \paren{
\normFlow{\phiA_1} + \normFlow{\phiA_2}},
\quad
\normFlow{\GGh_j[\phiA]} \leq C\de^2\normFlow{\phiA_j} \,
\text{for } j=3,4.	
\end{align}

Next, we state a technical lemma providing estimates for $\RRR^{\flow}$.
%
%
Its proof is a direct consequence of the definition of the operator in \eqref{def:operatorRFlow} and Corollary \ref{corollary:seriesH1Poi}, which gives estimates for $H_1^{\Poi}$ in \eqref{def:hamiltonianPoincare} (see also the change of coordinates \eqref{def:changeScaling} which relates  $H_1^{\Poi}$ and $H_1$).
%

\begin{lemma}\label{lemma:computationsRFlow}
Fix $\varrho>0$ and consider $\phiA \in \XcalFlowTotal$ with $\normFlowTotal{\phiA} \leq \varrho \de^3$. Then, for $\de>0$ small enough , there exists a constant $C>0$  such that  the  operator $\RRR^{\flow}$  in \eqref{def:operatorRFlow} satisfies
	\begin{align*}
		\normFlow{\RRR^{\flow}_1[\phiA]},\normFlow{\RRR^{\flow}_2[\phiA]} &\leq C \de^2,
		 \qquad
		\normFlow{\RRR^{\flow}_3[\phiA]},\normFlow{\RRR^{\flow}_4[\phiA]} \leq C \de,\\
		\normFlow{D_j \RRR^{\flow}_l[\phiA]} &\leq C \de, 
		\qquad
		j,l \in \claus{1,2,3,4}.
	\end{align*}
\end{lemma}

Denote by $\mathbf{e}_j$, $j=1,2,3,4$, the canonical basis in $\mathbb{R}^4$. Noticing that, by Corollary~\ref{corollary:changeuOut}, the function $\Gh=(\lah,\Lah,\xh,\yh)$ is already defined at $\claus{v_0,v_1,\conj{v}_1} \in \DOuterTilde$, we can consider the function
\[
\begin{split}
	F^0(v) = \Phi(v)  \left[
	\Phi^{-1}(v_0)\left(
	\lah(v_0)\mathbf{e}_1+ \Lah(v_0)\mathbf{e}_2\right)
	+\xh(\conj{v}_1) \Phi^{-1}(\conj{v}_1) \mathbf{e}_3
	+ \yh({v_1})\Phi^{-1}({v_1}) 
	\mathbf{e}_4
	\right].
\end{split}
\]
Then, since $\GGh(F^0)=0$, 
 it only remains to check that
$\FF = F^0 + \GGh \circ \RRR^{\flow}$
is contractive in a suitable ball of $\XcalFlowTotal$.

%
%
%


\begin{proof}[End of proof of Proposition~\ref{proposition:existenceFlow}]
First, we obtain a suitable estimate for $\FF[0]$. 
Applying  Propositions \ref{proposition:existenceComecocos} and  \ref{proposition:changeuOut} and using  \eqref{eq:equationuOutA} we obtain that,
for $v \in \DOuterTilde$,
\begin{align*}
	\vabsSmall{\lah(v)} \leq C \de^2, \qquad
	\vabsSmall{\Lah(v)} \leq C \de^2, \qquad
	\vabsSmall{\xh(v)} \leq C \de^3, \qquad
	\vabsSmall{\yh(v)} \leq C \de^3.
\end{align*}
Therefore, since $\claus{v_0,v_1,\conj{v}_1} \in \DOuterTilde$,
\begin{align*}
	\normFlowTotal{F^0} &\leq
	C\de \vabsSmall{\lah(v_0)}+
	C\de \vabsSmall{\Lah(v_0)}+
	C \vabs{\xh(\conj{v}_1)}+
	C \vabs{\yh({v_1})} \leq C \de^3,
\end{align*}
%
%
%
and, applying \eqref{eq:operatorGGhFlow} and Lemma \ref{lemma:computationsRFlow}, there exists $\cttOuterF>0$ independent of $\de$ such that 
\begin{align}\label{proof:operadorFF0Flow}
	\normFlowTotal{\FF[0]} \leq 
	\normFlowTotal{F^0} +
	\normFlowTotal{\GG \circ \RRR^{\flow}[0]} \leq \frac12 \cttOuterF \de^3.
\end{align}
%
Let us define $B(\cttOuterF \de^3)= \claus{\phiA \in \XcalFlowTotal \st \normFlowTotal{\phiA} \leq \cttOuterF \de^3}$.
%
%
By the mean value theorem and Lemma~\ref{lemma:computationsRFlow}, for $\phiA, \phiB \in B(\cttOuterF \de^3)$ and $j=1,..,4$, we obtain
%
\begin{align*}
\normFlow{\RRR^{\flow}_j[\phiA]-\RRR^{\flow}_j[\phiB]}&\leq
\sum_{l=1}^4 \boxClaus{
\sup_{\zeta \in B(\cttOuterF \de^3)} \claus{
\normFlow{D_l \RRR^{\flow}_j[\zeta]} }
\normFlow{\phiA_l-\phiB_l}
} 
\leq C  \normFlowTotal{\phiA-\phiB}.
\end{align*}
Then, by \eqref{eq:operatorGGhFlow} and taking $\de$ small enough,  
\begin{equation}\label{proof:operadorFFLipschitzFlow}
\begin{split}
\normFlowTotal{\FF[\phiA]-\FF[\phiB]}
\leq& \, 
C\de \boxClaus{ \sum_{j=1}^2 \normFlow{\RRR^{\flow}_j[\phiA]-\RRR^{\flow}_j[\phiB]}}
+ C\de^2 \boxClaus{ \sum_{l=3}^4 \normFlow{\RRR^{\flow}_l[\phiA]-\RRR^{\flow}_l[\phiB]}}
\\
\leq& \, 
C \de\normFlowTotal{\phiA-\phiB} \leq
\frac{1}{2}\normFlowTotal{\phiA-\phiB}.
\end{split}
\end{equation}

Therefore, by \eqref{proof:operadorFF0Flow} and
\eqref{proof:operadorFFLipschitzFlow},
$\FF$ is well defined and contractive in $B(\cttOuterF \de^3)$ and, as a result, has a fixed point $\Gh \in B(\cttOuterF \de^3)$.
\end{proof}

\subsection{Back to a graph parametrization}
\label{subsection:proofH-changevOut}

Now we prove Proposition \ref{proposition:changevOut} by obtaining the change of variables $\vOut: \DBoomerangTildeProp \to \complexs$  as a solution of equation \eqref{eq:equationsvOutA}. This equation is equivalent to  $\vOut=\NNN[\vOut]$  with
%
%
\begin{equation*}
\NNN[\phiA](u) =
\frac{1}{3\La_h(u)}
\left[
\lah(u+\phiA(u)) + 
\la_h(u+\phiA(u)) -
\la_h(u)
+3\La_h(u) \phiA(u)
\right].
\end{equation*}
We obtain $\vOut$ by means of a fixed point argument in the Banach space
\begin{align*}
\YcalBoom = \bigg\{\phiA: \DBoomerangTildeProp \to \complexs \st \phiA \text{ real-analytic, }
\normSup{\phiA} := \sup_{u\in\DBoomerangTildeProp} \vabs{\phiA(u)}<\infty\bigg\}.
\end{align*}
%


\begin{proof}[Proof of Proposition \ref{proposition:changevOut}]
Let us first notice  that, by Theorem \ref{theorem:singularities},
\begin{align}\label{proof:boundsHomoclinicVchange}
C^{-1} \leq \normSup{\La_h}  \leq C.
\end{align}

%
Since $\dBoomerangTilde<\dFlow$ and
$\kappaBoomerangTilde>\kappaFlow$, we have that
$
{\DBoomerangTildeProp} \subset \DFlow,
$
(see \eqref{def:domainBoomerangTilde} and \eqref{def:dominiFlow}).	
Then,  applying Proposition \ref{proposition:existenceFlow},
 there exists $b_6>0$ independent of $\de$ such that
\begin{align*}
\normSup{\NNN[0]} 
\leq \frac13 \normSup{(\La_h)^{-1}} \normSup{\lah} \leq \frac12 b_6 \de^2.
\end{align*} 	
Next, we compute the Lipschitz constant of $\NNN$ in $B(b_6\delta^2) = \{\phiA \in \YcalBoom \st \normSup{\phiA}\leq b_6 \de^2\big\}$.
By the mean value theorem, for $\phiA,\phiB \in B(b_6\delta^2)$ and $\phiA_s=(1-s)\phiA+s\phiB$, we have that
\begin{align*}
\normSup{\NNN[\phiA]-\NNN[\phiB]} 
\leq&
\sup_{u \in \DBoomerangTildeProp}
\vabs{\int_0^1 D\NNN[\phiA_s](u) ds} 
\normSup{\phiA-\phiB}.	
\end{align*}
For $u \in \DBoomerangTildeProp$ and $\de$ small enough, we have that $u+\phiA_s(u) \in \DFlow$.
Therefore, by Proposition \ref{proposition:existenceFlow}, \eqref{proof:boundsHomoclinicVchange}
and recalling that $\dot{\la}_h=-3\La_h$,
\begin{align*}
\vabs{D\NNN[\phiA_s](u)}
&\leq
\frac{1}{3\vabs{\La_h(u)}}
\claus{
|\partial_v \lah(u + \phiA_s(u))|
+
\vabs{\La_h(u + \phiA_s(u))-\La_h(u)}
} 
%
\leq C\de^2,
\end{align*}
and, taking $\de$ small enough,
$\normSup{\NNN[\phiA]-\NNN[\phiB]}\leq \frac{1}{2}  \normSup{\phiA - \phiB}$.
%
%
Therefore, the operator $\NNN$ is well defined and contractive in $B(b_6\delta^2)$ and, as a result, has a fixed point $\vOut \in B(b_6\delta^2)$.

Besides, since 
$
{\DBoomerangTildeProp} \subset \DFlow,
$
we obtain that $u+\vOut(u) \in \DFlow$ for $u \in \DBoomerangTildeProp$ and $\de$ small enough.

%
\end{proof}

\section{Complex matching estimates}
\label{section:proofG-matching}

This section is devoted to prove Theorem \ref{theorem:matching} which provides estimates for $\ZusMchU= \ZusMchO - \ZusInn$
in the matching domains $\DuMchInn$ and $\DsMchInn$, given in \eqref{def:matchingDomainInner}.
We only prove the theorem for $\ZuMchU$, being the proof for $\ZsMchU$ analogous.

%
%
%

\subsection{Preliminaries and set up}

Proposition \eqref{proposition:innerDerivation} shows that the Hamiltonian $H^{\out}$ expressed in inner coordinates, that is $H^{\inn}$ as given in \eqref{def:hamiltonianInnerComplete},
is of the form $H^{\inn}=W+XY+ \KK+H_1^{\inn}$.
Then, the equation associated to $H^{\inn}$ can be written as
\begin{equation}\label{eq:systemEDOsInnerComplete}
	\left\{ \begin{array}{l}
		\dot{U} = 1 + g^{\Inner}(U,Z) + g^{\mch}(U,Z),\\
		\dot{Z} = \AAA^{\Inn} Z + f^{\Inner}(U,Z) +  f^{\mch}(U,Z),
	\end{array} \right.
\end{equation}
where $\AAA^{\Inner}$ is given in \eqref{def:matrixAAA} and
\begin{equation}\label{def:fgInnfgMch}
\begin{aligned}	
f^{\Inn} &= \paren{-\partial_U \KK, 
	i \partial_Y \KK, -i\partial_X \KK}^T,
 & \quad
g^{\inn} &= \partial_{W} \KK, 
\\
f^{\mch} &= \paren{-\partial_U H_1^{\inn}, 
	i \partial_Y H_1^{\inn}, -i\partial_X H_1^{\inn} }^T,
 & \quad
g^{\mch} &= \partial_{W} H_1^{\inn}.
\end{aligned}
\end{equation}	
Notice that, since $(u,\zuOut(u))=\phi_{\Inner}(U,\ZuMchO(U))$ (see \eqref{def:ZdMchO}),
$(U,\ZuMchO(U))$ is an  invariant graph of equation  \eqref{eq:systemEDOsInnerComplete}.
Therefore, $\ZuMchO$ satisfies the invariance equation
\begin{equation*}
\begin{split}	
\partial_U {\ZuMchO} &= 
	\AAA^{\Inner} \ZuMchO 
	+ \RRR^{\Inner}[\ZuMchO]
	+ \RRR^{\mch}[\ZuMchO],
\end{split}
\end{equation*}
with $\RRR^{\Inner}$ as defined in \eqref{def:operatorRRRInner} and
\begin{equation}\label{def:operatorRRRmch}
	\RRR^{\mch}[\phiA] = 
	\frac{\AAA^{\Inner} \phiA
		+ f^{\Inn}(U, \phiA) 
		+ f^{\mch}(U, \phiA)  }
	{1+ g^{\Inn}(U, \phiA) + g^{\mch}(U, \phiA) } 
	- \AAA^{\Inner} \phiA 
	- \RRR^{\Inner}[\phiA].
\end{equation}
Similarly $\ZuInn$ satisfies the invariance equation 
$
\partial_U \ZuInn =
\AAA^{\Inner} \ZuInn 
+ \RRR^{\Inner}[\ZuInn]
$ (see  Theorem \ref{theorem:innerComputations}) and, 
therefore, the difference $\ZuMchU = \ZuMchO - \ZuInn$ must be a solution of
\begin{equation}\label{eq:invariantEquationMchU1}
	\partial_U {\ZuMchU}  =
	\AAA^{\inn} \ZuMchU +
	\BB(U) \ZuMchU + \RRR^{\mch}[\ZuMchO],
\end{equation}
with
\begin{equation}\label{def:operatorBBmch}
	\BB(U) = \int_{0}^{1} D_{Z}\RRR^{\Inn}
	[(1-s)\ZuInn + s \ZuMchO](U) ds.
\end{equation}
The key point is that, since the existence of both $\ZuInn$ and $\ZuMchO$ is already been proven,
we can think of   $\BB(U)$ and
$\RRR^{\mch}[\ZuMchO](U)$ as known functions. Therefore,  equation~\eqref{eq:invariantEquationMchU1} can be understood as  a non homogeneous linear equation with independent term $\RRR^{\mch}[\ZuMchO](U)$.
Moreover, defining the linear operator 
	$\LL^{\inn} \phiA = (\partial_U-\AAA^{\inn})\phiA$,
equation~\eqref{eq:invariantEquationMchU1} is equivalent to
\begin{equation}\label{eq:invariantEquationMchU2}
\LL^{\inn}{\ZuMchU}(U) =
 \BB(U)\ZuMchU(U) 
+\RRR^{\mch}[\ZuMchO](U).
\end{equation}
%
We prove Theorem \ref{theorem:matching} by solving this equation (with suitable initial conditions).
To this end, we define  the  Banach space 
$
\XcalMchTotal = \XcalMch_{\frac{4}{3}} \times \XcalMch_1 \times \XcalMch_1
$
with
\begin{equation*}
	\XcalMch_{\al}= 
	\Bigg\{ \phiA: \DuMchInn \to \complexs  \st  
	\phiA \text{ real-analytic, } 
	\normMch{\phiA}_{\al}=\sup_{U \in \DuMchInn} \vabs{U^{\al}\phiA(U)} < \infty \Bigg\},
\end{equation*}
endowed with the product norm
$
\normMchTotal{\phiA} = \normMch{\phiA_1}_{\frac{4}{3}} + \normMch{\phiA_2}_{1} + \normMch{\phiA_3}_{1}.
$

Next lemma gives some properties of these Banach spaces.
%
%
%
\begin{lemma}\label{lemma:sumNormsMch}
Let $\g \in [\frac{3}{5},1)$ and $\al, \beta \in \reals$. The following statements hold:
\begin{enumerate}
\item If $\varphi \in \XcalMch_{\al}$, then $\varphi \in \XcalMch_{\beta}$ for any $\beta \in \reals$.
Moreover,
\begin{align*}
	\begin{cases}
		\normMch{\varphi}_{\beta} \leq 
		C \kappa^{\beta-\al} 
		\normMch{\varphi}_{\al}, 
		&\quad \text{for } \al  > \beta, \\
		\normMch{\varphi}_{\beta} \leq 
		C \de^{2(\al-\beta)(1-\g)} 
		\normMch{\varphi}_{\al},
		&\quad \text{for } \al  < \beta.
	\end{cases}	
\end{align*}
\item If $\varphi \in \XcalMch_{\al}$ and 
$\zeta \in \XcalMch_{\beta}$, then
$\varphi \zeta \in \XcalMch_{\al+\beta}$ and
	$
	\normMch{\varphi \zeta}_{\al+\beta} \leq \normMch{\varphi}_{\al} \normMch{\zeta}_{\beta}.
	$
\end{enumerate}
\end{lemma}

This lemma is a direct consequence of the fact that, as explained in Section \ref{subsection:matching}, $U$ satisfies
\begin{equation}\label{eq:boundsDomainMchInner}
 	\kappa \cos \betaMchA  \leq 
	\vabs{U} \leq 
	\frac{C}{\de^{2(1-\g)}}.
\end{equation}

Now, we present the main result of this section, which implies Theorem~\ref{theorem:matching}.

\begin{proposition}\label{theorem:matchingProof}
%
There exist $\g^*\in[\frac35,1)$,
$\kappaMch\geq\max\claus{\kappaOuter,\kappaInner}$,
 $\de_0>0$ and $\cttMchA>0$  such that, 
for  $\g \in (\g^*,1)$,
 $\kappa\geq \kappaMch$
and $\de \in (0,\de_0)$, $\ZuMchU$ satisfies
	$\normMchTotal{\ZuMchU} \leq \cttMchA \, \de^{\frac{2}{3}(1-\g)}$.
\end{proposition}


\subsection{An integral equation formulation}
\label{subsection:linearOperatorsMch}

To prove Proposition \ref{theorem:matchingProof}, we first introduce a right-inverse of  $\LL^{\inn}=\partial_U-\AAA^{\inn}$.

\begin{lemma}\label{lemma:operadorEEmch}
The operator $\GG^{\inn}[\phiA]=\paren{\GG^{\inn}_1[\phiA_1],\GG^{\inn}_2[\phiA_2],\GG^{\inn}_3[\phiA_3]}^T$
defined as
\begin{equation}\label{def:operatorEEmch}
	\GG^{\inn}[\phiA](U) = \paren{
		\int_{U_3}^U \phiA_1(S) dS, \,
		\int^U_{U_3} e^{-i(S-U)} \phiA_2(S) dS, \,
		\int^U_{U_2} e^{i(S-U)}\phiA_3(S) dS}^T,
\end{equation}
where $U_2$ and $U_3$ are introduced in~\eqref{def:U12}, is a right inverse of $\LL^{\inn}$. 

Moreover, there exists a constant $C>0$ such that:
\begin{enumerate}
\item Let $\al>1$. If $\phiA \in \XcalMch_{\al}$, then $\GG^{\inn}_1[\phiA] \in \XcalMch_{\al-1}$ and
$
\normMch{\GG^{\inn}_1[\phiA]}_{\al-1} \leq C \normMch{\phiA}_{\al}.
$
\item Let $\al>0$, $j=2,3$. 
If $\phiA \in \XcalMch_{\al}$, 
then $\GG^{\inn}_j[\phiA] \in \XcalMch_{\al}$ and
$
\normMchSmall{\GG^{\inn}_j[\phiA]}_{\al} \leq C  \normMch{\phiA}_{\al}.
$
\end{enumerate}
\end{lemma}

The proof of this lemma  follows the same lines as the proof of Lemma 20 in \cite{BCS13}.
Using the operator $\GG^{\inn}$,  equation~\eqref{eq:invariantEquationMchU2} is equivalent to 
\begin{equation*}
	\ZMchU(U) = C^{\mch} e^{\AAA^{\inn} U} + 
	\GG^{\inn} \left[ \BB \cdot \ZMchU \right] (U) + 
	\paren{\GG^{\inn} \circ \RRR^{\mch}[Z]} (U),
\end{equation*}
where $C^{\mch}=(C^{\mch}_W,C^{\mch}_X, C^{\mch}_Y)^T$ is defined as
\begin{equation*}
	C_W^{\mch} =\WMchU(U_3), \qquad
	C_X^{\mch} =e^{-i U_3}\XMchU(U_3), \qquad
	C_Y^{\mch} =e^{i U_2}\YMchU(U_2). 
\end{equation*}
Then, defining the operator
$
\TTT[\phiA](U) = 
\GG^{\inn} \left[ \BB \cdot \phiA \right](U),
$
this equation  is equivalent to \begin{equation}\label{eq:invariantEquationMchU4}
	(\Id - \TTT)\ZuMchU = C^{\mch} e^{\AAA^{\inn} U} +  
	\paren{\GG^{\inn} \circ \RRR^{\mch}[\ZuMchO]}
\end{equation}
and  therefore, to estimate $\ZuMchU$, we need to prove that $\Id-\TTT$ is  invertible  in $\XcalMchUTotal$.


\begin{lemma}\label{lemma:operatorTTmch}
Let us consider operators $\BB$ and $\GG^{\inn}$ as given in \eqref{def:operatorBBmch} and \eqref{def:operatorEEmch}.
Then, for $\g \in [\frac35,1)$, $\kappa>0$ big enough and $\de>0$ small enough, for $\phiA \in \XcalMchTotal$,
\begin{align*}
\normMchTotal{\TTT[\phiA]} 
= \normMchTotal{\GG^{\inn}[\BB\cdot \phiA]}
\leq \frac{1}{2} \normMchTotal{\phiA}
\end{align*}
and therefore
\begin{align*}
\normMchTotal{(\Id-\TTT)^{-1}[\phiA]} \leq 2 \normMchTotal{\phiA}.
\end{align*}
\end{lemma}

To prove this lemma, we use the following estimates, whose proof is a direct result of Lemma 5.5 in \cite{articleInner}.

\begin{lemma}\label{lemma:technicalMatching}
Fix $\varrho>0$ and take $\kappa>0$ big enough. Then, there exists a constant $C$ (depending on $\varrho$ but independent of $\kappa$) such that, for $\phiA \in \XcalMchTotal$ with $\normMchTotal{\phiA}\leq\varrho$, the functions $g^{\inn}$ and $f^{\inn}$ in \eqref{def:operatorRRRInner} and  the operator  $\RRR^{\inn}$ in \eqref{def:fgInnfgMch}  satisfy
\begin{align*}
\normMch{g^{\inn}(\cdot,\phiA)}_2  \leq C,
\qquad
\normMch{f_1^{\inn}(\cdot,\phiA)}_{\frac{11}3} \leq C,
\qquad
\normMch{f_j^{\inn}(\cdot,\phiA)}_{\frac43}
\leq C,
\quad j=2,3
\end{align*}
and
\begin{align*}
\normMch{\partial_W \RRR^{\inn}_1[\phiA]}_3 &\leq C, &
\normMch{\partial_X \RRR^{\inn}_1[\phiA] }_{\frac73} &\leq C, &
\normMch{\partial_Y \RRR^{\inn}_1[\phiA] }_{\frac73} &\leq C, 
\\
\normMch{\partial_W \RRR^{\inn}_j[\phiA]}_{\frac23} &\leq C, &
\normMch{\partial_X \RRR^{\inn}_j[\phiA] }_{2} &\leq C, &
\normMch{\partial_Y \RRR^{\inn}_j[\phiA] }_{2} &\leq C,
\quad j=2,3.
\end{align*}
\end{lemma}

\begin{proof}[Proof of Lemma \ref{lemma:operatorTTmch}]
%
Let $\ZuMchO$ be as given in \eqref{def:ZdMchO}. Then, by Proposition \ref{proposition:existenceComecocos},
estimates \eqref{eq:boundsDomainMchInner} and taking $\g \in [\frac35,1)$, we have that, for $U \in \DuMchInn$,
\begin{align}\label{eq:boundsZMchO}
	\vabs{\WuMchO(U)} \leq \frac{C}{\vabs{U}^{\frac{8}{3}}}
	+ \frac{C \de^{\frac{4}{3}}}{\vabs{U}}
	\leq
	\frac{C}{\vabs{U}^{\frac{8}{3}}}, \qquad
	\normMch{\XuMchO}_{\frac43} \leq C, \qquad
	\normMch{\YuMchO}_{\frac43} \leq C.
\end{align}		
Then, using also Theorem \ref{theorem:innerComputations}, we obtain that
$(1-s)\ZuInn+s\ZuMchO \in \XcalMchTotal$ for $s \in [0,1]$ and $\g \in [\frac35,1)$ and 
$
\normMchTotal{(1-s)\ZuInn+s\ZuMchO}\leq C.
$
As a result, using the definition of $\BB$ in \eqref{def:operatorBBmch} and Lemma \ref{lemma:technicalMatching},
\begin{equation}\label{eq:boundsOperadorBBmch}
\begin{aligned}
\normMch{\BB_{1,1}}_3 &\leq C,  &
\normMch{\BB_{1,2}}_{\frac73} &\leq C, &
\normMch{\BB_{1,3}}_{\frac73} &\leq C, \\
\normMch{\BB_{j,1}}_{\frac23} &\leq C, &
\normMch{\BB_{j,2}}_2 &\leq C, &
\normMch{\BB_{j,3}}_2 &\leq C, 
\quad \text{for } j=2,3.
\end{aligned}	
\end{equation}
%
Therefore, by Lemmas~\ref{lemma:operadorEEmch} and \ref{lemma:sumNormsMch} and \eqref{eq:boundsOperadorBBmch}, we obtain
\begin{align*}
	\normMch{\TTT_1[\phiA]}_{\frac{4}{3}} &\leq
	C \normMch{\pi_1 \paren{\BB \phiA} }_{\frac{7}{3}}\\
%
	&\leq C \boxClaus{\normMch{\BB_{1,1}}_{1} \normMch{\phiA_1}_{\frac{4}{3}}+
		\normMch{\BB_{1,2}}_{\frac{4}{3}} \normMch{\phiA_2}_{1}+
		\normMch{\BB_{1,3}}_{\frac{4}{3}} \normMch{\phiA_3}_{1}} \\
	&\leq 
	\frac{C}{\kappa^2} \normMch{\phiA_1}_{\frac{4}{3}} +
	\frac{C}{\kappa} \normMch{\phiA_2}_{1} +
	\frac{C}{\kappa} \normMch{\phiA_3}_{1}
	\leq \frac{C}{\kappa} \normMchTotal{\phiA}.
\end{align*}
Proceeding analogously, for $j=2,3$, we have
\begin{align*}
	\normMch{\TTT_j[\phiA]}_{1}
	&\leq C \boxClaus{\normMch{\BB_{j,1}}_{-\frac{1}{3}} 
		\normMch{\phiA_1}_{\frac{4}{3}}+
		\sum_{l=2}^3
		\normMch{\BB_{j,l}}_{0} \normMch{\phiA_l}_{1}}
	%
	\leq \frac{C}{\kappa} \normMchTotal{\phiA}.
\end{align*}
Taking $\kappa>0$ big enough, we obtain the statement of the lemma.
%
\end{proof}

\subsection{End of the proof of Proposition \ref{theorem:matchingProof}}

To complete the proof of Proposition \ref{theorem:matchingProof}, we study the right-hand side of equation \eqref{eq:invariantEquationMchU4}.
%
 
First, we deal with  the term $C^{\mch} e^{\AAA^{\inn} U}$.
Recall that  $U_2$ and $U_3$ in \eqref{def:U12} satisfy
\begin{equation*}
	\frac{C^{-1}}{\de^{2(1-\g)}} \leq
	\vabs{U_j} \leq
	\frac{C}{\de^{2(1-\g)}},
	\hh
	\text{ for } j=2,3.
\end{equation*}
%
%
Then, taking into account that $\WuMchU= \WuMchO-\WuInn$,  \eqref{eq:boundsZMchO} and Theorem \ref{theorem:innerComputations} imply 
\begin{align*}
|C^{\mch}_W| = \vabs{\WuMchU(U_3)} \leq \vabs{\WuMchO(U_3)}+\vabs{\WuInn(U_3)}
\leq \frac{C}{\vabs{U_3}^{\frac{8}{3}}}
\leq 
C \de^{\frac{16}{3}(1-\g)}
\end{align*}
and, as a result, by Lemma \ref{lemma:sumNormsMch},
$	
\normMchSmall{C^{\mch}_W}_{\frac{4}{3}} 
\leq 
C \de^{\frac{8}{3}(1-\g)}.
$
Analogously, for $U \in \DuMchInn$,
\begin{align*}
|C^{\mch}_X e^{iU}| 
&= |e^{i(U-U_3)} \XuMchU(U_3)| 
\leq \frac{C e^{-\Im(U-U_3)}}{\vabs{U_3}^{\frac{4}{3}}} 
\leq C \de^{\frac{8}{3}(1-\g)}
\end{align*}
and then
$
\normMchSmall{C^{\mch}_X e^{iU}}_1
\leq C \de^{\frac{2}{3}(1-\g)}.
$
An analogous result holds for $C^{\mch}_Y e^{-iU}$. Therefore,
\begin{equation}\label{eq:independentTermsMchA}
	\normMchTotalSmall{C^{\mch} e^{\AAA^{\inn} U}}\leq
	C \de^{\frac{2}{3}(1-\g)}.
\end{equation}


Now, we estimate the norm of  $\GG^{\inn} \circ \RRR^{\mch}[\ZuMchO]$. 
The operator $\RRR^{\mch}$ in~\eqref{def:operatorRRRmch} can be rewritten as
\begin{align*}
\RRR^{\mch}[\ZuMchO] = \frac{f^{\mch}(1+g^{\Inn})-
	g^{\mch}(\AAA^{\Inn}\ZuMchO+f^{\Inn})}{(1+g^{\Inn})(1+g^{\Inn}+g^{\mch})}.
\end{align*}
%
Then by \eqref{eq:boundsZMchO}, Lemmas \ref{lemma:sumNormsMch} and \ref{lemma:technicalMatching} and  taking $\kappa$ big enough, we obtain
\begin{equation}\label{proof:boundsfgInn}
\begin{aligned}
	\normMch{g^{\inn}(\cdot,\ZuMchO)}_0 &\leq \frac{C}{\kappa^2} \leq \frac12, &\quad 
	\normMch{i\XuMchO + f_2^{\inn}(\cdot,\ZuMchO)}_0 
	&\leq C, \\
	\normMch{f_1^{\inn}(\cdot,\ZuMchO)}_0 &\leq C, 
	&\quad
	\normMch{-i\YuMchO + f_3^{\inn}(\cdot,\ZuMchO)}_0 &\leq C.
\end{aligned}
\end{equation}
%
%
To analyze 
$f^{\mch}$ and
$g^{\mch}$ (see \eqref{def:fgInnfgMch}) we rely on the estimates for $H_1^{\inn}$  in \eqref{eq:boundsH1Inn} and its derivatives, which can be easily  obtained by Cauchy estimates.
Indeed, they can be applied  since  $U \in \DuMchInn$ and, by \eqref{eq:boundsZMchO}, 
%
\[
\vabs{\WuMchO(U)},\vabs{\XuMchO(U)},\vabs{\YuMchO(U)}\leq C.
\]
%
%
Then, there exists $m>0$ such that
\begin{align}\label{proof:boundsMchA}
|g^{\mch}(U,\ZuMchO)| \leq C 
\de^{\frac43-2 m(1-\g)},
\quad
|f_j^{\mch}(U,\ZuMchO)|\leq C \de^{\frac43-2 m(1-\g)},
\text{ for } j=1,2,3.
\end{align} 
We note that,  for $\g \in (\g^*_0,1)$ with $\g^*_0=\max\{\frac{3}{5},\frac{3m-2}{3m}\}$, we have that ${\frac43}-2m(1-\g)>0$.
Then, for $\g \in (\g^*_0,1)$, $\de$ small enough and $\kappa$ big enough, using \eqref{proof:boundsfgInn} and
\eqref{proof:boundsMchA} we obtain
\begin{align*}
|\RRR^{\mch}_j[\ZuMchO](U)| \leq C \de^{\frac43-2 m(1-\g)}, \qquad
\text{ for }  j=1,2,3. 
\end{align*}
Then, by Lemmas \ref{lemma:sumNormsMch} and \ref{lemma:operadorEEmch}, 
\begin{align*}
\normMchTotalSmall{\GG^{\inn} \circ \RRR^{\mch}[\ZuMchO]} &=
\normMchSmall{\GG^{\inn}_1 \circ \RRR_1^{\mch}[\ZuMchO]}_{\frac{4}{3}} 
+
\textstyle\sum_{j=2}^3
\normMchSmall{\GG^{\inn}_j \circ \RRR_j^{\mch}[\ZuMchO]}_{1} \\
&\leq 
C \normMchSmall{\RRR_1^{\mch}[\ZuMchO]}_{\frac{7}{3}} +
\textstyle
\sum_{j=2}^3
C \normMchSmall{\RRR_j^{\mch}[\ZuMchO]}_{1} 
\leq
C \de^{\frac43-2\paren{m+\frac73}(1-\g)}.
\end{align*}
If we take  $\g^*=\max\claus{\frac35, \g^*_0, \g^*_1}$
with $\g^*_1=\frac{3m+5}{3m+7}$,
and $\g \in (\g^*,1)$, 
\begin{align}\label{eq:independentTermsMchB}
\normMchTotalSmall{\GG^{\inn} \circ \RRR^{\mch}[\ZuMchO]}
\leq
C \de^{\frac23(1-\g)}.
\end{align}
To complete the proof of Proposition~\ref{theorem:matchingProof}, we consider
equation~\eqref{eq:invariantEquationMchU4}.  By Lemma \ref{lemma:operatorTTmch}, $(\Id - \TTT)$ is invertible in $\XcalMchTotal$ and moreover
%
\begin{equation*}
\begin{split}
\normMchTotal{\ZuMchU}& = 
\normMchTotal{ (\Id - \TTT)^{-1}
\paren{C^{\mch} e^{\AAA^{\inn} U} + \GG^{\inn} \circ \RRR^{\mch}[\ZuMchO]}}\\
&\leq 2 \normMchTotal{ C^{\mch} e^{\AAA^{\inn} U} + \GG^{\inn} \circ \RRR^{\mch}[\ZuMchO]}.
\end{split}
\end{equation*}
Then, it is enough to apply   \eqref{eq:independentTermsMchA} and \eqref{eq:independentTermsMchB}.

\qed


\appendix

\section{Estimates for the invariant manifolds}
\label{appendix:proofH-technical}
In this appendix we prove the technical Lemmas \ref{lemma:computationsRRRInf} and
\ref{lemma:computationsRRROut}.
%
All these results involve, in some sense, estimates for the first and second derivatives of the Hamiltonian $H_1^{\out}$ in \eqref{def:hamiltonianOuterSplit}.
However, to obtain estimates for $H_1^{\out}$, we  first obtain some properties of $H_1^{\Poi}$ (see \eqref{def:hamiltonianPoincare01}), which can be written as
\begin{equation}\label{def:hamiltonianPoincare1}
\begin{split}
H_1^{\Poi} = 
\frac1{\mu} \PP[0] - \frac{1-\mu}{\mu}{\PP[\mu]}
	- {\PP[\mu-1]},
\end{split}
\end{equation}
where
\begin{equation}\label{def:functionD}
\PP[\zeta](\la,L,\eta,\xi) = 
\paren{\norm{q-(\zeta,0)}^{-1}}
\circ \phi_{\Poi}.
\end{equation}
%
In \cite{articleInner}, we computed the series expansion of $\PP[\zeta]$ in powers of $(\eta,\xi)$. In particular, $\PP[\zeta]$ can be written as 
\begin{equation}\label{def:functionD2}
 \PP[\zeta](\la,L,\eta,\xi) =\frac{1}{\sqrt{A[\zeta](\la)+	B[\zeta](\la,L,\eta,\xi)}}
\end{equation}
where $A$ and $B$ are of the form 
\begin{align}
A[\zeta](\la) =& \, 
1- 2\zeta \cos \la  +\zeta^2, \label{def:DA} \\
\begin{split}\label{def:DB}
	B[\zeta](\la,L,\eta,\xi) =& \,
	4 (L-1)(1-\zeta\cos\la) +
	\frac{\eta}{\sqrt2}
	\paren{3 \zeta - 2e^{-i\la}  - \zeta e^{-2i\la}} \\
	&+ \frac{\xi}{\sqrt2}
	\paren{3 \zeta - 2 e^{i\la} - \zeta e^{2 i\la}}
	+ R[\zeta](\la,L,\eta,\xi),
\end{split} 
\end{align}
and, for fixed $\varrho>0$,  $R$ is  analytic and satisfies that
\begin{equation}\label{eq:boundD2mes}
	\vabs{R[\zeta](\la,L,\eta,\xi)} \leq 
	K(\varrho) \vabs{(L-1,\eta,\xi)}^2,
\end{equation}
for $\vabs{\Im \la}\leq \varrho$, $\vabs{(L-1,\eta, \xi)} \ll 1$ and $\zeta\in [-1,1]$.

Then, wherever $|A[\zeta](\la)|>|B[\zeta](\la,L,\eta,\xi)|$, $\PP[\zeta](\la,L,\eta,\xi)$ can be written as 
%
%
%
\begin{equation}\label{def:SerieP}
	\PP[\zeta](\la,L,\eta,\xi) =  
	\frac1{\sqrt{A[\zeta]}} +
	\sum_{n=1}^{+\infty}
	\binom{-\frac12}{n}
	\frac{\paren{B[\zeta]}^n }{ \paren{A[\zeta]}^{n+\frac12}}.
\end{equation}

\begin{remark}\label{remark:collisions}
The Hamiltonian $H^{\Poi} = H_0^{\Poi}	+ \mu H_1^{\Poi}$
 (see \eqref{def:hamiltonianPoincare} and \eqref{def:hamiltonianPoincare1}) is analytic away from the collisions with the primaries, that is zeroes of the denominators of  $\PP[\mu]$ and $\PP[\mu-1]$.
%
%
For $0<\mu \ll 1$,  one has
\begin{align*}
A[\mu] = 1 + \OO(\mu), \qquad 
A[\mu-1] = 2 +2 \cos\la + \OO(\mu).
\end{align*}
Therefore, in the regime that we consider, collisions with the primary $S$ are not possible but collisions with $P$ may take place at $\la \sim \pi$. 
%
\end{remark}

We now obtain  estimates for $H_1^{\Poi}$ in domains ``far'' from $\la=\pi$. 

\begin{lemma}\label{corollary:seriesH1Poi}
Fix $\la_0 \in (0,\pi)$ and $\mu_0\in (0,\frac12)$ and consider the Hamiltonian $H_1^{\Poi}$ and the potential $V$ introduced in \eqref{def:hamiltonianPoincare1} and 
\eqref{def:potentialV}, respectively.	
Then, for 
for $\vabs{\la}<\la_0$, $\vabs{(L-1,\eta,\xi)}\ll 1$ and $\mu\in(0,\mu_0)$,
the Hamiltonian $H_1^{\Poi}$ can be written as
\begin{align*}
	H_1^{\Poi}(\la,L,\eta,\xi;\mu)-V(\la) 
	=&
	D_0(\mu,\la) +
	D_{1}(\mu,\la) \big((L-1),\eta,\xi \big) +
	D_2(\la,L,\eta,\xi;\mu),
\end{align*}
such that, for $j=1,2,3$,
\begin{align*}
	&\vabs{D_{0}(\mu,\la)} \leq K \mu,
	\qquad
	\vabs{(D_{1}(\mu,\la))_j} \leq K, 
	\qquad
	\vabs{D_2(\la,L,\eta,\xi;\mu)} \leq 
	K \vabs{(L-1,\eta,\xi)}^2,
\end{align*}
with $K$ a positive constant independent of $\la$ and $\mu$.
\end{lemma}

\subsection{Estimates in the infinity domain}
\label{subsubsection:ProofComputationsRRRInf}

To prove Lemma~\ref{lemma:computationsRRRInf}, we need to obtain estimates for $\RRR^{\out}$ and its derivatives.
Let us recall that, by its definition in~\eqref{def:operatorRRROuter}, for $z=(w,x,y)$ we have
\begin{align}\label{eq:expressionRRRout}
	\RRR^{\out}[z] = \paren{
		\frac{f_1^{\out}(\cdot,z)}{1+g^{\out}(\cdot,z)},
		\frac{{f_2}^{\out}(\cdot,z) - \frac{i x}{\de^2} \, g^{\out}(\cdot,z)}{1+g^{\out}(\cdot,z)},
		\frac{{f_3}^{\out}(\cdot,z) + \frac{i y}{\de^2} \, g^{\out}(\cdot,z)}{1+g^{\out}(\cdot,z)}
	}^T,
\end{align} 
where $g^{\out} = \partial_w H_1^{\out}$ and $f^{\out}=\paren{-\partial_u H_1^{\out}, i\partial_y H_1^{\out}, -i\partial_x H_1^{\out} }^T$.

Therefore, we need to obtain first estimates for the first and second derivatives of $H_1^{\out}$, introduced in~\eqref{def:hamiltonianOuterSplit}, that is
\begin{equation}\label{proof:H1outForM1}
H_1^{\out} = H \circ (\phi_{\equi} \circ \phi_{\out}) - \paren{w + \frac{xy}{\de^2}},
\end{equation}
where $H=H_0+H_1$ with $H_0= H_{\pend}+ H_{\osc}$ (see \eqref{def:hamiltonianScaling},  \eqref{def:hamiltonianScalingH0}).
%

Since $(\la_h,\La_h)$  is a solution of the Hamiltonian  $H_{\pend}$ and belongs to the energy level $H_{\pend}=-\frac12$,
\begin{align*}
H_0 \circ \phi_{\out}
=
H_{\pend}\paren{\la_h(u),\La_h(u)-\frac{w}{3 \La_h(u)}}
+
H_{\osc}(x,y;\de)
=
-\frac12 
+ w
- \frac{w^2}{6\La_h^2(u)}
+ \frac{x y}{\de^2}.
\end{align*}
Therefore, by \eqref{proof:H1outForM1}, the Hamiltonian $H_1^{\out}$ can be expressed (up to a constant) as
\begin{align}\label{eq:expressionHamiltonianInfty}
H_1^{\out} = 
M \circ \phi_{\out}
- \frac{w^2}{6\La_h^2(u)} ,
\end{align}
where 
\begin{align*}
M(\la,\La,x,y;\de) = 
(H \circ \phi_{\equi})(\la,\La,x,y;\de) - 
H_0(\la,\La,x,y).
\end{align*}
In the following lemma we give properties of $M$.

\begin{lemma}\label{lemma:expressionHamiltonianInfty}
Fix constants $\varrho>0$ and $\la_0 \in (0,\pi)$.
%
Then, there exists $\de_0>0$ such that, 
for $\de \in (0,\de_0)$,
$\vabs{\la}<\la_0$, 
$\vabs{\La}<\varrho$
and 
$\vabs{(x,y)}<\varrho\de$
,
the function $M$ satisfies
\begin{align*}
\vabs{ \partial_{\la} M} &\leq 
C\de^2 \vabs{(\la,\La)} +	C\de \vabs{(x,y)}, &
\vabs{ \partial_{x} M} &\leq 
C\de \vabs{(\la,\La,x,y)}, \\
\vabs{ \partial_{\La} M} &\leq
C\de^2 \vabs{(\la,\La)} +	C\de \vabs{(x,y)}, &
\vabs{ \partial_{y} M} &\leq 
C\de \vabs{(\la,\La,x,y)},
\end{align*}
and
\begin{align*}
\vabs{ \partial^2_{\la} M},
\vabs{ \partial_{\la\La} M},
\vabs{ \partial^2_{\La} M}
&\leq C \de^2, 
\qquad
\vabs{\partial_{i j} M} \leq C \de, 
\quad \text{for } i,j \in \claus{\la,\La,x,y}.
\end{align*}
\end{lemma}

\begin{proof}
%
Applying  $\phi_{\equi}$ (see \eqref{def:changeEqui}) to the Hamiltonian $H=H_0 + H_1$,  we have that
\begin{equation}\label{proof:M1expressionLtres}
\begin{split}	
M &= 
\paren{H_0 \circ \phi_{\equi} - H_0}
+ H_1 \circ \phi_{\equi} 
\\
&= 	 \de( x \Ltresy + y\Ltresx) 
+ 3\de^2 \La \LtresLa  
+ \de^4 \paren{-\frac{3}{2}\LtresLa^2 +\Ltresx \Ltresy}
+ H_1 \circ \phi_{\equi}.
\end{split}
\end{equation}
Then,  
\begin{align}\label{proof:estimatesDerivativesM1B}
\vabs{\partial_{i j} M} \leq 
\vabs{\partial_{i j} H_1(\la,\La + \de^2\LtresLa, x+ \de \Ltresx, y + \de \Ltresy;\de)}
, 
\quad \text{for } i,j \in \claus{\la,\La,x,y}.
\end{align}
Since $\vabs{\La}<\varrho$
and $\vabs{(x,y)}<\varrho\de$,
then
$\vabs{\La+\de^2\LtresLa}<2\varrho$
and ${\vabs{(x+\de^3 \Ltresx,y+\de^3\Ltresy)}<2\varrho\de}$,
for $\de$ small.
%
By the definition of $H_1$
in \eqref{def:hamiltonianScalingH1} we have that,
\begin{align*}	
H_1(\la,\La,x,y;\de) &=
H_1^{\Poi}
\paren{\la,1+\de^2\La,\de x, \de y;\de^4} 
- V\paren{\la} 
+ \frac{1}{\de^4} F_{\pend}(\de^2\La),
\end{align*}
where $H_1^{\Poi}$ is given in \eqref{def:hamiltonianPoincare01} (see also \eqref{def:changeScaling}),
$V$ is given \eqref{def:potentialV}
and $F_{\pend}$ is given \eqref{def:Fpend} and satisfies $F_{\pend}(s)=\OO(s^3)$. 
Since $\vabs{(\de^2\La,\de x, \de y)} < 2\varrho \de^2 \ll 1$, we  apply  Lemma \ref{corollary:seriesH1Poi} (recall that $\de=\mu^{\frac14}$) and Cauchy estimates to 
obtain 
\begin{align}\label{proof:estimatesDerivativesM1C}
	\vabs{ \partial^2_{\la} H_1}
	\vabs{ \partial_{\la\La} H_1},
	\vabs{ \partial^2_{\La} H_1}
	&\leq C \de^2, 
	\qquad
	\vabs{\partial_{i j} H_1} \leq C \de, 
	\quad \text{for } i,j \in \claus{\la,\La,x,y}.
\end{align}
%
Then,   \eqref{proof:estimatesDerivativesM1B} and
\eqref{proof:estimatesDerivativesM1C} give the estimates for the second derivatives of $M$.

For the first derivatives of $M$, let us take into account that, by Theorem~\ref{theorem:singularities},  $0$ is a critical point of both Hamiltonians $(H \circ \phi_{\equi})$ and $H_0$ and,
therefore, also of $M= (H \circ \phi_{\equi})- H_0$.
This fact and the estimates of the second derivatives, together with the mean value theorem, gives the estimates for the first derivatives of $M$.
%
\end{proof}

\begin{proof}[End of the proof of Lemma~\ref{lemma:computationsRRRInf}]

Let us consider $\varphi=(\varphi_w,\varphi_x,\varphi_y)^T \in \XcalInftyTotal$ such that
$\normInftyTotal{\phiA} \leq \varrho \de^3$.
We estimate the first and second derivatives of $H_1^{\out}$ evaluated at $(u,\phiA(u))$ (recall \eqref{eq:expressionRRRout}), given by
\begin{align}\label{proof:expressionH1sepInfty}
H_1^{\out}(u,\phiA(u);\de) = M\paren{\la_h(u),\La_h(u)-\frac{\phiA_w(u)}{3\La_h(u)},
\phiA_x(u),\phiA_y(u);\de}
- \frac{\phiA^2_w(u)}{6\La_h^2(u)}. 
\end{align}	
First, let us define
\begin{align*}
{\varphi}_{\la}(u) = \la_h(u), \qquad
{\varphi}_{\La}(u) = \La_h(u)-\frac{\varphi_w(u)}{3\La_h(u)}
\quad \text{and} \quad
\Phi = (\phiA_{\la},\phiA_{\La},\phiA_x,\phiA_y).
\end{align*}
Since $\normInftyTotal{\phiA} \leq \varrho \de^3$ and $\la_h, \La_h \in \XcalInfty_{\vap}$ (see \eqref{eq:separatrixBanachSpace}), 
\begin{equation}\label{proof:estimatesCoordinatesEqui}
\normInfty{\varphi_{w}}_{2\vap} \leq C \de^2, \qquad
\normInfty{\varphi_{x}}_{\vap},
\normInfty{\varphi_{y}}_{\vap} \leq C \de^3, \qquad
\normInfty{\varphi_{\la}}_{\vap},
\normInfty{\varphi_{\La}}_{\vap} \leq C.
\end{equation}
Moreover since, by Theorem \ref{theorem:singularities}, $\la_h(u)\neq\pi$ for  $u \in \DuInfty$, we have that
\[
\vabs{\phiA_{\la}(u)} = \vabs{\la_h(u)} < \pi,
\quad
\vabs{\phiA_{\La}(u)}
\leq C e^{-\vap \rhoInfty}
\leq C,
\quad
\vabs{(\phiA_{x}(u), \phiA_y(u))}
\leq C\de^3 e^{-\vap \rhoInfty}
\leq C\de^3
\]
and, therefore, we can apply  Lemma~\ref{lemma:expressionHamiltonianInfty} to \eqref{proof:expressionH1sepInfty}.
In the following computations, we use generously Lemma \ref{lemma:sumNormsOutInf} without mentioning it.

\begin{enumerate}
\item First, we consider $g^{\out}=\partial_w H_1^{\out}$.
By \eqref{proof:expressionH1sepInfty}, we have that
\begin{align*}
g^{\out}(u,\varphi(u)) 
&=-\frac{\partial_{\La} M \circ \Phi(u)}{3\La_h(u)} 
-\frac{\varphi_{w}(u)}{3\La_h^2(u)}.
\end{align*}
Notice that, by Theorem \ref{theorem:singularities}, $\vabs{\La_h(u)}\geq C$ for $u \in \DuInfty$. Then, $\normInftySmall{\La_h^{-1}}_{-\vap} \leq C$.
Therefore, by Lemma \ref{lemma:expressionHamiltonianInfty} and estimates~\eqref{proof:estimatesCoordinatesEqui}, we have that
\begin{equation}\label{proof:boundsInftyg}
\begin{split}	
\normInfty{g^{\out}(\cdot,\varphi)}_0 &\leq
C \de \boxClaus{\de\normInfty{\varphi_{\la}}_{\vap} +
\de\normInfty{\varphi_{\La}}_{\vap}  +
\normInfty{\varphi_{x}}_{\vap} +
\normInfty{\varphi_{y}}_{\vap}} 
+ C \normInfty{\varphi_{w}}_{2\vap} \\
&\leq C \de^2.
\end{split}
\end{equation}
To compute its derivative with respect to $w$, by \eqref{proof:expressionH1sepInfty}, we have that
\begin{align*}
\partial_w g^{\out} (u,\varphi(u)) 
&=
\frac{\partial^2_{\La} M \circ \Phi(u)}{9\La^2_h(u)} 
-\frac{1}{3\La^2_h(u)},
\end{align*}
and, by Lemma \ref{lemma:expressionHamiltonianInfty} and estimates~\eqref{proof:estimatesCoordinatesEqui}, 
%
$\normInfty{\partial_w g^{\out}(\cdot,\varphi)}_{-2\vap} \leq C$.
Following a similar procedure, we obtain
$\normInfty{\partial_x g^{\out}(\cdot,\varphi)}_{-\vap} \leq C \de$
and $\normInfty{\partial_y g^{\out}(\cdot,\varphi)}_{-\vap} \leq C \de$.

\item Now, we obtain estimates for $f_1^{\out}=-\partial_u H_1^{\out}$.
By \eqref{proof:expressionH1sepInfty}, we have that
\begin{align*}
f_1^{\out}(u,\varphi(u)) =&
- \lad_h(u)  \partial_{\la}{M} \circ \Phi(u)
-\frac{\Lad_h(u)}{3\La^3_h(u)} \varphi^2_{w}(u) \\
&- \paren{\Lad_h(u) + \frac{\Lad_h(u)}{3 \La_h^2(u)} \varphi_{w}(u)} 
\partial_{\La}{M} \circ \Phi(u).
\end{align*}
Then, since $\lad_h, \Lad_h \in \XcalInfty_{\vap}$,  by Lemma \ref{lemma:expressionHamiltonianInfty} and estimates~\eqref{proof:estimatesCoordinatesEqui}, we have that
$\normInfty{f_1^{\out}(\cdot,\varphi)}_{2\vap} \leq C\de^2$.
To compute its derivative with respect to  $x$, by \eqref{proof:expressionH1sepInfty}, 
\begin{align*}
\partial_x f_1^{\out}(u,\varphi(u)) =&
- \lad_h(u) \partial_{x \la} M \circ \Phi(u) 
- \paren{\Lad_h(u) + \frac{\Lad_h(u)}{3 \La_h^2(u)} \varphi_{w}(u)} 
\partial_{x \La} M \circ \Phi(u)
\end{align*}
and, therefore,
%
$\normInfty{\partial_x f_1^{\out}(\cdot,\varphi)}_{\vap} \leq C\de$.
Similarly one can obtain
$\normInfty{\partial_w f_1^{\out}(\cdot,\varphi)}_{0} \leq C\de^2$ and
$\normInfty{\partial_y f_1^{\out}(\cdot,\varphi)}_{\vap} \leq C\de$.
\item Analogously to the previous estimates, 
we can obtain bounds for $f^{\out}_2= i \partial_y H_1^{\out}$ and $f^{\out}_3=-i \partial_x H_1^{\out}$.
Then, for $j=2,3$, it can be seen that
$\normInftySmall{f_j^{\out}(\cdot,\varphi)}_{\vap} \leq C\de$, 
and differentiating we obtain
$\normInftySmall{\partial_w f_j^{\out}(\cdot,\varphi)}_{-\vap} \leq C\de$,
$\normInftySmall{\partial_x f_j^{\out}(\cdot,\varphi)}_{0} \leq C\de$ and
$\normInftySmall{\partial_y f_j^{\out}(\cdot,\varphi)}_{0} \leq C\de$.
\end{enumerate}
Then, by the definition of $\RRR^{\out}$ in~\eqref{eq:expressionRRRout} and the just obtained  estimates, we complete the proof of the lemma.
\end{proof}

\subsection{Estimates in the outer domain}
\label{subsubsection:proofComputationsRRROut}

To obtain estimates of $\RRR^{\out}$, we write
$H_1^{\out}$ in~\eqref{def:hamiltonianOuterSplit} (up to a constant) as
\begin{equation*}
H_1^{\out} = 
H_1 \circ \phi_{\equi} \circ \phi_{\out}
- \frac{w^2}{6\La_h^2(u)} 
+
\de( x \Ltresy + y\Ltresx) 
+ 3\de^2 \LtresLa  \paren{\La_h(u)-\frac{w}{3\La_h(u)}},
\end{equation*} 
(see \eqref{eq:expressionHamiltonianInfty} and \eqref{proof:M1expressionLtres}).
Then, by the definition of $H_1$ in \eqref{def:hamiltonianScalingH1}, we obtain
\begin{align*}
H_1^{\out} =& \,
(H^{\Poi}_1-V) 
\circ \phi_{\sca} \circ \phi_{\equi} \circ \phi_{\out}
+ \frac{1}{\de^4} 
F_{\pend}\paren{\de^2\La_h(u) - 
	\frac{\de^2  w}{3 \La_h(u)} + 
	\de^4\LtresLa} \\
&- \frac{w^2}{6\La_h^2(u)} 
+
\de( x \Ltresy + y\Ltresx) 
+ 3\de^2 \LtresLa  \paren{\La_h(u)-\frac{w}{3\La_h(u)}},
\end{align*}
where $H_1^{\Poi}$ is given in \eqref{def:hamiltonianPoincare1}, 
the potential $V$ in \eqref{def:potentialV} and $F_{\pend}$ in \eqref{def:Fpend}.
The changes of coordinates $\phi_{\sca}$, $\phi_{\equi}$ and $\phi_{\out}$ are given in \eqref{def:changeScaling}, \eqref{def:changeEqui} and \eqref{def:changeOuter}, respectively.

Considering $z=(w,x,y)$, we denote the composition of  change of coordinates as
\begin{equation}\label{def:changeTotalOuter}
	(\la, L, \eta, \xi) =
	\Theta(u,z) = (\phi_{\sca} \circ \phi_{\equi} \circ \phi_{\out})(u,z).
\end{equation}
%
%
Then, since $\mu=\de^4$, the Hamiltonian $H_1^{\out}$ can be split (up to a constant) as
\begin{equation}\label{proof:H1outExpressionOuter}
	H_1^{\out} = M_P + M_{S} + M_R,
\end{equation}
where
\begin{align}
M_P(u,z;\de) =& 
-\paren{\PP[\de^4-1]
	- \frac{1}{\sqrt{2+2\cos \la}}
}\circ \Theta(u,z), 
\label{def:expressionMJOuter}
\\
%
%
M_{S}(u,z;\de) =& \,
\paren{ 
	\frac{1}{\de^4} \PP[0] -
	\frac{1-\de^4}{\de^4} \PP[\de^4] 
	- 1 + \cos \la} \circ\Theta(u,z), 
\label{def:expressionMSOuter} \\
\begin{split}\label{def:expressionMROuter}
	M_R(u,z;\de) =& 
	- \frac{w^2}{6\La_h^2(u)} 
	+ \de^2 \LtresLa \paren{3 \La_h(u)
		- \frac{w}{\La_h(u)} } 
	+ \de (x \Ltresy + y \Ltresx ), \\
	&+ \frac{1}{\de^4} 
	F_{\pend}\paren{\de^2\La_h(u) - 
		\frac{\de^2  w}{3 \La_h(u)} + 
		\de^4\LtresLa},
\end{split}	
\end{align}
and  $\PP$ is the function given in~\eqref{def:functionD}.
%
%

To obtain estimates for the  derivatives of $M_P$, $M_S$ and $M_R$,  we first analyze  the change of coordinates $\Tht$ in~\eqref{def:changeTotalOuter}. It can be expressed as
\begin{equation}\label{proof:ThtExpression}
\Tht(u,z)=
\Big(\pi + \Tht_{\la}(u), 1 + \Tht_{L}(u,w),
\Tht_{\eta}(x),\Tht_{\xi}(y) \Big),
\end{equation}
where
\begin{equation*}
\begin{aligned}	
\Tht_{\la}(u) &= \la_h(u)-\pi,  &
\Tht_{\eta}(x) &= \de x + \de^4 \Ltresx(\de), \\
\Tht_L(u,w) &= \de^2 \La_h(u) 
- \frac{\de^2 w}{3 \La_h(u)} + \de^4 \LtresLa(\de), \quad&
\Tht_{\xi}(x) &= \de y + \de^4 \Ltresy(\de).
\end{aligned}
\end{equation*}
Next lemma, which is a direct consequence of Theorem \ref{theorem:singularities}, gives estimates for this change of coordinates.

\begin{lemma}\label{lemma:changeTotalOuter}
Fix  $\varrho>0$ and $\de>0$ small enough. Then, for
$\phiA \in \ballOuter \subset \XcalOutTotal$,
\begin{align*}
\normOut{\Tht_{\la}}_{0,-\frac{2}{3}}&\leq C, &
\normOut{\Tht_{L}(\cdot,\phiA)}_{0,\frac{1}{3}} 
&\leq C \de^2, &
\normOut{\Tht_{\eta}(\cdot,\phiA)}_{0,\frac{4}{3}} 
&\leq C \de^4, \\
\normOut{\Tht_{\la}^{-1}}_{0,\frac{2}{3}}&\leq C, &
\normOut{1+\Tht_{L}(\cdot,\phiA)}_{0,0} 
&\leq C, &
\normOut{\Tht_{\xi}(\cdot,\phiA)}_{0,\frac{4}{3}} 
&\leq C \de^4.
\end{align*}
Moreover, its derivatives satisfy
\begin{align*}
\normOut{\partial_u \Tht_{\la}}_{0,\frac{1}{3}} &\leq C, 
& 
\normOut{\partial_u \Tht_{L}(\cdot,\phiA)}_{0,\frac{4}{3}} &\leq C \de^2, 
&
\normOut{\partial_w \Tht_{L}(\cdot,\phiA)}_{0,-\frac{1}{3}} &\leq C \de^2,
\\
\normOut{\partial_{uw} \Tht_{L}(\cdot,\phiA)}_{0,\frac{2}{3}} &\leq C\de^2,
& 
\partial_x \Tht_{\eta},
\partial_y \Tht_{\xi} &\equiv \de, 
& 
\partial^2_{w} \Tht_L, 
\partial^2_{x} \Tht_{\eta}, 
\partial^2_y \Tht_{\xi} &\equiv 0. 
\end{align*}
\end{lemma}

In the next lemma we obtain estimates for the  derivatives of $M_P$.

\begin{lemma}\label{lemma:boundsMPout}
Fix $\varrho>0$,  $\de>0$ small enough and
$\kappa>0$ big enough. Then, for 
$\phiA \in \ballOuter$ and $*=x,y$, 
\begin{equation*}
\begin{aligned}	
\normOut{\partial_u M_P(\cdot,\phiA)
}_{1,1}
&\leq C \de^2, & 
\normOut{\partial_w M_P(\cdot,\phiA)
}_{1,-\frac{2}{3}}
&\leq C \de^{2},  & 
\normOut{\partial_* M_P(\cdot,\phiA)
}_{0,\frac{4}{3}}
&\leq C \de, \\
\normOut{\partial_{u w} M_P(\cdot,\phiA)
}_{1,\frac{1}{3}}
&\leq C \de^{2},  &
\normOut{\partial_{u *} M_P(\cdot,\phiA)
}_{0,\frac{7}{3}}
&\leq C \de,  &
\normOut{\partial^2_{w} M_P(\cdot,\phiA)
}_{0,\frac{4}{3}}
&\leq C \de^4, \\
\normOut{\partial_{w *} M_P(\cdot,\phiA)
}_{0,\frac{5}{3}}
&\leq C \de^3, &
\normOut{\partial^2_{*} M_P(\cdot,\phiA)
}_{0,2}
&\leq C \de^2,  &
\normOut{\partial_{xy} M_P(\cdot,\phiA)
}_{0,2}
&\leq C \de^2.
\end{aligned}
\end{equation*}
\end{lemma}

\begin{proof}

We consider $\phiA \in \ballOuter \subset \XcalOutTotal$
and 
we estimate the derivatives of $\PP[\de^4-1] \circ \Tht(u,\phiA(u))$. We
first we obtain bounds for $A[\de^4-1]$ and $B[\de^4-1]$ (see \eqref{def:DA} and \eqref{def:DB}).
To simplify the notation, we define
\begin{align}\label{proof:defABP}
\wt{A}(u) = A[\de^4-1](\pi +\Tht_{\la}(u)), \qquad
\wt{B}(u,z) = B[\de^4-1] \circ \Tht(u,z).
\end{align}

%
%
%


In the following computations we use extensively the results in Lemma~\ref{lemma:sumNormsOuter} without mentioning them.

\begin{enumerate}
\item Estimates of $\wtA(u)$:
Defining $\lah=\la-\pi$, by \eqref{def:DA},
\begin{align*}
A[\de^4-1](\lah+\pi) = 2(1-\cos\lah) -2\de^4(1-\cos\lah) + \de^8.
\end{align*}
Then, applying Lemma~\ref{lemma:changeTotalOuter}, 
\begin{equation*}\label{proof:estimateTrigoOuter}
\begin{aligned}
\normOut{\sin \Tht_{\la}}_{0,-\frac{2}{3}} 
\leq 
C \normOut{\Tht_{\la}}_{0,-\frac{2}{3}}
\leq C, 	\qquad
\normOut{(1-\cos \Tht_{\la})^{-1}
}_{0,\frac{4}{3}} 
\leq
C \normOut{\Tht_{\la}^{-2}
}_{0,\frac{4}{3}} 
\leq C
\end{aligned}	
\end{equation*}
and, as a result,
\begin{equation}\label{proof:DA}
\begin{split}	
\normOutSmall{\wtA^{-1}}_{0,\frac{4}{3}} 
&\leq 
C \normOut{(1-\cos \Tht_{\la})^{-1}
}_{0,\frac{4}{3}} 
\leq C,
\\
\normOutSmall{\partial_u \wtA}_{0,-\frac{1}{3}}
&\leq C
\normOut{\sin \Tht_{\la}}_{0,-\frac23}
\normOut{\partial_u \Tht_{\la}}_{0,\frac13}
\leq C.
\end{split}
\end{equation}

\item Estimates of $\wtB(u,\phiA(u))$:
Considering the auxiliary variables
$(\lah,\Lh)=(\la-\pi,L-1)$,
we have that
\begin{equation}\label{proof:D0expansion}
\begin{split}	
B[\de^4-1](\pi+\lah,1+\Lh,\eta,\xi)
= \,& 
4 \Lh (1-\cos \lah + \de^4 \cos \lah) \\ 
&+
\frac{\eta}{\sqrt{2}} 
(-3 +2 e^{-i\lah}+e^{-2i\lah} +\de^4(3+ e^{-2i\lah})) \\
&+
\frac{\xi}{\sqrt{2}} 
(-3 +2 e^{i\lah}+e^{2i\lah} +\de^4(3+ e^{2i\lah})) \\
&+ R[\de^4-1](\pi+\lah,1+\Lh,\eta,\xi).
\end{split}
\end{equation}
Then, by the estimates in \eqref{eq:boundD2mes} and Lemma~\ref{lemma:changeTotalOuter},
\begin{equation}\label{proof:DB}
\begin{aligned}
\normOutSmall{\wtB(\cdot,\phiA)}_{1,-2} 
\leq \,&
C \normOut{\Tht_{L}(\cdot,\phiA)
\Tht_{\la}^2}_{0,-1}
+ \frac{C}{\de^2}
\normOut{\Tht_{\eta}(\cdot,\phiA) \Tht_{\la}}_{0,\frac23} \\
&+ \frac{C}{\de^2}
\normOut{\Tht_{\xi}(\cdot,\phiA) \Tht_{\la}}_{0,\frac23}
+ 
\frac{C}{\de^2}
\normOut{(\Tht_{L},\Tht_{\eta},\Tht_{\xi})^2}_{0,\frac23} \leq C \de^2.
\end{aligned}
\end{equation}

Now, we look for estimates of the first derivatives of $\wtB(u,\phiA(u))$.
By its definition in \eqref{proof:defABP} and the expression of $\Tht$ in \eqref{proof:ThtExpression}, we have that
\begin{equation}\label{proof:boundsBtildePoincare}
\begin{aligned}	
\partial_u \wtB =& \,
\boxClaus{ \partial_{\la} B[\de^4-1] \circ \Tht}
\partial_u  \Tht_{\la}
+
\boxClaus{\partial_{L} B[\de^4-1] \circ \Tht}
\partial_u  \Tht_{L}, \\
\partial_w \wtB =& \,
\boxClaus{
\partial_{L} B[\de^4-1] \circ \Tht
}
\partial_w  \Tht_{L}, \\
\partial_x \wtB =& \,
\boxClaus{
	\partial_{\eta} B[\de^4-1] \circ \Tht
}
\partial_{x}  \Tht_{\eta}, \qquad
\partial_y \wtB = \,
\boxClaus{
	\partial_{\xi} B[\de^4-1] \circ \Tht
}
\partial_y \Tht_{\xi}.
\end{aligned}
\end{equation}
Differentiating \eqref{proof:D0expansion} and applying Lemma~\ref{lemma:changeTotalOuter}, 
\begin{align*}
\normOut{\partial_{\la} B[\de^4-1] \circ \Tht(\cdot,\phiA)}_{1,-\frac43}
\leq& \, 
C \normOut{\Tht_{L}(\cdot,\phiA) \Tht_{\la}}_{-\frac13} +
\frac{C}{\de^2}
\normOut{\Tht_{\eta}(\cdot,\phiA)}_{0,\frac43}
\\
&+ \frac{C}{\de^2}
\normOut{\Tht_{\xi}(\cdot,\phiA)}_{0,\frac43}
+ C\de^2
\leq C \de^2, 
\\
\normOut{\partial_{L} B[\de^4-1] \circ \Tht(\cdot,\phiA)}_{1,-\frac73}
\leq&
C \normOut{\Tht_{\la}^2}_{0,-\frac43} +
\frac{C}{\de^2} \normOut{\Tht_{L}(\cdot,\phiA)}_{0,\frac13}
+ \frac{C}{\kappa} 
\leq C,
\\
\normOut{\partial_{*} B[\de^4-1] \circ \Tht(\cdot,\phiA)}_{0,-\frac23}
\leq&
C \normOut{\Tht_{\la}}_{0,-\frac23}
+
\frac{C}{\kappa}
\leq C, \quad \text{for } *=\eta,\xi.
\end{align*}
Then, using also \eqref{proof:boundsBtildePoincare} and taking $*=x,y$,
\begin{align}\label{proof:DBfirst}
\normOutSmall{\partial_u \wtB(\cdot,\phiA)
}_{1,-1}
&\leq C \de^2, & 
\normOutSmall{\partial_w \wtB(\cdot,\phiA)
}_{1,-\frac83}
&\leq C \de^2, & 
\normOutSmall{\partial_* \wtB(\cdot,\phiA)
}_{0,-\frac23}
&\leq C \de.	
\end{align}
Analogously, for the second derivatives, one can obtain the estimates
\begin{equation}\label{proof:DBsecond}
\begin{aligned}	
\normOutSmall{\partial_{u w} \wtB(\cdot,\phiA) }_{1,-\frac53}
&\leq C \de^2, & \hspace{-1mm}
\normOutSmall{\partial^2_{w} \wtB(\cdot,\phiA)
}_{0,\frac23}
&\leq C \de^4, & \hspace{-1mm}
\normOutSmall{\partial_{u *} \wtB(\cdot,\phiA)
}_{0,\frac13}
&\leq C \de, \\
\normOutSmall{\partial_{w *} \wtB(\cdot,\phiA)
}_{0,-\frac13}
&\leq C \de^3,& \hspace{-1mm}
\normOutSmall{\partial^2_{*} \wtB(\cdot,\phiA)
}_{0,0}
&\leq C \de^2, & \hspace{-1mm}
\normOutSmall{\partial_{xy} \wtB(\cdot,\phiA)
}_{0,0}
&\leq C \de^2.
\end{aligned}
\end{equation}
\end{enumerate}
Now, we are ready to obtain estimates for $M_P(u,\phiA(u))$  by using the series expansion \eqref{def:SerieP}.
%
First, we check that it is convergent.
Indeed, by \eqref{proof:DA} and \eqref{proof:DB}, for $u \in \DuOut$ and taking $\kappa$ big enough we have that
\begin{align*}
\vabs{\frac{\wtB(u,\phiA(u))}{\wtA(u)}}
&\leq 		
\normOutSmall{\wtB(\cdot,\phiA)}_{0,-\frac{4}{3}}
\normOutSmall{\wtA^{-1}}_{0,\frac{4}{3}} 
\leq 
\frac{C}{\kappa^2\de^2}
\normOutSmall{\wtB(\cdot,\phiA)}_{1,-2}
\leq \frac{C}{\kappa^2} \ll 1.
\end{align*}
Therefore, by \eqref{def:functionD2} and  \eqref{def:expressionMJOuter},
\begin{align}\label{proof:expressionMP}
\vabs{M_P(u,\phiA(u))} \leq 
\vabs{\frac1{\textstyle\sqrt{A[\de^4-1](\la_h(u))}} - \frac1{\sqrt{2+2\cos \la_h(u)}}  }
+
C \frac{|\wtB(u,\phiA(u))|}{|\wtA(u)|^{\frac32}}.
\end{align}
Then, to estimate $M_P$ and its derivatives, it only remains to analyze the $u$-derivative of  its first term.
Indeed, by the definition of $A[\de^4-1]$ in \eqref{def:DA}.
\begin{align}\label{proof:DApotential}
\normOut{\partial_u \paren{\frac1{\textstyle\sqrt{A[\de^4-1](\la_h(u))}} - \frac1{\sqrt{2+2\cos \la_h(u)}} } }_{0,\frac43} \leq C \de^4.
\end{align}
Therefore, applying estimates \eqref{proof:DA}, \eqref{proof:DB}, \eqref{proof:DBfirst}, 
\eqref{proof:DBsecond} and  \eqref{proof:DApotential}, to the derivatives of $M_P$ and using \eqref{proof:expressionMP}, we obtain the statement of the lemma.
\end{proof}

Analogously to Lemma \ref{lemma:boundsMPout}, we obtain estimates for the first and second derivatives of $M_S$ and $M_R$ (see \eqref{def:expressionMSOuter} and \eqref{def:expressionMROuter}).

\begin{lemma}\label{lemma:boundsMSout}
Fix $\varrho>0$, 
$\de>0$ small enough
and $\kappa>0$ big enough. Then, for 
$\phiA \in \ballOuter$  and $*=x,y$, we have  
\begin{equation*}
\begin{aligned}	
\normOut{\partial_u M_S(\cdot,\phiA)
}_{0,\frac{4}{3}}
&\leq C \de^2, & 
\normOut{\partial_w M_S(\cdot,\phiA)
}_{0,-\frac{1}{3}}
&\leq C \de^{2}, & 
\normOut{\partial_* M_S(\cdot,\phiA)
}_{0,0}
&\leq C \de, \\
\normOut{\partial_{u w} M_S(\cdot,\phiA)
}_{0,\frac{2}{3}}
&\leq C \de^{2},  &
\normOut{\partial_{u *} M_S(\cdot,\phiA)
}_{0,\frac{1}{3}}
&\leq C \de,  &
\normOut{\partial^2_{w} M_S(\cdot,\phiA)
}_{0,-\frac{2}{3}}
&\leq C \de^4, \\
\normOut{\partial_{w *} M_S(\cdot,\phiA)
}_{0,-\frac{1}{3}}
&\leq C \de^3,  &
\normOut{\partial^2_{*} M_S(\cdot,\phiA)
}_{0,0}
&\leq C \de^2,  &
\normOut{\partial_{xy} M_S(\cdot,\phiA)
}_{0,0}
&\leq C \de^2.
\end{aligned}
\end{equation*}
%
%
and
\begin{equation*}
\begin{aligned}	
\normOut{\partial_u M_R(\cdot,\phiA)
}_{1,1}
&\leq C \de^2,  & 
\normOut{\partial_w M_R(\cdot,\phiA)
}_{1,-\frac{2}{3}}
&\leq C \de^{2},  & 
\normOut{\partial_* M_R(\cdot,\phiA)
}_{0,0}
&\leq C \de, \\
\normOut{\partial_{u w} M_R(\cdot,\phiA)
}_{1,\frac{1}{3}}
&\leq C \de^{2},  &
\partial_{u *} M_R(\cdot,\phiA)
&\equiv 0,  &
\normOut{\partial^2_{w} M_R(\cdot,\phiA)
}_{0,-\frac{2}{3}}
&\leq C, \\
\partial_{w *} M_R(\cdot,\phiA)
&\equiv 0,  &
\partial^2_{*} M_R(\cdot,\phiA)
&\equiv 0,   &
\partial_{xy} M_R(\cdot,\phiA)
&\equiv 0.
\end{aligned}
\end{equation*}
\end{lemma}

\begin{proof}[End of the proof of Lemma~\ref{lemma:computationsRRROut}]

We start by estimating the first and second derivatives of $H_1^{\out}(u,\phiA(u);\de)$ in suitable norms.
Recall that by \eqref{proof:H1outExpressionOuter}, $H_1^{\out}=M_P+M_S+M_R$. 
Therefore, taking $\varphi \in \ballInfty \subset \XcalOutTotal$ and  applying  Lemmas \ref{lemma:boundsMPout} and \ref{lemma:boundsMSout}:

\begin{enumerate}
\item For $g^{\out}=\partial_w H_1^{\out}$ one has
\[
\begin{split}	
\normOut{g^{\out}(\cdot,\phiA)}_{1,-\frac23} \leq&
\normOut{\partial_w M_P(\cdot,\phiA)}_{1,-\frac23} 
+ C
\normOut{\partial_w M_S(\cdot,\phiA)}_{0,-\frac13} +
\normOut{\partial_w M_R(\cdot,\phiA)}_{1,-\frac23} \\ \leq& C \de^2
\end{split}
\]
and, in particular, for $\kappa$ big enough
\begin{equation}\label{def:Fitag}
\normOut{g^{\out}(\cdot,\phiA)}_{0,0} \leq C \kappa^{-2} \ll 1.	
\end{equation}
Analogously, 
$\normOut{\partial_w g^{\out}(\cdot,\phiA)}_{0,-\frac23}
\leq C$ and  
$\normOut{\partial_* g^{\out}(\cdot,\phiA)}_{0,\frac53}
\leq C \de^3$, 
 for $*=x,y$.

\item For $f_1^{\out}=-\partial_u H_1^{\out}$, one has that
\[
\normOut{f^{\out}_1(\cdot,\phiA)}_{1,1} 
\leq
\normOut{\partial_u M_P(\cdot,\phiA)}_{1,1}
+
C \normOut{\partial_u M_S(\cdot,\phiA)}_{0,\frac43}+
\normOut{\partial_u M_R(\cdot,\phiA)}_{1,1}
\leq
C \de^2,
\] 
$\normOut{\partial_w f_1^{\out}(\cdot,\phiA)}_{1,\frac13}
\leq
C \de^2$ and 
$\normOut{\partial_* f_1^{\out}(\cdot,\phiA)}_{0,\frac73}
\leq 
C \de,
\quad $
for $*=x,y$.

\item For $f_2^{\out} = i \partial_y H_1^{\out}$ and $f_3^{\out} = -i \partial_x H_1^{\out}$,
we can obtain the estimates
\begin{equation}\label{proof:estimatef2f3Out}
\begin{split}	
\normOut{f_2(\cdot,\phiA)}_{0,\frac43} 
\leq&
\normOut{\partial_y M_P(\cdot,\phiA)}_{0,\frac43} 
+ C
\normOut{\partial_y M_S(\cdot,\phiA)
+ 
\partial_y M_R(\cdot,\phiA)}_{0,0}
\leq 
C \de, \\
\normOut{f_3(\cdot,\phiA)}_{0,\frac43} 
\leq&
\normOut{\partial_x M_P(\cdot,\phiA)}_{0,\frac43} 
+ C
\normOut{\partial_x M_S(\cdot,\phiA)
+	
\partial_x M_R(\cdot,\phiA)}_{0,0}
\leq 
C \de.
\end{split}
\end{equation}
Analogously, we have that
$\normOutSmall{\partial_w f_j^{\out}(\cdot,\phiA)}_{0,\frac53}
\leq
C \de^3$ and 
$\normOutSmall{\partial_* f_j^{\out}(\cdot,\phiA)}_{0,2}
\leq 
C \de^2$,
for $j=2,3$ and $*=x,y$.

\end{enumerate}

Joining these estimates  and taking $\kappa$ big enough, we complete the proof of the lemma. 
%
\end{proof}

\begin{remark}\label{rmk:R} 
Note that that $\DOuterTilde \subset \DuOut$ and  $\YcalOuter \subset 
\XcalOut_{0,0}$ (see \eqref{def:Yout} and \eqref{def:XcalOut}). Then, the proof 
of Lemma~\ref{lemma:computationsRRRtransitionOuter} is a direct consequence of 
the estimates  for $g^{\out}$ and its derivatives in Item 1 above and the 
fact that, by \eqref{eq:systemEDOsOuter} and \eqref{eq:InvU},
\[
R[\uOut](v)= \partial_w H_1^{\out}
	\paren{v + \uOut(v), \zuOut(v+\uOut(v))}= g^{\out}\paren{v + \uOut(v), 
\zuOut(v+\uOut(v))}.
\]
\end{remark}

\section{Estimates for the difference}
\label{appendix:proofF-technical}

In this section we prove
Lemmas \ref{lemma:boundsBspl} and \ref{lemma:boundsBonsBeta}.

\subsection{Proof of Lemma \ref{lemma:boundsBspl}}
\label{subappendix:proofH-technicalFirst}

First, we prove the estimates for the operator $\Ups$ given in \eqref{def:operatorPPdifference}.
%
%
%
For $\sigma \in [0,1]$, we define
$z_{\sigma}=\sigma\zuOut + (1-\sigma) \zsOut$ 
with $z_{\sigma}=(w_{\sigma},x_{\sigma},y_{\sigma})^T$.
%
%
Then, by Theorem \ref{theorem:existence}, for $u \in \DBoomerang$, we have that
\begin{equation}\label{proof:ztauBounds}
	\vabs{w_{\sigma}(u)}\leq \frac{C\de^2}{\vabs{u^2+A^2}} +
	\frac{C\de^4}{\vabs{u^2+A^2}^{\frac83}}, \qquad
	\vabs{x_{\sigma}(u)},\vabs{y_{\sigma}(u)}\leq \frac{C\de^3}{\vabs{u^2+A^2}^{\frac43}}.
\end{equation}		
Recalling that $H^{\out} = w + \frac{xy}{\de^2} + H_1^{\out}$ (see \eqref{def:hamiltonianOuter}), one has
\begin{equation*}
	\begin{split}
		\vabs{\Ups_1(u)-1}& \leq
		\sup_{\sigma \in[0,1]} \vabs{\partial_w H_1^{\out}(u,z_{\sigma}(u))},\\
		\vabs{\Ups_2(u)}&\leq 
		\frac{\vabs{y_{\sigma}(u)}}{\de^2}
		+
		\sup_{\sigma \in[0,1]} \vabs{
			\partial_x H_1^{\out}(u,z_{\sigma}(u))},
		\\ 
		\vabs{\Ups_3(u)} 
		&\leq 
		\frac{\vabs{x_{\sigma}(u)}}{\de^2}
		+
		\sup_{\sigma \in[0,1]} \vabs{
			\partial_y H_1^{\out}(u,z_{\sigma}(u))}.
	\end{split}
\end{equation*}
Then, by \eqref{proof:ztauBounds} and applying 
 \eqref{def:Fitag} and \eqref{proof:estimatef2f3Out} in the proof of
Lemma \ref{lemma:computationsRRROut} we obtain the estimates for $\Ups_1, \Ups_2$ and $\Ups_3$.

We also need estimates for the matrix $\wt{\BB}^{\spl}$ given in \eqref{def:Bspl1}, which  satisfies 
\begin{align*}
|\wt{\BB}^{\spl}_{i,j}(u)|
\leq 
\sup_{\sigma \in[0,1]}
\vabs{\paren{D_z \RRR^{\out}
[\zOut_{\sigma}](u)}_{i,j}},
\end{align*}
for $z_{\sigma}= \sigma\zuOut + (1-\sigma) \zsOut$.
Then, by \eqref{proof:ztauBounds} and applying Lemma \ref{lemma:computationsRRROut}, for $u \in \DBoomerang$,
\begin{equation} \label{proof:boundsBsplTilde}	
\begin{aligned}
\vabs{\wt{\BB}^{\spl}_{2,1}(u)} &\leq 
\frac{C\de}{\vabs{u^2 + A^2}^{\frac{2}{3}}},
& 
\vabs{\wt{\BB}^{\spl}_{3,1}(u)} &\leq 
\frac{C\de}{\vabs{u^2 + A^2}^{\frac{2}{3}}}, 
\\
\vabs{\wt{\BB}^{\spl}_{2,2}(u)} &\leq 
\frac{C}{\vabs{u^2 + A^2}^{\frac13}} +
\frac{C \de^2}{\vabs{u^2 + A^2}^{2}}, 
&
\vabs{\wt{\BB}^{\spl}_{3,2}(u)} &\leq 
\frac{C\de^2}{\vabs{u^2 + A^2}^{2}}, 
\\
\vabs{\wt{\BB}^{\spl}_{2,3}(u)} &\leq 
\frac{C\de^2}{\vabs{u^2 + A^2}^{2}}, 
&
\vabs{\wt{\BB}^{\spl}_{3,3}(u)} &\leq 
\frac{C}{\vabs{u^2 + A^2}^{\frac13}} +
\frac{C \de^2}{\vabs{u^2 + A^2}^{2}}.
\end{aligned}
\end{equation}
Then, by \eqref{proof:boundsUpsilon} and taking $\kappa$ big enough,
\begin{align*}
	\vabs{{\BB}^{\spl}_{1,1}(u)} &\leq  
	\frac{\vabs{\Ups_2(u)}}{\vabs{\Ups_1(u)}} \vabs{\wt{\BB}^{\spl}_{2,1}(u)} 
	\leq 
	\frac{C \de^2}{\vabs{u^2 + A^2}^{2}}, \\
	\vabs{{\BB}^{\spl}_{1,2}(u)} &\leq 
	\vabs{\wt{\BB}^{\spl}_{2,3}(u)} +
	\frac{\vabs{\Ups_3(u)}}{\vabs{\Ups_1(u)}} \vabs{\wt{\BB}^{\spl}_{2,1}(u)} 
	\leq 
	\frac{C \de^2}{\vabs{u^2 + A^2}^{2}},	
\end{align*}
and analogous estimates hold for ${\BB}^{\spl}_{2,1}$ and ${\BB}^{\spl}_{2,2}$.

%
Finally, we compute estimates for $\bb_y(u)$  (see \eqref{def:mxmyalxaly}) and  $u \in \DBoomerang$. The estimates for  $\bb_x(u)$ can be computed analogously.
%
%
Let us consider the integration path $\rho_t = u_* + (u-u_*)t $, for $t\in[0,1]$. 
Then
\begin{align*}
\bb_y(u) = \exp \paren{ 
\int_0^{1}  \wt{\BB}^{\spl}_{2,2}\paren{\rho_t}
(u-u_*) dt}.
\end{align*} 
Using the bounds in \eqref{proof:boundsBsplTilde}, we have that
\begin{align*}
\vabs{\log \bb_y(u)} &\leq
C \vabs{u-u_*} \vabs{\int^{1}_0  
 \frac{1}{\vabs{\rho^2_t + A^2}^{\frac{1}{3}}} +
\frac{\de^2}{\vabs{\rho^2_t + A^2}^{2}} dt} 
\leq C,
%
\end{align*}
which implies $C^{-1} \leq \vabs{\bb_y(u)} \leq C$.

%

\subsection{Proof of Lemma \ref{lemma:boundsBonsBeta}}
\label{subappendix:proofH-technicalSecond}

We only give an expression for $\bb_y(u_+)$. The result for $\bb_x(u_-)$ is analogous. 
First, we analyze $\wt{\BB}^{\spl}_{3,3}$.

\begin{lemma}\label{lemma:proofConstantA}
For $\de>0$ small enough, $\kappa>0$ large enough and $u \in \DBoomerang$, the function $\wt{\BB}^{\spl}_{3,3}$  defined in \eqref{def:Bspl1} is of the form 
\begin{align*}
	\wt{\BB}_{3,3}^{\spl}(u) = 
	-\frac{4i}3 \La_h(u)
	+ \de^2 m(u;\de),
\end{align*}
for some function $m$ satisfying
\begin{align*}
	\vabs{m(u;\de)}
	\leq \frac{C}{\vabs{u^2+A^2}^{2}}.
\end{align*}
\end{lemma}	

\begin{proof}
Let us define $z_{\tau}= \tau\zuOut + (1-\tau) \zsOut$ and recall that, for $u \in \DBoomerang$,
\begin{align}\label{proof:B33tildeDef}
\wt{\BB}_{3,3}(u) = 
\int_0^1 \partial_y \RRR_3^{\out}[z_{\tau}](u) d\tau.
\end{align}
Then, by the expression of $\RRR_3^{\out}$ in \eqref{eq:expressionRRRout}, the estimates 
in the proof of Lemma \ref{lemma:computationsRRROut} (see Appendix \ref{subsubsection:proofComputationsRRROut}) and Theorem \ref{theorem:existence}, we have that
\[
\partial_y \RRR_3^{\out}[z_{\tau}](u) 
=
\frac{i}{\de^2} g^{\out}(u,z_{\tau}(u)) +
\de^2 \wt{m}(u;\de),
\]
where
$
\vabs{\wt{m}(u;\de)}
\leq \frac{C}{\vabs{u^2+A^2}^{2}}.
$
In the following, to simplify notation, we denote by $\wt{m}(u;\de)$ any function satisfying the previous estimate.
Since $g^{\out}=\partial_w H_1^{\out}$, by \eqref{proof:H1outExpressionOuter} one has
\[
g^{\out}(u,z_{\tau}(u)) = \partial_w M_P(u,z_{\tau}(u);\de)+
\partial_w M_S(u,z_{\tau}(u);\de)+
\partial_w M_R(u,z_{\tau}(u);\de),
\]
with $M_P$, $M_S$ and $M_R$ as given in \eqref{def:expressionMJOuter},
\eqref{def:expressionMSOuter} and
\eqref{def:expressionMROuter}, respectively.
Then, taking into account that $F_{\pend}(s)=2z^3+\OO(z^4)$ (see \eqref{def:Fpend}) and following the proofs of Lemmas \ref{lemma:boundsMPout} and \ref{lemma:boundsMSout}, it is a tedious but an easy computation to see that,
\begin{equation*}
\begin{split}	
	g^{\out}(u,z_{\tau}(u)) =& \,
	\partial_w M_P(u,0,0,0;\de) + 
	\partial_w M_S(u,0,0,0;\de)  \\
	&- \frac{w_{\tau}(u)}{3 \La_h^2(u)}
	-\frac{\de^2 \LtresLa(\de)}{\La_h(u)}
	- 2 \de^2\La_h(u) 
	+
	\de^4\wt{m}(u;\de),
\end{split}
\end{equation*}
and, by \eqref{proof:B33tildeDef},
\begin{equation}\label{proof:B33tildeDefB}
\begin{split}
\wt{\BB}_{3,3}(u) =& \, \frac{i}{\de^2} \boxClaus{\partial_w M_P(u,0,0,0;\de) + 
\partial_w M_S(u,0,0,0;\de) } \\
&- i\frac{\wuOut(u)+\wsOut(u)}{6 \de^2\La_h^2(u)}
- i\frac{\LtresLa(\de)}{\La_h(u)}
- 2i \La_h(u) 
+
\de^2\wt{m}(u;\de).
\end{split}
\end{equation}

Next, we study the terms $\wusOut(u)$.
Since $H^{\out}=w + \frac{xy}{\de^2} + M_P+M_S+M_R$ (see
\eqref{def:hamiltonianOuter} and \eqref{proof:H1outExpressionOuter}), one can see that
\[
H^{\out}(u,\zuOut(u);\de) 
= 
H^{\out}(u,\zsOut(u);\de) 
= 
\lim_{\Re u \to \pm \infty} H^{\out}(u,0,0,0;\de)
=
\de^4 K(\de),
\]
with $\vabs{K(\de)}\leq C$, for $\de$ small enough.
Then, by Theorem \ref{theorem:existence}, for $\diamond=\unstable,\stable$,
\begin{align*}
\vabs{\wdOut(u) + M_P(u,\zdOut(u);\de)+
M_S(u,\zdOut(u);\de)+ M_R(u,\zdOut(u);\de)}
\leq 
\frac{C\de^4}{\vabs{u^2+A^2}^{\frac83}}.
\end{align*}
Again, following the proofs of Lemmas \ref{lemma:boundsMPout} and \ref{lemma:boundsMSout}, one obtains
\begin{align*}
\vabs{\wdOut(u) + M_P(u,0,0,0;\de) 
+ M_S(u,0,0,0;\de)
+\de^2\La_h(u)(3\LtresLa+2\La_h^2(u))}
\leq \frac{C\de^4}{\vabs{u^2+A^2}^{\frac83}},
\end{align*}
and, by \eqref{proof:B33tildeDefB},
\begin{equation*}
\begin{split}
\wt{\BB}_{3,3}(u) 
=& \, 
-  \frac{4i}3 \La_h(u)
+
\frac{i}{\de^2} \boxClaus{\partial_w M_P(u,0,0,0;\de) 
+ \frac{M_P(u,0,0,0;\de)}{3\La_h^2(u)}
 }
\\
&+ \frac{i}{\de^2} \boxClaus{\partial_w M_S(u,0,0,0;\de) 
+ \frac{M_S(u,0,0,0;\de)}{3\La_h^2(u)} }
+\de^2\wt{m}(u;\de).
\end{split}
\end{equation*}
Therefore, it only remains to check that
\begin{align*}
\vabs{\partial_w M_{P,S}(u,0,0,0;\de) 
+ \frac{M_{P,S}(u,0,0,0;\de)}{3\La_h^2(u)}}
\leq \frac{C\de^4}{\vabs{u^2+A^2}^2}.
\end{align*}

Indeed, by \eqref{def:SerieP} and the definition \eqref{def:expressionMJOuter} of $M_P$, one has 
\begin{align*}
M_P(u,w,0,0;\de) = \MM_P\paren{u,\de^2\La_h(u)-\frac{\de^2 w}{3\La_h(u)}+\de^4\LtresLa(\de)},	
\end{align*}
where $\MM_P(u,\La)$ is an analytic function for $u \in \DBoomerang$ and $\vabs{\La} \ll 1$. 
Moreover, following the proof of Lemma  \ref{lemma:boundsMPout}, there exist $a_0$ and $a_1$ such that 
\begin{align*}
\vabs{\MM_P(u,\La)-a_0(u;\de)- a_1(u;\de) \La} \leq 
\frac{C\La^2}{\vabs{u^2+A^2}^2},
\end{align*}
with
\[
\vabs{a_0(u;\de)} \leq \frac{C\de^4}{\vabs{u^2+A^2}^{\frac23}},
\qquad
\vabs{a_1(u;\de)} \leq \frac{C}{\vabs{u^2+A^2}^{\frac23}}.
\]
Therefore,
\begin{align*}
\vabs{\partial_w M_{P}(u,0,0,0;\de) 
+ \frac{M_{P}(u,0,0,0;\de)}{3\La_h^2(u)}}
\leq& \,
\frac{\vabs{a_0(u)}}{3\La_h^2(u)}
+ \frac{\de^4\LtresLa(\de)\vabs{a_1(u)}}{3\La_h^2(u)}
+ \frac{C\de^4}{\vabs{u^2+A^2}^{2}} 
\\
\leq& \,
\frac{C\de^4}{\vabs{u^2+A^2}^{2}} .
\end{align*}
An analogous estimate holds for $M_S$.
\end{proof}

\begin{proof}[End of the proof of Lemma \ref{lemma:boundsBonsBeta}]
By Lemma \ref{lemma:proofConstantA} and recalling that  $u_+=iA-\kappa \de^2$,
\begin{equation}
	\label{proof:constantDifferenceZ}
\begin{split}	
\log B_y(u_+) 
=&
\int_{u_*}^{u_+} \wt{\BB}_{3,3}^{\spl}(u) du
= 
-\frac{4i}3  \int_{u^*}^{i A} \La_h(u) du 
\\ &+
\frac{4i}3 \int^{i A}_{u_+} \La_h(u) du 
+
\de^2
\int_{u^*}^{u_+} {m}(u;\de).
\end{split}
\end{equation}
Then, by Theorem \ref{theorem:singularities} and taking into account that $\kappa=\kappa_* \vabs{\log \de}$ (see Lemma \ref{lemma:IconstantsDiff}), we obtain
\begin{align*}
\vabs{\log B_y(u_+)  + \frac{4 i}3
\int_{u^*}^{i A} \La_h(u) du}
\leq 
\frac{C}{\kappa} 
+
C \kappa^{\frac23}\de^{\frac43}
+
\frac{C\de^2}{\vabs{u_* - iA}}
\leq 
\frac{C}{\vabs{\log \de}}.
\end{align*}
Finally, recalling that $\dot{\la}_h=-3\La_h$, applying the change of coordinates $\la=\la_h(u)$
and using that $\la_h(iA) = \pi$, we have that
\begin{align*}
\frac{4i}3 \int_{u^*}^{i A} \La_h(u) du 
=
-\frac{4i}9 \int_{\la_h(u_*)}^{\pi} d\la
=
-\frac{4i}9 \paren{\pi-\la_h(u_*)}.
\end{align*}
Joining the last statements with \eqref{proof:constantDifferenceZ}, we obtain the statement of the lemma.
\end{proof}

\section*{Acknowledgments}
\addcontentsline{toc}{section}{Acknowledgments}
I. Baldom\'a has been partly supported by the Spanish MINECO--FEDER Grant PGC2018 -- 098676 -- B -- 100 (AEI/FEDER/UE) and the Catalan grant 2017SGR1049. 

M. Giralt and M. Guardia have received funding from the European Research Council (ERC) under the European Union's Horizon 2020 research and innovation programme (grant agreement No 757802). 

M. Guardia is also supported by the Catalan Institution for Research and Advanced Studies via an ICREA Academia Prize 2019.


\addcontentsline{toc}{section}{Bibliography}

\begingroup
\footnotesize
\setlength\bibitemsep{2pt}
\setlength\biblabelsep{4pt}
\printbibliography
\endgroup


\end{document}